\documentclass{article}
\usepackage{amsmath}
\usepackage{amsfonts}
\usepackage{graphicx, subfigure}
\usepackage[pdftex]{hyperref}
\usepackage[dutch,english]{babel}
\usepackage{amsmath,amssymb}
\usepackage{amsthm}
\usepackage[usenames,dvipsnames]{color}
\usepackage{tikz}
\usepackage{accents}
\usepackage{xcolor}
\usepackage{hyperref}
\usepackage{mathtools}
\usepackage{float}
\usepackage{multicol}
\usepackage{comment}
\usepackage[T1]{fontenc}
\usepackage[font={small,it}]{caption}

\DeclarePairedDelimiter{\ceil}{\lceil}{\rceil}
\usepackage{tikz}
\usetikzlibrary{shadows, arrows}

\newcommand{\la}{\lambda}

\DeclareMathOperator{\spn}{span}

\DeclareMathOperator{\Id}{Id}

\DeclareMathOperator{\End}{End}
\DeclareMathOperator{\mat}{Mat}
\DeclareMathOperator{\tr}{Tr}
\DeclareMathOperator{\nil}{Nil}
\DeclareMathOperator{\cen}{Cen}
\DeclareMathOperator{\im}{Im}

\DeclareMathOperator{\Exp}{exp}
\DeclareMathOperator{\Gl}{Gl}
\DeclareMathOperator{\Det}{Det}
\DeclareMathOperator{\ad}{ad}
\DeclareMathOperator{\degt}{deg_2}

\theoremstyle{plain}
\newtheorem{thr}{Theorem}[section]

\newtheorem{lem}[thr]{Lemma}
\newtheorem{cor}[thr]{Corollary}
\newtheorem{prop}[thr]{Proposition}
\theoremstyle{definition}
\newtheorem{defi}[thr]{Definition}
\newtheorem{ex}[thr]{Example}
\theoremstyle{remark}
\newtheorem{remk}{Remark}
\theoremstyle{remark}

\newcommand{\field}[1]{\mathbb{#1}}
\newcommand{\R}{\field{R}}
\newcommand{\N}{\field{N}}
\newcommand{\C}{\field{C}}

\newcommand{\K}{\field{K}}

\definecolor{wred}{rgb}{0.7,0.18,0.12}
\definecolor{wgreen}{rgb}{0.1,0.53,0.37}

\numberwithin{equation}{section}

\def\barray{\begin{eqnarray*}}             \def\earray{\end{eqnarray*}}
\def\beq{\begin{equation}} \def\eeq{\end{equation}}

\title{Transversality in Dynamical Systems with Generalized Symmetry}
\author{Eddie Nijholt\footnote{\mbox{Department of Mathematics, VU University Amsterdam, The Netherlands, \href{mailto:eddie.nijholt@gmail.com}{eddie.nijholt@gmail.com} }  }, Bob Rink\footnote{ \mbox{Department of Mathematics, VU University Amsterdam, The Netherlands, \href{mailto:b.w.rink@vu.nl}{b.w.rink@vu.nl} }} }
\date{\today}

\excludecomment{percent}

\begin{document}

\maketitle

\begin{abstract}
\noindent We prove that a generic $k$-parameter bifurcation of a dynamical system with a monoid symmetry occurs along a generalized kernel or center subspace of a particular type. More precisely, any (complementable) subre\-presentation $U$ is given a number $K_U$ and a number $C_U$. A $k$-parameter bifurcation can generically only occur along a generalized kernel isomorphic to $U$ if $k \geq K_U$. It can generically only occur along a center subspace isomorphic to $U$ if $k \geq C_U$. The numbers $K_U$ and $C_U$ depend only on the decomposition of $U$ into indecomposable subrepresentations. In particular, we prove that a generic one-parameter steady-state bifurcation occurs along one absolutely indecomposable subrepresentation. Likewise, it follows that a generic one-parameter Hopf bifurcation occurs along one indecomposable subrepresentation of complex or quaternionic type, or along two isomorphic absolutely indecomposable subrepresentations. In order to prove these results, we show that the set of endomorphisms with generalized kernel (or center subspace) isomorphic to $U$ is the disjoint union of a finite set of conjugacy invariant submanifolds of codimension $K_U$ and higher (or $C_U$ and higher). The results in this article hold for any monoid, including non-compact groups. 
\end{abstract}

%\noindent Demon Cat: Greetings, Frank the Human Boy. \\
%\noindent Finn: How did you almost know my name?! \\
%\noindent Demon Cat: I have approximate knowledge of many things.

\section{Introduction} \label{Geometric Reduction}
Symmetries play an important role in the study of dynamical systems. Equi\-variant dynamics, the mathematical discipline concerned with this interplay, has correspondingly gained a lot of attention and has developed into a well established field of research. It should be noted however, that many results from this field require the symmetries in question to form a group, often a finite one or a compact topological group. See for example \cite{sym2}, \cite{sym3}, \cite{sym1} or \cite{sym4} for more on equivariant dynamics.\\
\noindent Recent developments in the study of network dynamical systems have called for a generalization of this. More precisely, it can be shown that under mild conditions, a dynamical system with a network structure can be seen as the restriction of an equivariant system to some invariant subspace. This equivariant system is referred to as the fundamental network of the original network, see Figure \ref{Fig1}. Often the symmetries appearing in this latter network system do not form a group, but rather a more relaxed structure such as a semigroup, monoid (i.e. a semigroup with an identity) or category. See \cite{cen}, \cite{proj}, \cite{fibr}, \cite{RinkSanders2}, \cite{RinkSanders3} and \cite{CCN} for more on this formalism. Other authors have likewise linked network structures to more general algebraic concepts, such as the groupoid formalism by Golubitsky and Stewart (\cite{golstew}), or the categorical approach by Lerman and Deville (\cite{deville}).\\

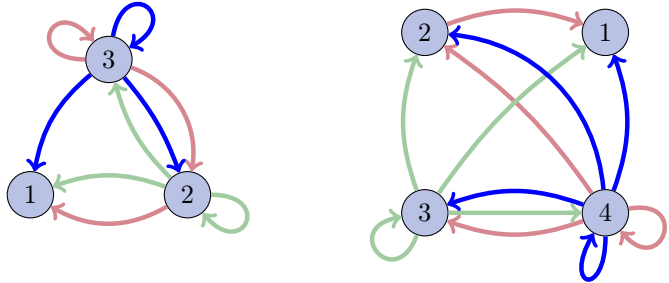
\begin{figure}[H]
\begin{center}
\begin{tikzpicture}[scale=1.2]
\tikzstyle{every node} = [ draw, circle, fill={rgb:red,1;green,2;blue,6;white,20}, anchor=center]
\node (1) at ( -0.87, -0.5) {1};
\node (2) at ( 0.87, -0.5) {2};
\node (3) at (0,1) {3};
\path[ ->, line width=1.7pt,  color={rgb:red,12; blue,3; green, 2 ;white,15}]
	(2) edge[bend left=30] (1)
	(3) edge[bend left=40] (2)
	(3) edge[loop, in=130, out=180,  distance=7mm] (3);
\path[ ->, ,line width=1.7pt, color={rgb:red,1; blue,1; green, 6 ;black,4;white,18}]
	(2) edge[bend right=20] (1)
	(2) edge[loop, in=-50, out=0,  distance=7mm] (2)
	(2) edge[bend left= 15] (3);
\path[ ->, ,line width=1.7pt,  color={blue}]
	(3) edge[bend right=20] (1)
	(3) edge[bend left=10 ] (2)
	(3) edge[loop, in=30, out=80, min distance = 7mm ] (3);
\node (2') at ( 3.5, 1.3) {2};
\node (1') at ( 5.5, 1.3) {1};
\node (4') at (5.5,-0.7) {4};
\node (3') at (3.5,-0.7) {3};
\path[ ->, line width=1.7pt,  color={rgb:red,12; blue,3; green, 2 ;white,15}]
	(2') edge[bend left=20] (1')
	(4') edge[bend left=-10] (2')
	(4') edge[bend left=20] (3')
	(4') edge[loop, in=-50, out=10,  distance=7mm] (4');
\path[ ->, ,line width=1.7pt, color={rgb:red,1; blue,1; green, 6 ;black,4;white,18}]
	(3') edge[bend right=-10] (1')
	(3') edge[bend left=20 ] (2')
	(3') edge[bend left=0 ] (4')
	(3') edge[loop, in=190, out=250, min distance = 7mm ] (3');
\path[ ->, line width=1.7pt, color={blue}]
	(4') edge[bend left=-20] (1')
	(4') edge[bend left=-40] (2')
	(4') edge[bend left=-20] (3')
	(4') edge[loop, in=-120, out=-90,  distance=7mm] (4');
\end{tikzpicture}
\caption{A network (left) with its fundamental network (right). For convenience, we have left out additional self-loops representing interior dynamics for all the cells in both the networks. The network vector fields corresponding to the fundamental network are exactly the vector fields with a linear monoid-symmetry, for which the fundamental network is in fact the Cayley graph. The network vector fields for the graph on the left correspond to those on the right restricted to a linear (synchrony) subspace. This is reflected in the graphs by the fact that identifying cells 2 and 3 in the right network yields the left one, up to renumbering.}
\label{Fig1}
\end{center}
\end{figure}
\noindent This article deals with the question of generic bifurcations in equivariant systems for such generalized symmetries. Let $f(x,\la)$ be a family of vector fields, indexed by some parameter $\la$. It is known that for a bifurcation to occur in the differential equation $\dot{x} = f(x,\la)$, often a certain condition on the spectrum of the linearization $D_xf$ has to hold. For example, the implicit function theorem excludes (non-trivial) steady-state bifurcations unless the matrix $D_xf(x_0,\la_0)$ at the bifurcation point $(x_0,\la_0)$ has a non-trivial kernel. Furthermore, Hopf bifurcations are associated with a pair of complex-conjugate eigenvalues of $D_xf(x_0, \la)$ passing through the imaginary axis as $\la$ varies. However, when considering generic one-parameter steady-state bifurcations, one does not expect the kernel of $D_xf(x_0,\la_0)$ to be two-dimensional or even bigger either, as a small perturbation of the family of vector fields $f(x,\la)$ would generically perturb the kernel to a lower dimensional one. The situation is more complicated for equivariant systems, as the symmetry might exclude certain spaces to appear as kernels or center subspaces, thereby making other spaces more 'likely'. \\
\noindent As it turns out, the correct generalization of a one-dimensional kernel in the case of a compact Lie-group symmetry, is that of an \textit{irreducible} subrepresentation. As a one-dimensional space can be characterized by the property that it does not contain any non-trivial subspaces, so too can an irreducible subre\-presentation be defined by the property that it does not contain any non-trivial invariant subspaces. One can further generalize this concept to that of an \textit{absolutely} irreducible subrepresentation, by imposing the condition that the only symmetry-respecting  endomorphisms of the space are multiples of the identity.  It can then be proven that for compact groups a steady-state bifurcation occurs generically along one absolutely irreducible subrepresentation of the symmetry. See Proposition 3.2 in chapter XIII of \cite{sym4}. \\
When the symmetries only form a monoid, one can define an \textit{indecomposable} subrepresentation as an invariant space that cannot be written as the direct sum of two (non-trivial) invariant subspaces. Such an indecomposable subrepresentation is called \textit{absolutely} indecomposable if the only symmetry preserving endomorphisms are multiples of the identity, up to nilpotent maps. It was shown in \cite{RinkSanders3} (Theorem 6.2) that under a certain technical condition on the representation of the symmetry-monoid, a steady-state bifurcation occurs generically along one absolutely indecomposable subrepresentation. \\
In this article we prove the following more general result about generic $k$-parameter bifurcations in monoid-symmetric dynamical systems. Note that any group is a particular example of a monoid. One often calls an absolutely indecomposable representation an indecomposable representation of real type, after the algebra of its endomorphisms. Likewise, there are the notions of complex type and of quaternionic type. Any (finite-dimensional) indecomposable representation falls in either of these three classes. Our main result is the following.

\begin{thr}\label{mmain}
Let $W$ be a finite-dimensional representation space of a monoid $\Sigma$. Let $U \subset W$ be an invariant subspace satisfying 
\begin{equation}
U \cong \bigoplus^{r_1} W_{1}^R \dots \bigoplus^{r_u} W_{u}^R \bigoplus^{c_1} W_{1}^C \dots \bigoplus^{c_v} W_{v}^C \bigoplus^{h_1} W_{1}^H \dots \bigoplus^{h_w} W_{w}^H\, .
\end{equation}
Here, the $W_{i}^R$, $W_{i}^C$ and $W_{i}^H$ are non-isomorphic indecomposable representations of respectively real type, complex type and quaternionic type, and the numbers  $r_i$,  $c_i$ and $h_i$ denote multiplicities in $U$. Suppose furthermore that there exists an invariant subspace $U'$ such that $W = U \oplus U'$. Then generically a $k$-parameter family of endomorphisms of $W$ has an element with generalized kernel isomorphic to $U$ only when $k$ is bigger or equal to 
\[K_U := r_1 + \dots + r_u + 2c_1 + \dots  + 2c_v + 4h_1 + \dots + 4h_w\, .\] 
\noindent Likewise, generically a $k$-parameter family of endomorphisms of $W$ has an element with center subspace isomorphic to $U$ only when $k$ is bigger or equal to 
\[ C_U:= \ceil{r_1/2} + \dots + \ceil{r_u/2} + c_1 + \dots  + c_v + h_1 + \dots + h_w\, .\] 
Here, $\ceil{x}$ means $x$ rounded up to the nearest integer. (It is not hard to see that if such a $U'$ does not exist, then $U$ will not appear as a generalized kernel or as a center subspace.) \\ \noindent More precisely, let $\End(W)$ denote the set of all $\Sigma$-equivariant linear maps from $W$ to itself. We furthermore denote by $\nil(U)$ the set of elements in $\End(W)$ with generalized kernel isomorphic to $U$ and by $\cen(U)$ the set of elements in $\End(W)$ with center subspace isomorphic to $U$. Then the set  $\nil(U)$ is the union of a finite set of conjugacy invariant submanifolds of codimension $K_U$ and higher. Likewise, $\cen(U)$ is the union of a finite set of conjugacy invariant submanifolds of codimension $C_U$ and higher.\\
\noindent Consequently, when $\Omega \subset \R^k$ is some open parameter-space, the set 
\[\{f \in C^{\infty}(\Omega, \End(W)) \mid f(\Omega) \cap \nil(U) = \emptyset \}\]
\noindent is dense in the weak and strong topologies on $C^{\infty}(\Omega, \End(W))$ whenever \\ 
\noindent $k < K_U$. Likewise, the set 
\[\{f \in C^{\infty}(\Omega, \End(W)) \mid f(\Omega) \cap \cen(U) = \emptyset \}\]
\noindent is dense in the weak and strong topologies on $C^{\infty}(\Omega, \End(W))$ whenever $k < C_U$. Moreover, the sets
\[\{f \in C^{\infty}(\Omega, \End(W)) \mid f(\Omega) \cap \nil(U) \not= \emptyset \}\]
and
\[\{f \in C^{\infty}(\Omega, \End(W)) \mid f(\Omega) \cap \cen(U) \not= \emptyset \}\]
contain a non-empty open set in $C^{\infty}(\Omega, \End(W))$ for $k \geq K_U$ and $k \geq C_U$, respectively.
\end{thr}
\noindent This result will be proven in several steps. In Section \ref{Geometric Reduction} we will show how (the technical formulation of) the result follows, provided it holds in the special case when $W = U$. In Section  \ref{Algebra: Reduction} we then reduce the situation to one that only involves three families of real algebras, essentially stripping away the symmetry-monoid itself.  As some of the results of this section have value on their own, we have decided to split this reduction in three separate steps, and to furthermore elaborate on  quite some of the intermediate findings. Next, Section \ref{proof of main results} is dedicated to proving Theorem \ref{mmain} in the forms of Theorem \ref{main} and Remark \ref{mainremk}. For the proof of Theorem \ref{main} we need some technical results. These are proven in Sections \ref{Intermezzo; Some algebraic geometry} and \ref{Geometry; Counting dimensions}. More precisely,  Section \ref{Intermezzo; Some algebraic geometry} serves as a short excursion into algebraic geometry needed to prove a technical result. Section \ref{Geometry; Counting dimensions} then uses this result to count the dimensions of the set of elements with a vanishing or purely imaginary spectrum in the three reduced algebras that we obtain in Section \ref{Algebra: Reduction}. To start off, Section \ref{Pre} gives an overview of the results from representation theory that we will be using. \\
\noindent Despite all sections working towards the single goal of proving Theorem \ref{mmain}, it is best to think of this article as consisting of two separate parts. The first part consists of Sections \ref{Pre} till \ref{proof of main results} and comprises the main discussion. The second part is Sections \ref{Intermezzo; Some algebraic geometry} and \ref{Geometry; Counting dimensions} and could be thought of as an appendix where we treat some of the harder geometrical results (where there is no given symmetry anymore). Many of the results in the second part are known to experts, but hard or even impossible to find in the literature. Furthermore, an in debt study of these topics was needed to generalize some of the results to for example matrices with  quaternion entries. What is more, we believe some of the results in Section \ref{Geometry; Counting dimensions} constitute meaningful results in geometry.  Whereas little to no knowledge of the geometrical techniques used in these last two sections is required, a reader interested in only the main result of this article can simply skip this second part.

\section{Preliminaries} \label{Pre}
In this section we present some basic results from the representation theory of monoids. We furthermore fix the notation that will be used throughout this article. It should be noted that we put no restrictions on the monoid $\Sigma$. In particular, it may be finite or infinite, and it could correspond to a (non-compact) topological group or Lie-group. It is, however, essential that the representation space $V$ is finite dimensional. Proofs and additional remarks can be found in \cite{proj} and \cite{RinkSanders3}.

\begin{defi}
A \textit{monoid} is a triple $(\Sigma, e, \circ)$, where $\Sigma$ is a set, $e$ is an element of $\Sigma$ (called the \textit{unit}) and $\circ$ is a map from $\Sigma \times \Sigma$ to $\Sigma$ (notation: $x \circ y \in \Sigma$ for $x,y \in \Sigma$). This triple has to satisfy the following properties:
\begin{enumerate}
\item $(x \circ y) \circ z = x \circ (y \circ z)$ for all $x,y,z \in \Sigma$
\item $e \circ x = x \circ e = x$ for all $x \in \Sigma$.
\end{enumerate}
Note that a group is a particular instance of a monoid. If one drops the existence of a unit in the definition of a monoid (and therefore the second condition), one obtains a \textit{semigroup}. Note that any semigroup can be made into a monoid by artificially adding a unit as an extra element to $\Sigma$. The multiplication $\circ$ is then expanded to $\Sigma \cup \{e\}$ by imposing the second condition in the definition of a monoid. \\

\noindent If $V$ is a finite dimensional vector space over a field $\mathbb{K}$, then we denote by $\mat_{\mathbb{K}}(V)$ the space of $\mathbb{K}$-linear maps from $V$ to itself. We will often drop the subscript $\mathbb{K}$ and simply write $\mat(V)$ when the underlying field is clear. Furthermore, if we have $V = \R^n$ and $\mathbb{K} = \R$ then we will write $\mat(\R,n) := \mat_{\R}(\R^n)$. Likewise, we write $\mat(\C,n) := \mat_{\C}(\C^n)$ for the space of complex matrices. Using this notation, we say that a \textit{representation} of the monoid $\Sigma$ in the vector space $V$ over $\mathbb{K}$ is a map $\phi$ from $\Sigma$ to $\mat_{\mathbb{K}}(V)$ satisfying:
\begin{enumerate}
\item $\phi(x \circ y) = \phi(x) \circ \phi(y)$ for all $x,y \in \Sigma$, and with multiplication between elements in $\mat_{\mathbb{K}}(V)$ understood as composition of operators.
\item $\phi(e) = \Id_V$.
\end{enumerate}
A representation of a semigroup can be defined analogously, by dropping the second condition. Given such a representation $\phi$ of a semigroup $\Sigma$, one obtains a representation of the induced monoid $\Sigma \cup \{e\}$ by setting $\phi(e) := \Id_V$. As we will mostly be interested in those (linear) operators that commute with all elements of the form $\phi(x)$, results will often not change if one passes from a semigroup to its induced monoid. Likewise, if one has a subset $S$ of a monoid $T$, then one may often pass to the smallest monoid $\Sigma \subset T$ containing $S$. In this article a representation will always be over the field $\R$. We also note in passing that the maps $\phi(x)$ need by no means be invertible.\\

\noindent Given a monoid $\Sigma$ and a representation $(V,\phi)$, a linear subspace $W \subset V$ is said to be \textit{invariant} if it holds that $\phi(x)$ maps $W$ into itself for all $x \in \Sigma$. In that case, $W$ becomes a representation space itself via the maps $\phi(x)|_W$. We say that an invariant space $W \subset V$ is \textit{complementable} if there exists an invariant space $U \subset V$ such that $V = W \oplus U$. It is in general not true that every invariant subspace is complementable. If we have two representations $(V, \phi)$ and $(V,' \phi')$, then a \textit{morphism} between these two representations is a linear map $f: V \rightarrow V'$ satisfying $f \circ \phi(x) = \phi'(x) \circ f$ for all $x \in \Sigma$. If $f$ is furthermore invertible, then we call it an \textit{isomorphism}.  Note that in that case, it follows that $f^{-1} \circ \phi'(x) = \phi(x) \circ f^{-1}$, so that $f^{-1}$ is also a morphism. We call two representations \textit{isomorphic} if there exists an isomorphism between them. The space of morphisms between $(V, \phi)$ and itself wil be denoted by $\End(V)$ (Note that we suppress $\phi$ here, as we will often do in $(V, \phi)$ once $\phi$ is fixed). Some examples of morphisms include the inclusion of $(W,\phi(\bullet)|_W)$ in $(V, \phi)$ when $W \subset V$ is invariant, and the projection of $V = W \oplus U$ onto $W$ when $W$ and $U$ are (complementable) invariant spaces. \\

\noindent An element of $\End(V)$ can give rise to a number of invariant spaces. For example, the image, kernel and more generally the span of the eigenvectors of a real eigenvalue or pair of complex-conjugate eigenvalues is always an invariant space. Furthermore, the span of the \textit{generalized} eigenvectors of an eigenvalue or pair of complex-conjugate eigenvalues is invariant and is in fact complementable (by the span of the generalized eigenvectors of other eigenvalues). In particular, we define the \textit{generalized kernel} of an endomorphism to be the span of the generalized eigenvectors corresponding to the eigenvalue $0$. Likewise, the \textit{center subspace} of an endomorphism is the span of the generalized eigenvectors corresponding to all eigenvalues with vanishing real part. By the foregoing, the generalized kernel and center subspace of an element of $\End(V)$ are examples of complementable invariant subspaces.\\

\noindent Lastly, we say that a nonzero representation $V$ is \textit{indecomposable} if it cannot be written as $V = W \oplus U$ for non-trivial invariant spaces $W$ and $U$. Note that an indecomposable representation may still have non-trivial invariant subspaces. 
\end{defi}
\noindent The following result states that indecomposable representations can be seen as the building blocks of other representations.

\begin{thr}[The Krull-Schmidt theorem]
Any (finite-dimensional) representation space $W$ is isomorphic to the direct sum of a finite number of indecomposable representations. I.e. we have 
\begin{equation}
W \cong W_{1} \oplus W_2 \oplus \dots W_k\, , 
\end{equation}
for certain indecomposable representations $W_1$ till $W_k$. This decomposition is unique in the following sense. If it also holds that 
\begin{equation}
W \cong W'_{1}  \oplus \dots W'_l\, , 
\end{equation}
for certain indecomposable representations $W'_1$ till $W'_l$, then $k = l$ and we have that $W_i \cong W'_i$ for all $i$, after renumbering. 
\end{thr}
\noindent The space $\End(W)$ has some special properties in the case that $W$ is indecomposable.

\begin{lem}[The Fitting Lemma] \label{fitting}
If $W$ is an indecomposable representation of a monoid $\Sigma$, then every element $A$ of $\End(W)$ is either invertible, or nilpotent (i.e satisfies $A^n = 0$ for some $n \in \N$). Moreover, the set of nilpotent elements of $\End(W)$ forms an ideal. That is, if we have $A, N, N' \in \End(W)$ with $N$ and $N'$ nilpotent and $\la \in \R$, then $AN$, $NA$, $N+N'$ and $\la N$ are all nilpotent as well. 
\end{lem}

\begin{defi}
If we write $\nil(W)$ for the ideal of nilpotent endomorphisms of an indecomposable representation $W$, then it follows that the space \\
$\End(W)/\nil(W)$ is a real associative division algebra of finite dimension. By the Frobenius Theorem, it follows that $\End(W)/\nil(W)$ is isomorphic to either $\R$, $\C$ or $\mathbb{H}$. Depending on which, we say that $W$ is of \textit{real type}, \textit{complex type} or \textit{quaternionic type}. It can be shown that isomorphic indecomposable representations are of the same type. An indecomposable representation of real type is sometimes also referred to as an \textit{absolutely indecomposable} representation.
\end{defi}

\noindent We will also make use of the following lemma.

\begin{lem}
Let $W_1$ and $W_2$ be indecomposable representations of a monoid $\Sigma$, and let $f: W_1 \rightarrow W_2$ and $g: W_2 \rightarrow W_1$ be morphisms. If the morphism $g \circ f \in \End(W_1)$ is invertible, then $W_1$ and $W_2$ are isomorphic representations. Combining with Lemma \ref{fitting}, we see that if $W_1$ and $W_2$ are non-isomorphic, then $g \circ f$ is necessarily nilpotent.
\end{lem}

\section{Geometric Reduction} \label{Geometric Reduction}

\begin{defi}
Let $W$ be a finite dimensional representation space of the monoid $\Sigma$ and let $W = U \oplus U' $ be a decomposition of $W$ into invariant spaces. We denote by 

\begin{equation}
\nil(U) \subset \End(W)
\end{equation}
those elements of $\End(W)$ whose generalized kernel is isomorphic to $U$ as representations of $\Sigma$. In particular, $\nil(W)$  simply denotes the nilpotent elements of $\End(W)$. Likewise, we denote by

\begin{equation}
\cen(U) \subset \End(W)
\end{equation}
those elements of $\End(W)$ whose center subspace is isomorphic to $U$. 
\end{defi}
A general finite dimensional invariant space $W$ can be written as

\begin{equation}\label{decom1a}
W \cong \bigoplus^{r_1} W_{1}^R \dots \bigoplus^{r_u} W_{u}^R \bigoplus^{c_1} W_{1}^C \dots \bigoplus^{c_v} W_{v}^C \bigoplus^{h_1} W_{1}^H \dots \bigoplus^{h_w} W_{w}^H\, 
\end{equation}
where the $W_{i}^K$, $K \in \{R,C,H\}$ are non isomorphic indecomposable representations of real ($R$), complex ($C$) or quaternionic ($H$) type. If we are given a decomposition $W = U \oplus V$ then we may furthermore write

\begin{equation}\label{decom1b}
U \cong \bigoplus^{r_1'} W_{1}^R \dots \bigoplus^{r_u'} W_{u}^R \bigoplus^{c_1'} W_{1}^C \dots \bigoplus^{c_v'} W_{v}^C \bigoplus^{h_1'} W_{1}^H \dots \bigoplus^{h_w'} W_{w}^H\, 
\end{equation}
for some numbers $r_1' \leq r_1, \dots h_w' \leq h_w$. We will hold on to this notation for the rest of this chapter. The following result can be considered the core result of this paper.

\begin{thr}\label{main}
In $\End(W)$, the set $\nil(U)$ is the disjoint union of a finite set of embedded manifolds having codimension

\[K_U := r_1' + \dots + r_u' + 2c_1' + \dots  + 2c_v' + 4h_1' + \dots + 4h_w'\] 
or higher. Exactly one of these manifolds has codimension precisely equal to this number. Furthermore, these manifolds are conjugacy invariant. That is, if $M$ denotes any of these manifolds and if $A$ is an element of $M$ and $C \in \End(W)$ is invertible, then  $CAC^{-1}$ is an element of $M$ as well.\\
\noindent Likewise, the set $\cen(U)$ is the disjoint union of a finite set of conjugacy invariant embedded manifolds having codimension
\[ C_U := \ceil{r_1'/2} + \dots + \ceil{r_u'/2} + c_1' + \dots  + c_v' + h_1' + \dots + h_w'\] 
or higher. Here, $\ceil{x}$ denotes $x$ rounded up to the nearest integer. Exactly one of these manifolds has codimension precisely equal to $C_U$.

\end{thr}
In this section we will prove Theorem \ref{main}, under the assumption that it holds in the special case when $W = U$. In the next sections we will then prove Theorem \ref{main} for $U = W$. More precisely, the following theorem will be proven in the next sections.

\begin{thr}\label{codimsteady}
In $\End(W)$, the set $\nil(W)$ is the disjoint union of a finite set of conjugacy invariant embedded manifolds having codimension

\[K_W = r_1 + \dots + r_u + 2c_1 + \dots  + 2c_v + 4h_1 + \dots + 4h_w\] 
or higher. Exactly one of these manifolds has codimension precisely equal to this number. \\
Likewise, the set $\cen(W)$ is the disjoint union of a finite set of conjugacy invariant embedded manifolds having codimension
\[C_W = \ceil{r_1/2} + \dots + \ceil{r_u/2} + c_1 + \dots  + c_v + h_1 + \dots + h_w\] 
or higher. Exactly one of these manifolds has codimension precisely equal to this number.
\end{thr}

\noindent In order to prove Theorem \ref{main} from Theorem \ref{codimsteady}, we will need the following, technical lemmas. Some of these will also play a major role in Section \ref{Geometry; Counting dimensions}. Hence, they do not assume the result of Theorem \ref{codimsteady}. The first of these lemmas is well known (see for instance \cite{wim}), but included here for completeness. Furthermore, it demonstrates some techniques that will play an important role in Section \ref{Geometry; Counting dimensions}.

\begin{lem}\label{cominv}
Let $A \in \mat(\C,n)$ and $B \in \mat(\C,m)$ be square matrices and denote by $\mat(\C^n, \C^m)$ the space of complex $m \times n$ matrices. Define the linear map

\begin{equation}
\begin{split}
\mathcal{L}_{A,B}:  &\mat(\C^n, \C^m) \rightarrow  \mat(\C^n, \C^m)\\
&X \mapsto XA - BX \, .
\end{split}
\end{equation}
The eigenvalues of $\mathcal{L}_{A,B}$ are exactly given by $\la - \mu$ for $\la$ an eigenvalue of $A$ and $\mu$ an eigenvalue of $B$. In particular, this map is invertible if and only if $A$ and $B$ have no eigenvalues in common.
\end{lem}

\begin{proof}
Denote by $\{ e_i\}_{i=1}^m$ a basis such that $B$ is in upper triangular form. That is, we write 

\begin{equation}
Be_i = \mu_ie_i + \sum_{j<i}B_{j,i}e_j \, ,
\end{equation}
for $(\mu_1, \dots \mu_m)$ the set of eigenvalues of $B$. Likewise, we denote by $\{ f_i\}_{i=1}^n$ a basis such that $A^T$ is in upper triangular form:

\begin{equation}
A^Tf_i = \la_if_i + \sum_{j<i}A^T_{j,i}e_j \, ,
\end{equation}
for $(\la_1, \dots \la_n)$ the set of eigenvalues of $A^T$. Here, $A^T$ denotes the entry-wise transpose of $A$. In other words, we have $(A^T)_{i,j} = A_{j,i}$, so that the eigenvalues of $A^T$ are those of $A$. We will first show that the set $\{e_if_j^T \}_{i,j}$ is a basis for the linear space $\mat(\C^n, \C^m)$. By looking at the dimension of $\mat(\C^n, \C^m)$, this statement holds if and only if the $e_if_j^T$ are linearly independent over $\C$. To this end, let us write
\begin{equation}\label{lind1}
\sum_{i=1}^m \sum_{j=1}^n a_{i,j} (e_if_j^T) = 0 \, ,
\end{equation}
for $a_{i,j} \in \C$. Let $N \in \mat(\C,n)$ be a matrix satisfying $f_i^TNf_j = \delta_{i,j}$ (for example by setting $N := C^TC$, where $C$ maps the basis $\{ f_j\}_{j=1}^n$ to the standard basis of $\C^n$). Multiplying equation $\eqref{lind1}$ by $Nf_k$ for a given value of $k$ yields

\begin{equation}
\sum_{i=1}^m \sum_{j=1}^n a_{i,j} (e_if_j^T)Nf_k = \sum_{i=1}^m\sum_{j=1}^n a_{i,j} e_i(f_j^TNf_k) = \sum_{i=1}^m  a_{i,k} e_i = 0 \, .
\end{equation}
By linear independence of the basis $\{ e_i\}_{i=1}^m$ we see that $a_{i,k} = 0$ for all $i$. Since $k$ was chosen arbitrary, it follows that $a_{i,k} = 0$ for all $i$ and $k$. This proves that $\{e_if_j^T \}_{i,j}$ is a basis for $\mat(\C^n, \C^m)$. \\
Next, we order the set $\{e_if_j^T \}_{i,j}$ lexicographically. That is, we say that $e_if_j^T > e_kf_l^T$ if $i > k$ or if it holds that $i = k$ and $j > l$. It follows that

\begin{align}
\mathcal{L}_{A,B}(e_if_j^T) &= (e_if_j^T)A - B(e_if_j^T) = e_i(A^Tf_j)^T -  (Be_i)f_j^T  \\  \nonumber
&= e_i( \la_jf_j + \sum_{k<j}A^T_{k,j}f_k)^T - (\mu_ie_i + \sum_{l<i}B_{l,i}e_l)f_j^T  \\ \nonumber
&= (\la_j - \mu_i)(e_if_j^T) + \sum_{k<j}A^T_{k,j}(e_if_k^T) - \sum_{l<i}B_{l,i}(e_lf_j^T)   \\ \nonumber
&=(\la_j - \mu_i)(e_if_j^T) + \{\text{lexicographically lower order terms} \} \nonumber \, .
\end{align}
We see that, with respect to the ordered basis $\{e_if_j^T \}_{i,j}$, the matrix of $\mathcal{L}_{A,B}$ is in upper diagonal form, with diagonal entries $\{(\la_j - \mu_i)\}_{i,j}$. This proves the statement.
\end{proof}

\noindent The following lemma will be key in proving Theorem \ref{main} from Theorem \ref{codimsteady}. It can be seen as an extension of Lemma 6.3 from \cite{RinkSanders3}.

\begin{lem}\label{trivia1}
Let $L \in \End(W)$ be an equivariant linear map and denote by $Z \subset \C$ any subset of the complex numbers. (In this article, $Z$ will either be $\{0\}$ or the imaginary axis.) Write $W = W_Z \oplus W_{Z^c}$ for the decomposition of $W$ into the space spanned by the generalized eigenvectors corresponding to eigenvalues of $L$ in $Z$ ($W_Z$) and in the complement of $Z$ ($W_{Z^c}$). Note that both $W_Z$ and $W_{Z^c}$ are invariant spaces for the symmetry, as well as for $L$. That is, $L$ is in block diagonal form corresponding to this decomposition of $W$. We write  $L_{1,1} := L|_{W_Z}$ and $L_{2,2} := L|_{W_{Z^c}}$ for the two blocks. \\
\noindent Then, there exist an open neighborhood $S \subset \End(W)$ containing $L$ and smooth maps $M: S \rightarrow  \End(W)$, $B_{1}: S \rightarrow \End(W_Z)$ and $B_{2}: S \rightarrow \End(W_{Z^c})$ such that the following holds.

\begin{itemize}
\item $M(L) = \Id$, $B_1(L) = L_{1,1}$, $B_2(L) = L_{2,2}$.
\item $M(X)$ is invertible for all $X \in S$. 
\item $B_1$ and $B_2$ are submersions.
\item For all $X \in S$ it holds that 

\begin{equation}
M(X)XM(X)^{-1} = \begin{pmatrix}
 B_1(X) &  0 \\
0 &  B_2(X)  \\
\end{pmatrix}\, 
\end{equation}
corresponding to the decomposition $W = W_Z \oplus W_{Z^c}$.  
\end{itemize}
\end{lem}

\begin{proof}
Let $\mathcal{M}$ be the linear space of elements $m \in \End(W)$ of the form 

\begin{equation}
m = \begin{pmatrix}
0 &  m_{1,2} \\
m_{2,1} & 0  \\
\end{pmatrix}\, 
\end{equation}
with respect to the decomposition $W = W_Z \oplus W_{Z^c}$. Define the map \\
$\Psi: \mathcal{M} \times \End(W) \rightarrow \mathcal{M}$ given by 

\begin{equation}
\Psi(m,X) = \begin{pmatrix}
0 &  (\exp(m)X\exp(-m))_{1,2} \\
(\exp(m)X\exp(-m))_{2,1} & 0  \\
\end{pmatrix}\, .
\end{equation}
Here, $\exp(m)$ denotes the matrix exponential of $m$, defined by the usual power series. Note that $\exp(m)$ is again an equivariant map, as the set of equivariant maps is closed in the set of all linear maps. We will apply the implicit function theorem to the map $\Psi$. First of all, we have $\Psi(0,L) = 0$, as $L$ is block diagonal with respect to the given decomposition. Secondly, the derivative at $(0,L)$ in the direction of $V \in \mathcal{M}$ is given by

\begin{align}
D_m\Psi{(0,L)}V &= \begin{pmatrix}
0 &  [V,L]_{1,2} \\
[V,L]_{2,1} & 0  \\
\end{pmatrix} \\ &= \nonumber
\begin{pmatrix}
0 &  V_{1,2}L_{2,2} - L_{1,1}V_{1,2} \\
V_{2,1}L_{1,1} -L_{2,2} V_{2,1} & 0  \\
\end{pmatrix}  \\ &= \nonumber
\begin{pmatrix}
0 &  \mathcal{L}_{L_{2,2}, L_{1,1}}(V_{1,2}) \\
\mathcal{L}_{L_{1,1}, L_{2,2}}(V_{2,1}) & 0  \\
\end{pmatrix} \, ,
\end{align}
where $[V,L]$ denotes the commutator between the two operators. Looking at the eigenvalues of $L_{1,1}$ and $L_{2,2}$, we see that the difference between an eigenvalue of the first and an eigenvalue of the second can never be $0$. Hence, it follows from Lemma \ref{cominv} that the operators  $\mathcal{L}_{L_{2,2}, L_{1,1}}$ and  $\mathcal{L}_{L_{1,1}, L_{2,2}}$ are bijections. As they moreover send equivariant maps to equivariant maps, we conclude that they are bijective on the set of equivariant maps. By the implicit function theorem, it therefore holds that there exists a smooth map $m$ from some open neighborhood $S \subset \End(W)$ containing $L$ to $\mathcal{M}$ such that $\Psi(m(X),X) = 0$. It furthermore holds that $m(L) = 0$. By setting $M(X) := \exp(m(X))$, we get a map satisfying $M(L) = \Id$ and with $M(X)$ invertible for all $X \in S$. By construction, $M(X)XM(X)^{-1}$ is of block diagonal form. Finally, we set $B_1(X) := (M(X)XM(X)^{-1})_{1,1}$ and $B_2(X) := (M(X)XM(X)^{-1})_{2,2}$, so that $B_1(L) = L_{1,1}$ and $B_2(L) = L_{2,2}$.\\
\noindent It remains to show that these two smooth maps are in fact submersions. For this, it is enough to show that their derivatives have maximal rank at $L$. The lemma is then proven by choosing $S$ small enough so that the derivatives of $B_1$ and $B_2$ have maximal rank throughout. The derivative of the map $X \mapsto M(X)XM(X)^{-1}$ at $L$ in the direction of $V \in \End(W)$ is given by 

\begin{align}
&\left. \frac{d}{dt} \right|_{t=0} M(L+tV)(L+tV)M(L+tV)^{-1} \\ \nonumber
= &\left. \frac{d}{dt} \right|_{t=0} \exp(m(L+tV))(L+tV)\exp(-m(L+tV)) \\ \nonumber
=&\left. \frac{d}{dt} \right|_{t=0} (\Id + m(L+tV)+ \dots)(L+tV)(\Id - m(L+tV)+ \dots) \\ \nonumber
=&(Dm(L)V)L - L(Dm(L)V) + V \\ \nonumber
=&\, [Dm(L)V, L] + V\, .
\end{align}
As $m(X)$ is an element of $\mathcal{M}$ for all $X \in S$ and as $L$ is block diagonal, we see that $[Dm(L)V, L]_{1,1} = 0$ and $[Dm(L)V, L]_{2,2} = 0$. It follows that ${DB_1}(L)V = V_{1,1}$ and ${DB_2}(L)V = V_{2,2}$, which are indeed of full rank. This proves the lemma.
\end{proof}
\noindent Another, simple lemma that we will use is the following.

\begin{lem}\label{blok1a}
Let $A, B \in End(W)$ be two endomorphisms that are conjugate by an equivariant map. That is, there exists an invertible $M \in End(W)$ such that $B = MAM^{-1}$. As in Lemma \ref{trivia1}, denote by $Z \subset \C$ any subset of the complex numbers. Let $W_Z(X)$ be the span of the generalized eigenvectors corresponding to eigenvalues of $X \in \End(W)$ that lie in $Z$. Likewise, denote by $W_{Z^c}(X)$ the span of the generalized eigenvectors corresponding to eigenvalues of $X$ not in $Z$. Then it holds that $W_Z(A)$ and $W_Z(B)$ are isomorphic representations, and likewise for $W_{Z^c}(A)$ and $W_{Z^c}(B)$. More precisely, $M$ restricts to isomorphisms $M_1:= M|_{W_Z(A)}: W_Z(A) \rightarrow W_Z(B)$ and $M_2:= M|_{W_{Z^c}(A)}: W_{Z^c}(A) \rightarrow W_{Z^c}(B)$, and we have $B|_{W_Z(B)} = M_1A|_{W_Z(A)}M_1^{-1}$ and  $B|_{W_{Z^c}(B)} = M_2A|_{W_{Z^c}(A)}M_2^{-1}$.
\end{lem}

\begin{proof}
It can directly be verified that if $v \in W$ is a generalized eigenvector of $A$ for an eigenvalue $\la \in \C$, then $Mv$ is a generalized eigenvector of $B = MAM^{-1}$ for the same eigenvalue $\la$. Hence we see that $W_Z(B) = MW_Z(A)$ and $W_{Z^c}(B) = MW_{Z^c}(A)$. In particular, as we have $W = W_Z(A) \oplus W_{Z^c}(A) =  W_Z(B) \oplus W_{Z^c}(B)$, we see that we may write $M = M_1 \oplus M_2$ for $M_1:= M|_{W_Z(A)}: W_Z(A) \rightarrow W_Z(B)$ and $M_2:= M|_{W_{Z^c}(A)}: W_{Z^c}(A) \rightarrow W_{Z^c}(B)$. $M_1$ and $M_2$ are furthermore both injective, and therefore both isomorphisms. Now let $v \in W_Z(B)$ be given, then 

\begin{align}
M_1A|_{W_Z(A)}M_1^{-1}(v) &= M_1A|_{W_Z(A)}(M_1^{-1}(v) + M_2^{-1}(0)) \\ \nonumber
&= M_1A|_{W_Z(A)}M^{-1}(v) = M_1AM^{-1}(v) \\ \nonumber
&= MAM^{-1}(v) = B(v) \,.
\end{align}
Therefore, $B|_{W_Z(B)} = M_1A|_{W_Z(A)}M_1^{-1}$. One likewise finds that $B|_{W_{Z^c}(B)} = M_2A|_{W_{Z^c}(A)}M_2^{-1}$. This proves the lemma.
\end{proof}

%so be important in section \ref{Geometry; Counting dimensions} is the following
%
%
%\
\noindent We are now in a position to prove Theorem \ref{main}, assuming Theorem \ref{codimsteady}.

\begin{proof}[Proof that Theorem \ref{codimsteady} implies Theorem \ref{main}]
We fix a decomposition \\ $W = U \oplus U'$. By Theorem \ref{codimsteady}, we may write 

\begin{align}
&\End(U) \supset \nil(U) = \coprod_{i=1}^k M_i\\
&\End(U) \supset \cen(U) = \coprod_{i=1}^l N_i \, ,
\end{align}
where the $M_i$ and $N_i$ are conjugacy invariant, embedded manifolds. In $\End(W)$ we define the sets $M'_i$, consisting of those endomorphisms $A$ with generalized kernel $W_0(A)$ isomorphic to $U$, for which there exists an isomorphism $\phi: W_0(A) \rightarrow U$ such that $\phi A|_{W_0(A)} \phi^{-1} \in M_i$. Analogously, we define $N'_i$ to be given by

\begin{equation}
N'_i := \{A \in \End(W) \mid \exists \, \phi: W_c(A) \rightarrow U \text{ iso, s.t. }\phi A|_{W_0(A)} \phi^{-1} \in N_i  \}\, ,
\end{equation}
where $W_c(A)$ denotes the center subspace of $A$. \\
Note that if $A$ is in some $M'_i$ (or $N'_i$), then its generalized kernel (or center subspace) is isomorphic to $U$. Conversely, if the generalized kernel of $A$ is isomorphic to $U$, then there exists an isomorphism $\phi: W_0(A) \rightarrow U$. The map $\phi A|_{W_0(A)} \phi^{-1} \in \End(U)$ is nilpotent, and hence contained in some $M_i$. We conclude that $A \in M'_i$ for some $1 \leq i \leq k$. Likewise, if $A$ has its center subspace isomorphic to $U$, then $A \in N'_i$ for some $1 \leq i \leq l$. We conclude that the union of the $M'_i$ is exactly all elements in $\End(W)$ with generalized kernel isomorphic to $U$, and likewise for the center subspace case and the $N'_i$. \\
First, we show that the definitions of $M'_i$ and $N'_i$ are independent of the choice of isomorphism $\phi$. If $\phi A|_{W_0(A)} \phi^{-1}$ is an element of $M_i$, and if $\psi: W_0(A) \rightarrow U$ is any other isomorphism, then 

\begin{align}
\psi A|_{W_0(A)} \psi^{-1} &= \psi  \phi^{-1} \phi A|_{W_0(A)} \phi^{-1} \phi \psi^{-1} \\
&= (\psi  \phi^{-1}) \phi A|_{W_0(A)} \phi^{-1}(\psi  \phi^{-1})^{-1}\,.
\end{align}
As $\psi  \phi^{-1} \in \End(U)$ and as $M_i$ is conjugacy invariant, we conclude that\\
\noindent  $\psi A|_{W_0(A)} \psi^{-1} \in M_i$ as well. The same proof works for the $N_i'$. This shows that the sets $M'_i$ are in fact disjoint, and likewise for the $N_i'$. For, if $\phi A|_{W_0(A)} \phi^{-1} \in M_i$ and $\psi A|_{W_0(A)} \psi^{-1} \in M_j$ for some isomorphisms $\phi$ and $\psi$, then by the foregoing, $\phi A|_{W_0(A)} \phi^{-1} \in M_j$. As the $M_i$  are disjoint, we conclude that $i = j$. The same reasoning works to show that the $N_i'$ are disjoint sets. \\
\noindent Next, we show that none of the sets $M'_i$ is empty. To this end, pick an element $B \in M_i$. It follows that the element

\begin{align}
A:= \begin{pmatrix}
B &  0 \\
0 & \Id_{U'}  \\
\end{pmatrix}\in \End(U \oplus U') = \End(W)
\end{align}
belongs to $M'_i$ (by choosing $\phi = \Id_U$). A similar proof shows that none of the $N_i'$ is empty. \\
\noindent It also holds that each $M_i'$ and $N_i'$ is conjugacy invariant. For, if $A$ is an \\ element of $M_i'$ and $B \in \End(W)$ is conjugate to $A$, then by Lemma \ref{blok1a} there exists an isomorphism $M_1: W_0(A) \rightarrow W_0(B)$ such that $M_1A|_{W_0(A)}M_1^{-1} = B|_{W_0(B)}$. By assumption, there exists an isomorphism $\phi: W_0(A) \rightarrow U$ for which $\phi A|_{W_0(A)}\phi^{-1} \in M_i$. It follows that $ \phi M_1^{-1}: W_0(B) \rightarrow U$ is an isomorphism satisfying $(\phi M_1^{-1}) B|_{W_0(B)}(\phi M_1^{-1})^{-1} \in M_i$. This proves that $B \in M_i'$ as well. The proof is analogous for the $N_i'$\\

\noindent It remains to show that the $M'_i$ and $N'_i$ are embedded manifolds satisfying the proposed conditions on their dimensions. We will in fact show that every $M'_i$ has the same codimension as its counterpart $M_i$, and likewise for the $N'_i$ with respect to the $N_i$. \\
\noindent To this end, let $L \in M'_i$ be given. We will set $W_0 := W_0(L)$ and write $W_0^{\circ} := W_0^{\circ}(L)$ for the span of the generalized eigenvectors of $L$ corresponding to its non-zero eigenvalues. It follows that there exists an isomorphism \\ $\phi: W_0 \rightarrow U$ such that $\phi L|_{W_0} \phi^{-1} \in M_i$. We fix such an isomorphism $\phi$. By Lemma \ref{trivia1} there exist an open neighborhood $S \subset \End(W)$ containing $L$ and smooth submersions $B_1: S \rightarrow \End(W_0)$, $B_2: S \rightarrow \End(W_0^{\circ})$ such that every element $A \in S$ is conjugate to $B_1(A) \oplus B_2(A) \in \End(W_0 \oplus W_0^{\circ}) = \End(W)$. It furthermore holds that $B_1(L) = L|_{W_0}$ and $B_2(L) = L|_{W_0^{\circ}}$. As $B_2(L)$ is invertible, there exists an open neighborhood $T \subset \End(W_0^{\circ})$ containing $B_2(L)$ of only invertible linear operators. By redefining $S$ as $S \cap B_2^{-1}(T)$, we may therefore assume $B_2(A)$ to be invertible for all $A \in S$. \\

\noindent We claim that $M'_i \cap S$ is exactly the set of all elements $A \in S$ for which $\phi B_1(A) \phi^{-1} \in M_i$. Because $A \in S$ is conjugate to $B_1(A) \oplus B_2(A)$, it follows from the conjugacy invariance of $M'_i$ that $A$ is an element of $M'_i$ if and only if $B_1(A) \oplus B_2(A)$ is. Therefore, let us first assume $B_1(A) \oplus B_2(A)$ is an element of $M'_i$. It follows that the generalized kernel of $B_1(A) \oplus B_2(A)$ is isomorphic to $U$, and therefore to $W_0$. As $B_2(A)$ is furthermore assumed to be invertible, we see that the generalized kernel of $B_1(A) \oplus B_2(A)$ is necessarily contained in $W_0$. Hence, we conclude equality of the two spaces, i.e. $W_0(B_1(A) \oplus B_2(A)) = W_0$. In particular, we see that $(B_1(A) \oplus B_2(A))|_{W_0(B_1(A) \oplus B_2(A))} = B_1(A)$. As $B_1(A) \oplus B_2(A) \in M'_i$, it holds that $\psi B_1(A) \psi^{-1} \in M_i$ for some isomorphism $\psi: W_0 \rightarrow U$. By the first part of the proof, we also get $\phi B_1(A) \phi^{-1} \in M_i$. \\
Conversely, if $A \in S$ is such that $\phi B_1(A) \phi^{-1} \in M_i$, then the generalized kernel of $B_1(A) \oplus B_2(A)$ contains $W_0$. As $B_2(A)$ is invertible, we see that exactly $W_0(B_1(A) \oplus B_2(A)) = W_0$. From $\phi (B_1(A) \oplus B_2(A))|_{W_0} \phi^{-1} =  \phi B_1(A) \phi^{-1} \in M_i$ we conclude that $B_1(A) \oplus B_2(A) \in M'_i$, and therefore $A \in M'_i$.\\
\noindent From this we see that $M_i' \cap S = B_1^{-1}(\phi^{-1} M_i \phi)$. In particular, as $B_1$ is a submersion, $M_i'$ is an embedded submanifold of $\End(W)$ of the same codimension as $M_i$. \\

\noindent The case for the center subspace is completely analogous. For a given \\ $L \in N'_i$, choose an open neighborhood $S$ on which any element $A$ is conjugate to $B_1(A) \oplus B_2(A)$. Here, the direct sum is with respect to the decomposition $W = W_c(L) \oplus W_c^{\circ}(L)$, where $W_c^{\circ}(L)$ corresponds to all eigenvalues away from the imaginary axis. Analogous to the case of the $M'_i$, we want to assume that for all $A \in S$, $B_2(A)$ has only eigenvalues away from the imaginary axis. This can be assumed if it holds that the set of elements in $\End(W_c^{\circ})$ with no purely imaginary eigenvalues is an open set. However, note that there exists a continuous inclusion from $ \End(W_c^{\circ})$ into $\mat(\C,n)$ for some $n$. In $\mat(\C,n)$, the set of matrices with no purely imaginary eigenvalues is indeed an open set. See for example \cite[p.~118]{kato} or see \cite{Alexanderian2013OnCD} for a short proof using Rouch\'e's theorem. Therefore, the set of elements in $\End(W_c^{\circ})$ with no purely imaginary eigenvalues is indeed open. It follows that $N_i' \cap S = B_1^{-1}(\phi^{-1} N_i \phi)$, where $\phi: W_c(L) \rightarrow U$ is any (fixed) isomorphism. Therefore each $N_i'$ is an embedded submanifold of the same codimension as $N_i$. This proves the theorem.
\end{proof}

\section{Algebraic Reduction} \label{Algebra: Reduction}
The proof of Theorem \ref{codimsteady} consists of two steps. First, we reduce the problem from one involving $\End(W)$ to one involving certain matrix algebras that are easier to analyse. The most important aspect of this reduction is the fact that it does not (in essence) change the spectrum of the endomorphisms. The second step is to then construct the manifolds in these reduced spaces that contain all matrices with a vanishing or purely imaginary spectrum, and to count their dimensions. This section is dedicated to the first step, whereas Sections \ref{Intermezzo; Some algebraic geometry} and \ref{Geometry; Counting dimensions} will cover the second. In Section \ref{proof of main results} we present the proof of Theorem \ref{codimsteady}, using the results from this section and from Sections \ref{Intermezzo; Some algebraic geometry} and \ref{Geometry; Counting dimensions}.   \\

\noindent The first step comes down to three consecutive reductions. In the first reduction, we isolate an ideal in $\End(W)$ whose cosets have a constant spectrum. That is, the algebraic multiplicity of the eigenvalues of an endomorphism does not change when one adds an element of this ideal. We furthermore identify a full set of representatives for the cosets of this ideal. In the second reduction, we show that choosing different generators for the real, complex and quaternionic structure does not change the eigenvalues of the endomorphism, and has a predictable effect on the algebraic multiplicities. In the third step, we further reduce the problem to one involving three families of algebras. This last reduction forgets about some of the eigenvalues. However, the property of having a vanishing or purely imaginary spectrum is still respected. Throughout this section, we have chosen to elaborate on quite some intermediate results, as we believe they have significance outside of the proof of Theorem \ref{codimsteady} as well. 

\subsection{The First Reduction} \label{The First Reduction}
In this part, we identify an ideal in $\End(W)$ whose cosets have a constant spectrum. We furthermore identify a suitable set of representatives for these cosets. The main tool in this subsection will be the following:

\begin{defi}
Writing 

\begin{equation}\label{decomp3}
W \cong \bigoplus^{r_1} W_{1}^R \dots \bigoplus^{r_u} W_{u}^R \bigoplus^{c_1} W_{1}^C \dots \bigoplus^{c_v} W_{v}^C \bigoplus^{h_1} W_{1}^H \dots \bigoplus^{h_w} W_{w}^H\, ,
\end{equation}
we fix an isomorphism between $W$ and the right hand side of equation \eqref{decomp3}. We may then denote any element of $\End(W)$ as a matrix with entries formed by equivariant maps between two (isomorphic or non-isomorphic) indecomposable components of \eqref{decomp3}. We denote by $\mathcal{J} \subset \End(W)$ the set of all elements for which there are no isomorphisms among the entries of this matrix.  Equivalently, $\mathcal{J}$ consists of those endomorphisms for which there are only nilpotent entries between isomorphic components, alongside entries between non-isomorphic components. We will later see in Corollary \ref{corinde} that this definition is independent of the chosen isomorphism between the right hand side and the left hand side of equation \eqref{decomp3}.

\end{defi}
\begin{ex}\label{ex1}
Let $W$ be given by 

\begin{equation}
W \cong W_1^R \oplus W_1^R \oplus W_1^C\, ,
\end{equation}
where $W_1^R$ and $W_1^C$ are (necessarily non-isomorphic) indecomposable representations of real, respectively complex type. An element  $A \in \End(W)$ may then be written with respect to this decomposition as 

\begin{equation}
A = 
\begin{pmatrix}
 a\Id + N_{1,1} &  b\Id + N_{1,2}  & A_{1,3} \\
c\Id + N_{2,1} &  d\Id + N_{2,2}  & A_{2,3} \\
A_{3,1} & A_{3,2} & e\Id + fI + N_{3,3}
\end{pmatrix}\, ,
\end{equation}
for $a,b, \dots f \in \R$. Here, $N_{i,j}$ denotes a nilpotent map between isomorphic representations and $I \in \End(W_1^C)$ is an isomorphism such that $\{[\Id], [I]\} \subset \End(W_1^C)/\nil(W_1^C)$ generates a complex structure. It follows that $A$ is an element of $\mathcal{J}$ if and only if $a = b = \dots = f = 0$. \hspace*{\fill}$\triangle$
\end{ex}

\begin{prop}
The set $\mathcal{J}$ is a (two-sided) ideal in the algebra $\End(W)$.
\end{prop}

\begin{proof}
Recall that $\nil(W_i^K)$ is an ideal in $\End(W_i^K)$ for every indecomposable representation $W_i^K$. In particular, $\nil(W_i^K)$ is a linear subspace of the vector space $\End(W_i^K)$.  From this it follows that $\mathcal{J}$ is a linear subspace of $\End(W)$. \\
To prove that it is an ideal, let us denote the indecomposable components of $W$ (i.e. $W_1^R$ ($r_1$ times) up to  $W_w^H$ ($h_w$ times)) by $W_1, \dots W_k$ for $k = r_1 + \dots r_u + c_1 + \dots c_v + h_1+ \dots h_w$. Let us furthermore denote the entries of $A \in \End(W)$ and $X \in \mathcal{J}$ by $A = (A_{i,j})$ and $X = (X_{i,j})$ with respect to this decomposition of $W$. If $p$ and $q$ are indices such that $W_p$ is isomorphic to $W_q$, then we have

\begin{equation}
(AX)_{p,q} = \sum_{l = 1}^k A_{p,l}X_{l,q} = \sum_{l \in P(p)} A_{p,l}X_{l,q} + N \,.
\end{equation}
Here $P(p)$ denotes the set of indices of representations isomorphic to $W_p$ and $N$ is some nilpotent map. Now, because $X$ is an element of $\mathcal{J}$, we know that $X_{l,q}$ is an element of $\nil(W_p)$ for all $l \in P(p)$. Using the fact that $\nil(W_p)$ is an ideal in $\End(W_p)$ we conclude that the entire term $(AX)_{p,q}$ is nilpotent. Since this holds for all $p$ and $q$ such that $W_p$ is isomorphic to $W_q$, we see that $AX \in \mathcal{J}$. This proves that $\mathcal{J}$ is a left ideal in $\End(W)$. The proof that it is a right ideal is similar, which concludes the proof.
\end{proof}

\begin{cor}\label{corinde}
The ideal $\mathcal{J}$ is independent of the decomposition of $W$ into indecomposable representations.
\end{cor}

\begin{proof}
Let 

\[
d_0, d_1:W \rightarrow W_1 \oplus \dots \oplus W_k
\]
denote two identifications of $W$ with the sum of indecomposable representations $W_1$ till $W_k$. We will furthermore denote by 

\[
\mathcal{J}' \subset \End(W_1 \oplus \dots \oplus W_k)
\]
the ideal of endomorphisms without isomorphisms among the entries. We have to show that 

\begin{equation}\label{jacwed}
d_0^{-1}\mathcal{J}'d_0 = d_1^{-1}\mathcal{J}'d_1 \, .
\end{equation}
However, from the fact that $\mathcal{J}'$ is an ideal it follows that

\begin{equation}
(d_1d_0^{-1})\mathcal{J}'(d_0d_1^{-1})\subset \mathcal{J}' \, ,
\end{equation}
and
\begin{equation}
(d_0d_1^{-1})\mathcal{J}'(d_1d_0^{-1})\subset \mathcal{J}' \, .
\end{equation}
This shows that equation \eqref{jacwed} indeed holds, which concludes the proof.
\end{proof}

\begin{prop}\label{eigen1}
Let $A \in \End(W)$ and let $X \in \mathcal{J}$. The set of eigenvalues of $A$, counted with algebraic multiplicity, is the same as that of $A+X$. In other words, the map assigning the set of eigenvalues to an endomorphism descends to a map on $\End(W)/\mathcal{J}$. In particular, we have that $\mathcal{J} \subset \nil(W)$.
\end{prop}
\noindent To prove Proposition \ref{eigen1} we need the following useful lemma. This result is known to experts, but included here for completeness and for its significance throughout this section. 

\begin{lem} \label{samemat}
Let  $A, B \in \mat(\C,n)$ (or in particular $\mat(\R,n)$) be two $n$ times $n$ matrices such that the following identities hold 

\begin{equation}
\begin{aligned}
\tr(A) &= \tr(B)\\
\tr(A^2) &= \tr(B^2)\\
&\hspace{1.8mm} \vdots \\
\tr(A^n) &= \tr(B^n)\,,
\end{aligned}
\end{equation}
then the set of eigenvalues of $A$ and $B$, counted with algebraic multiplicity, are the same.
\end{lem}

\begin{proof}
Let $(\lambda_1, \dots \lambda_n)$ and $(\mu_1, \dots \mu_n)$ denote the eigenvalues of $A$, respectively $B$ (taking into account algebraic multiplicity). We see that $\tr(A^k) = p_k(\lambda_1 \dots \lambda_n)$ for all $k \geq 0$, where $p_k$ is the \textit{power sum symmetric polynomial} given by 

\begin{equation}
p_k(x_1.,\dots x_n) = x_1^k + \dots + x_n^k \, .
\end{equation}
It is known that these polynomials form a basis of the symmetric polynomials. In other words, every symmetric polynomial in $n$ variables can be written as a polynomial expression of the functions  $p_0$ until $p_n$. (Note that $p_0 := n$). See \cite[p.~2-3]{smith}. In particular, the coefficients of the polynomial 

\begin{equation}\label{pol1}
(x-\lambda_1)(x - \lambda_2) \dots (x-\lambda_n) 
\end{equation}
are symmetric polynomials in the variables $\lambda_1$ till $\lambda_n$. It follows that they can be expressed in the symmetric polynomials $p_0$ till $p_n$. Therefore, they are determined by the values $\tr(A^k) = p_k(\lambda_1 \dots \lambda_n)$ for  $k \leq n$. We conclude that the polynomial $\eqref{pol1}$ is equal to the polynomial 

\begin{equation}\label{pol2}
(x-\mu_1)(x - \mu_2) \dots (x-\mu_n) \, ,
\end{equation}
and from this we see that the roots $(\lambda_1, \dots \lambda_n)$ and $(\mu_1, \dots \mu_n)$ coincide. This proves the lemma.
\end{proof}

\begin{proof}[Proof of Proposition \ref{eigen1}]
We first note that $\tr(X) = 0$ for all $X \in \mathcal{J}$. This follows from the fact that $X$ has only nilpotent maps as its diagonal entries. In particular, we see that $\tr(A) = \tr(A+X)$ for all $A \in \End(W)$. Furthermore, from the fact that $\mathcal{J}$ is an ideal in $\End(W)$ it follows that $(A+X)^m = A^m + X_m$ for some $X_m \in \mathcal{J}$ and for all $m>0$. From this we conclude that $\tr((A+X)^m) = \tr(A^m)$ for all $m>0$. The claim of the theorem now follows from applying Lemma \ref{samemat} to $A$ and $A+X$. In particular, we conclude that $X$ has only $0$ as an eigenvalue and is hence nilpotent.
\end{proof}
Before we move on, it will be convenient to introduce a full set of representatives for the classes of $\End(W)/\mathcal{J}$. To this end, let $W_i$ be an indecomposable representation. If $W_i$ is of quaternionic type, we fix isomorphisms $\{\Id_{i}, I_{i}, J_{i}, K_{i} \}\subset \End(W_{i})$ such that $\{[\Id_{i}], [I_{i}], [J_{i}], [K_{i}]\} \subset \End(W_{i})/\nil(W_{i})$ generates the quaternionic structure on  $ \End(W_{i})/\nil(W_{i})$. Likewise we will have that $ [\Id_{i}] \in \End(W_{i})/\nil(W_{i})$ generates the real structure on \\ $ \End(W_{i})/\nil(W_{i})$ if $W_i$ is of real type and that $\{[\Id_{i}], [I_{i}]\} \subset \End(W_{i})/\nil(W_{i})$ generates the complex structure on $ \End(W_{i})/\nil(W_{i})$ if $W_i$ is of complex type. We note in passing that $\Id_{i} \in \End(W_{i})$ may be chosen to equal the identity operator $\Id_{W_i}$. As a matter of fact, because $\Id_{W_i}^2 = \Id_{W_i}$ and because there is only one non-zero idempotent element in any division ring, we see that necessarily $\Id_{i} = \Id_{W_i} + N$ for any nilpotent element $N$. Therefore, we will always choose $\Id_{i}$ to be $\Id_{W_i}$. \\
Given this choice of generators, we will construct out of an endomorphism $A \in \End(W)$ an endomorphism $D_A \in \End(W)$ that only differs from $A$ by an element of $\mathcal{J}$. 
To this end, we write

\begin{equation}\label{decomp}
\begin{split}
W &\cong \bigoplus^{r_1} W_{1}^R \dots \bigoplus^{r_u} W_{u}^R \bigoplus^{c_1} W_{1}^C \dots \bigoplus^{c_v} W_{v}^C \bigoplus^{h_1} W_{1}^H \dots \bigoplus^{h_w} W_{w}^H\\
&= \bigoplus_{i=1}^{k} W_{i} \, ,
\end{split}
\end{equation}
for $k = r_1 + \dots r_u + c_1 + \dots c_v + h_1+ \dots h_w$ and where each $W_i$ is equal to one of the $W_{j}^K$, $K \in \{R,C,H\}$ from the first line of \eqref{decomp}. If $p, q \in \{1. \dots k\}$ are such that $W_p = W_q$, then $A_{p,q}$ can be written uniquely as

\begin{equation}\label{exprkrt}
A_{p,q} = \left\{
\begin{tabular}{l l}
 $a\Id_p + N$ & if $W_p$ is of real type \\
 $a\Id_p + bI_p + N$ & if $W_p$ is of complex type \\
 $a\Id_p + bI_p + cJ_p + dK_p+N$ & if $W_p$ is of quaternionic type \\
\end{tabular} \right. \, .
\end{equation}
Here, $a,b,c,d \in \R$ and $N \in \End(W_p)$ is a nilpotent endomorphism. (Note that $\Id_p = \Id_q$, $I_p = I_q$ and so on, as we assume $W_p = W_q$. In other words, if some of the $W_i$ in equation \eqref{decomp} are the same, then they are given the same generators for the division algebra). We then define $(D_A)_{p,q}$ by simply removing the nilpotent terms:

\begin{equation}
(D_A)_{p,q} := \left\{
\begin{tabular}{l l}
 $a\Id_p $ & if $W_p$ is of real type \\
 $a\Id_p + bI_p$ & if $W_p$ is of complex type \\
 $a\Id _p+ bI_p + cJ_p + dK_p$ & if $W_p$ is of quaternionic type \\
\end{tabular} \right. \, .
\end{equation}
Finally, for $p$ and $q$ such that $W_p$ and $W_q$ are non-isomorphic, we set

\begin{equation}
(D_A)_{p,q} := 0 \, .
\end{equation}
As a result, $D_A$ is a block-diagonal endomorphism, where the blocks correspond to isomorphic representations (sometimes referred to as the \textit{isotypical components} of the representation). By construction, we  see that $A$ and $D_A$ differ only by an element of $\mathcal{J}$. What is more, if $A$ and $B$ in $\End(W)$ are in the same coset with respect to $\mathcal{J}$, then the real numbers $a$, $b$, $c$ and $d$ as in equation \eqref{exprkrt} will have to be the same. From this we see that necessarily $D_A = D_B$. We conclude that the elements $D_A$ for $A \in \End(W)$ form a full set of representatives for the cosets of $\mathcal{J}$.

\begin{ex}\label{ex2}
As in Example \ref{ex1}, let $W$ be given by 

\begin{equation}
W = W_1^R \oplus W_1^R \oplus W_1^C\, ,
\end{equation}
where $W_1^R$ and $W_1^C$ are indecomposable representations of real, respectively complex type. An element $A \in \End(W)$ is given with respect to this decomposition as 

\begin{equation}
A = 
\begin{pmatrix}
 a\Id + N_{1,1} &  b\Id + N_{1,2}  & A_{1,3} \\
c\Id + N_{2,1} &  d\Id + N_{2,2}  & A_{2,3} \\
A_{3,1} & A_{3,2} & e\Id + fI + N_{3,3}
\end{pmatrix}\, .
\end{equation}
It follows that $D_A$ is the block diagonal matrix 

\begin{equation}
D_A = 
\begin{pmatrix}
 a\Id &  b\Id   & 0 \\
c\Id  &  d\Id  & 0 \\
0& 0 & e\Id + fI 
\end{pmatrix}\, .
\end{equation} \hspace*{\fill}$\triangle$
\end{ex}

As $A$ and $D_A$ differ by an element of $\mathcal{J}$, it follows from Proposition \ref{eigen1} that they have the same eigenvalues, counted with algebraic multiplicity. Furthermore, as the set of endomorphisms $\{D_A \mid A \in \End(W)\}$ forms a complete set of representatives for the equivalence classes of $\End(W)/\mathcal{J}$, we see that we can define any map on $\End(W)/\mathcal{J}$ by specifying its value on the elements $D_A$. \\
Note that $D_A + D_B = D_{A+B}$ and $\mu \cdot D_A = D_{\mu \cdot A}$ for all $A,B \in \End(W)$ and $\mu \in \R$. It does however not hold that $D_AD_B = D_{AB}$. Indeed, the element $D_AD_B$ is again a block diagonal endomorphism, but it may have nilpotent terms among its entries. To realize this, we note that for example the identity $[I]^2 = -[\Id]$ for $[I] \in  \End(W_{i}^C)/\nil(W_{i}^C)$ does not imply $I^2 = -\Id$, but rather $I^2 = -\Id + N$ for some nilpotent term $N$. This seems to suggest that $D_AD_B$ and $D_{AB}$ differ only by an element of $\mathcal{J}$, which the following theorem confirms. 

\begin{prop}\label{prods}
Let $\{A_i\}_{i=1}^m \subset \End(W)$  be any finite set of endomorphisms, then the following holds:

\begin{equation}
\prod_{i=1}^m D_{A_i} - D_{\left(\prod_{i=1}^m A_i\right)} \in \mathcal{J}\, .
\end{equation}
\end{prop}

\begin{proof}
Given $A \in \End(W)$, we denote by $A+\mathcal{J}$ the set of endomorphisms that differ from $A$ by an element of $\mathcal{J}$. We have already seen that $D_{A_i} \in A_i + \mathcal{J}$ for all $i \in \{1, \dots m\}$. Consequently, it follows that 

\begin{equation}
\prod_{i=1}^m D_{A_i} \in \prod_{i=1}^m (A_i + \mathcal{J}) \subset  \left(\prod_{i=1}^m A_i\right) + \mathcal{J}  \, .
\end{equation}
Since it also holds that 

\begin{equation}
D_{\left(\prod_{i=1}^m A_i\right)} \in   \left(\prod_{i=1}^m A_i \right) + \mathcal{J}  \, ,
\end{equation}
we see that $\prod_{i=1}^m D_{A_i}$ and $D_{\left(\prod_{i=1}^m A_i\right)}$ indeed differ by an element of  $\mathcal{J}$. This proves the claim.
\end{proof}

\subsection{The Second Reduction} \label{The Second Reduction}

Our goal is to define a map from $\End(W)/\mathcal{J}$ to a space where it is easier to identify exactly those matrices that have a specified condition on their eigenvalues. Since the set $\{D_A \mid A \in \End(W)\}$ forms a complete set of representatives for the equivalence classes of $\End(W)/\mathcal{J}$, we may define this map by giving its value on endomorphisms of the form $D_A$. Since $D_A$ is a block diagonal matrix, where the blocks correspond to isomorphic indecomposable representations, we will first focus our attention on the case that $W$ is the direct sum of isomorphic indecomposable representations. 

\noindent The key point will be that we may replace $I$, $J$ and $K$ by any other real  repre\-sentation of the quaternionic (or complex, or real) numbers, without losing information about the eigenvalues. More specifically, we make the following construction.

\begin{defi}
Let $Q :=\{ \tilde{I}, \tilde{J},\tilde{K}\} \subset \mat(\R,m)$ denote any three $m \times m$ matrices such that  $\R\Id_m \oplus \R \tilde{I} \oplus \R \tilde{J} \oplus \R \tilde{K}$ has a quaternionic structure. That is, we have the identities

\begin{equation}
\tilde{I}^2 = \tilde{J}^2 = \tilde{K}^2 = \tilde{I}\tilde{J}\tilde{K} = -\Id_m \, .
\end{equation}
Given an element 

\[ A \in \End\left(\bigoplus^{h_i} W_{i}^H\right)\, ,\] 
we may write

\begin{equation}
A_{p,q} = a_{p,q}\Id +b_{p,q}I + c_{p,q}J+ d_{p,q}K+ N_{p,q} \, ,
\end{equation}
for $p,q \in \{1, \dots h_i\}$ and where $\Id, I,J,K \in \End(W_i^H)$ generate the quaternionic structure on $\End(W_i^H)/\nil(W_i^H)$. Here we furthermore have that $a_{p,q}$, $b_{p,q}$, $c_{p,q}$ and $d_{p,q}$ are real numbers and that $N_{p,q} \in \End(W_{i}^H)$ is a nilpotent endomorphism. From $A$, we can now construct the element 

\[ A_Q \in \mat\left(\bigoplus^{h_i} \R^m \right) \, ,\]
given by 

\begin{equation}
(A_Q)_{p,q} := a_{p,q}\Id_m +b_{p,q}\tilde{I} + c_{p,q}\tilde{J}+ d_{p,q}\tilde{K} \, .
\end{equation}
Similarly, if 
\[ A \in \End\left(\bigoplus^{k_i} W_{i}^K\right), \quad K \in \{R,C\} \]
has components of complex or real type, we can define $A_Q$ corresponding to some real representation $Q$ of the complex or real numbers. For example, given a matrix $\tilde{I} \subset \mat(\R,m)$ satisfying $\tilde{I}^2 = -\Id_m$, and writing

\begin{equation}
A = 
\begin{pmatrix}
a_{1,1}\Id +b_{1,1}I + N_{1,1} &\cdots & a_{1,c_i}\Id + b_{1,c_i}I + N_{1,c_i} \\
\vdots & \ddots & \vdots \\
a_{c_i,1}\Id + b_{c_i,1}I + N_{c_i,1} &\cdots & a_{c_i,c_i}\Id + b_{c_i,c_i}I + N_{c_i,c_i}
\end{pmatrix}\, ,
\end{equation}
for an element

\[ A \in \End\left(\bigoplus^{c_i} W_{i}^C \right)\, ,\] 
we get

\begin{equation}
A_Q = 
\begin{pmatrix}
a_{1,1}\Id_m +b_{1,1}\tilde{I}  &\cdots & a_{1,c_i}\Id_m + b_{1,c_i}\tilde{I} \\
\vdots & \ddots & \vdots \\
a_{c_i,1}\Id_m + b_{c_i,1}\tilde{I}&\cdots & a_{c_i,c_i}\Id_m + b_{c_i,c_i}\tilde{I}
\end{pmatrix}\, .
\end{equation}
\end{defi}

%
%
%
%\begin{equation}
%\begin{split}
%A = 
%&\left( \begin{matrix}
%a_{1,1}\Id +b_{1,1}I + c_{1,1}J+ d_{1,1}K+ N_{1,1} &\cdots  \\
%\vdots & \ddots \\
%a^{i,H}_{h_i,1}\Id + b^{i,H}_{h_i,1}I+ c^{i,H}_{h_i,1}J+ d^{i,H}_{h_i,1}K + N^{i,H}_{h_i,1} &\cdots 
%\end{matrix}  \right. \\
%& \\
%&\left. \begin{matrix} 
%a^{i,H}_{1,h_i}\Id + b^{i,H}_{1,h_i}I+  c^{i,H}_{1,h_i}J+ d^{i,H}_{1,h_i}K + N^{i,H}_{1,h_i} \\
%\vdots \\
%a^{i,H}_{h_i,h_i}\Id + b^{i,H}_{h_i,h_i}I+ c^{i,H}_{h_i,h_i}J+ d^{i,H}_{h_i,h_i}K + N^{i,H}_{h_i,h_i} \\
%\end{matrix} \right) \, .
%\end{split}
%\end{equation}
%
%
%
%
%
%
%we can construct the element 
%
%\[ A_Q \in \End\left(\bigoplus^{h_i} \R^m \right)\]
%
%\begin{equation}
%\begin{split}
%&A_Q = \\
%&\begin{pmatrix}
%a_{1,1}\Id_m +b_{1,1}\tilde{I} + c_{1,1}\tilde{J}+ d_{1,1}\tilde{K} &\cdots & a_{1,h_i}\Id_m + b_{1,h_i}\tilde{I}+  c_{1,h_i}\tilde{J}+ d_{1,h_i}\tilde{K} \\
%\vdots & \ddots & \vdots \\
%a_{h_i,1}\Id_m + b_{h_i,1}\tilde{I}+ c_{h_i,1}\tilde{J}+ d_{h_i,1}\tilde{K} &\cdots & a_{h_i,h_i}\Id_m + b_{h_i,h_i}\tilde{I}+ c_{h_i,h_i}\tilde{J}+ d_{h_i,h_i}\tilde{K} 
%\end{pmatrix}\, ,
%\end{split}
%\end{equation}
%
%
%
%
%
%
%
%
%
%
%
%Similarly, if 
%\[ A \in \End\left(\bigoplus^{k_i} W_{i}^K\right), \quad K \in \{R,C\}\]
%has components of complex or real type, we can define $A_Q$ corresponding to some real representation $Q$ of the complex or real numbers. 
\noindent The following theorem relates the eigenvalues of $A$ and $A_Q$.

\begin{prop}\label{repres1}
Let $n$ be the dimension of $W_{i}^K$ and let $Q \subset \mat(\R,m)$ be a representation of the division algebra $\K$ by real $m \times m$ matrices. Then, a complex number $\lambda$ is an eigenvalue of $A$ if and only it is an eigenvalue of the matrix $A_Q$. Furthermore, if we denote the algebraic multiplicity of $\lambda$ in $A$ by $m_{\la}$ and its algebraic multiplicity in $A_Q$ by $m_{\la}'$ then these numbers satisfy

\begin{equation}
m_{\lambda} \cdot m = m_{\lambda}' \cdot n 
\end{equation}
\end{prop}
\noindent In order to prove Proposition \ref{repres1} we will need the following useful lemma.

\begin{lem}\label{mapalg}
Let $U$ and $V$ be two (real or complex) finite dimensional vector spaces of the same dimension, and let $\mathcal{A} \subset \mat(U,U)$ and $\mathcal{B} \subset \mat(V,V)$ be two sub-algebras of the algebras of linear operators. Suppose furthermore that $\psi: \mathcal{A} \rightarrow \mathcal{B}$ is a map satisfying 
\begin{itemize}
\item $\psi(A_1A_2) = \psi(A_1)\psi(A_2)$ for all $A_1, A_2 \in \mathcal{A}$ \, ,
\item $\tr(\psi(A)) = \tr(A)$ for all $A \in \mathcal{A}$\, .
\end{itemize}
Then, $A$ and $\psi(A)$ have the same eigenvalues, counted with algebraic multiplicity. 
\end{lem}

\begin{proof}
From the properties of $\psi$ it follows that

\begin{equation}
\tr(\psi(A)^n) = \tr(\psi(A^n)) = \tr(A^n)\, ,
\end{equation}
for all $n>0$ and $A \in \mathcal{A}$. Note furthermore that $\tr(A^0) = \dim(U) = \dim(V) = \tr(\psi(A)^0)$. It follows from Lemma \ref{samemat} that the eigenvalues of $A$ and $\psi(A)$, counted with algebraic multiplicity, coincide. This proves the lemma.
\end{proof}

\begin{proof}[Proof of Proposition \ref{repres1}]
From the operator 

\[ A \in \End\left(\bigoplus^{k_i} W_{i}^K\right)\]
we may construct the operator

\[ \Id_m  \otimes A  \in \mat\left( \R^m \otimes \bigoplus^{k_i} W_{i}^K  \right) \, .\] 
Note that a value $\la \in \C$ is an eigenvalue of $A$ with algebraic multiplicity $m_{\la}$, if and only if it is an eigenvalue of $\Id_m \otimes A $ with algebraic multiplicity $m_{\la} \cdot m$. Likewise, for 

\[ A_Q \in \mat\left(\bigoplus^{k_i} \R^m \right)\]
we may construct

\[\Id_n \otimes A_Q  \in \mat\left(\R^n \otimes  \bigoplus^{k_i} \R^{m}  \right) \, ,\]
and a value $\la \in \C$ is an eigenvalue of $A_Q$ with algebraic multiplicity $m_{\la}'$, if and only if it is an eigenvalue of $ \Id_n \otimes A_Q$ with algebraic multiplicity $m_{\la}' \cdot n$. Our aim is to apply Lemma \ref{mapalg} to the map 

\begin{equation}
\begin{split}
\psi'_Q: \quad &\left\{ \Id_m \otimes A \mid A \in \End\left(\bigoplus^{k_i} W_{i}^K\right)\right\} \subset \mat\left(\R^m \otimes \bigoplus^{k_i}  W_{i}^K \right) \\ 
\rightarrow  \, &\left\{\Id_n \otimes B \mid B \in \mat\left(\bigoplus^{k_i} \R^{m} \right) \right\} \subset \mat\left(\R^n \otimes \bigoplus^{k_i} \R^{m} \right) \\
& \Id_m \otimes A  \mapsto \Id_n \otimes A_Q
\end{split}
\end{equation} 
Note that by construction we have

\begin{equation}
\dim\left(\R^m \otimes \bigoplus^{k_i} W_{i}^K  \right) = \dim \left(\R^n \otimes \bigoplus^{k_i} \R^{m}\right) \, ,
\end{equation}
both being equal to $k_i \cdot m \cdot n$.\\

\noindent The first thing to show is that 

\begin{equation}
\begin{split}
&\psi'_Q(\Id_m \otimes A  \cdot \Id_m \otimes B ) = \psi'_Q(\Id_m \otimes A ) \cdot \psi'_Q(\Id_m \otimes B)  \\
&\text{ for all } A,B \in \End\left(\bigoplus^{k_i} W_{i}^K\right) \,.
\end{split}
\end{equation}
To this end, let us write 

\begin{equation}
A_{i,j} = a_{i,j,0}\Id + a_{i,j,1}I + a_{i,j,2}J+ a_{i,j,3}K+ N^A_{i,j}\, ,
\end{equation}
and
\begin{equation}
B_{i,j} = b_{i,j,0}\Id + b_{i,j,1}I + b_{i,j,2}J+ b_{i,j,3}K+ N^B_{i,j}\, ,
\end{equation}
where the terms $N^A_{i,j}$ and $N^B_{i,j}$ are nilpotent. We see that

\begin{equation}
\begin{aligned}
(AB)_{i,j} &= \sum_{l =1}^{k_i} (a_{i,l,0}\Id + a_{i,l,1}I + a_{i,l,2}J+ a_{i,l,3}K+ N^A_{i,l})\\ 
 &\hspace{0.67cm} \cdot \,(b_{l,j,0}\Id + b_{l,j,1}I + b_{l,j,2}J+ b_{l.j,3}K+ N^B_{l,j}) \\
&=  c_{i,j,0}\Id + c_{i,j,1}I + c_{i,j,2}J+ c_{i,j,3}K+ {N}_{i,j}  \, ,
\end{aligned}
\end{equation}
for a nilpotent term ${N}_{i,j}$ and where the coefficients $c_{i,j,h}$ are determined by the regular quaternionic (or real, or complex) multiplication. That is, we have 

\begin{equation}
\begin{split}
\sum_{l = 1}^{k_i} (&a_{i,l,0} + a_{i,l,1}\tilde{i}+ a_{i,l,2}\tilde{j}+ a_{i,l,3}\tilde{k})(b_{l,j,0}+ b_{l,j,1}\tilde{i} + b_{l,j,2}\tilde{j}+ b_{l,j,3}\tilde{k})  \\
= \, &c_{i,j,0} + c_{i,j,1}\tilde{i} + c_{i,j,2}\tilde{j}+ c_{i,j,3}\tilde{k}
\end{split}
\end{equation}
in $\mathbb{H} = \spn\{1, \tilde{i}, \tilde{j}, \tilde{k}\}$. (We use tildes to distinguish from indices.)
It follows that 

\begin{equation}\label{produ0}
([AB]_Q)_{i,j} = c_{i,j,0}\Id_m + c_{i,j,1}\tilde{I} + c_{i,j,2}\tilde{J}+ c_{i,j,3}\tilde{K} \, .
\end{equation}
On the other hand, we see that 

\begin{equation}
(A_Q)_{i,j} = a_{i,j,0}\Id_m + a_{i,j,1}\tilde{I} + a_{i,j,2}\tilde{J}+ a_{i,j,3}\tilde{K}\, ,
\end{equation}
and
\begin{equation}
(B_Q)_{i,j} = b_{i,j,0}\Id_m + b_{i,j,1}\tilde{I} + b_{i,j,2}\tilde{J}+ b_{i,j,3}\tilde{K}\, ,
\end{equation}
from which it follows that 

\begin{equation}\label{produ1}
(A_QB_Q)_{i,j} = c_{i,j,0}\Id_m + c_{i,j,1}\tilde{I} + c_{i,j,2}\tilde{J}+ c_{i,j,3}\tilde{K} \, .
\end{equation}
Comparing expressions \eqref{produ0} and \eqref{produ1} we see that $(AB)_Q = A_QB_Q$. From this it follows that $\Id_n \otimes (AB)_Q = \Id_n \otimes (A_QB_Q)$ and hence that indeed \\
 $\psi'_Q(\Id_m \otimes A  \cdot \Id_m \otimes B ) = \psi'_Q(\Id_m \otimes A ) \cdot \psi'_Q(\Id_m \otimes B)$. \\
 
\noindent Finally, we have to show that $\tr(\Id_m \otimes A)  = \tr(\psi'_Q(\Id_m \otimes A))$. For this, let $S \in \mat(\R,s)$ be any real $s \times s$ matrix satisfying $S^2 = -\Id_s + N$ for some nilpotent matrix $N \in \mat(\R,s)$ (possibly the zero matrix). It follows that the eigenvalues of $S$ satisfy the equation $\lambda^2 = -1$ and are hence all equal to $i$ or $-i$. As the matrix $S$ is real, it follows that the algebraic multiplicity of $i$ is the same as that of $-i$, and we conclude that necessarily $\tr(S) = 0$. In particular, we see that

 \begin{equation}
 \tr(I) = \tr(J) = \tr(K) = \tr(\tilde{I}) = \tr(\tilde{J}) = \tr(\tilde{K}) = 0 \, .
 \end{equation}
Using this, we calculate the trace of $A$ as

\begin{equation}
\tr(A) = \sum_{l=1}^{k_i} \tr(A_{l,l}) = \sum_{l=1}^{k_i} \tr(a_{l,l,0}\Id_{W_i^K}) = n \cdot \sum_{l=1}^{k_i} a_{l,l,0} \, .
\end{equation}
It follows that 

\begin{equation}
\tr(\Id_m \otimes A) = mn \cdot \sum_{l=1}^{k_i} a_{l,l,0} \, .
\end{equation}
Likewise, we argue that 

\begin{equation}
\tr(A_Q) = \sum_{l=1}^{k_i} \tr(a_{l,l,0}\Id_{m}) = m \cdot \sum_{l=1}^{k_i} a_{l,l,0} \, ,
\end{equation}
and hence that 

\begin{equation}
\tr(\Id_n \otimes A_Q) = mn \cdot \sum_{l=1}^{k_i} a_{l,l,0} \, .
\end{equation}
We conclude that indeed $\tr(\Id_m \otimes A)  = \tr(\psi'_Q(\Id_m \otimes A))$. From Lemma \ref{mapalg} we now see that the numbers $m_{\la}$ and $m_{\la}'$ are indeed related by the identity $m_{\lambda} \cdot m = m_{\lambda}' \cdot n$. This proves the theorem.
\end{proof}

\begin{remk}\label{switch}
The matrix $A_Q$ can be seen as an element of $\mat(\R,k_i) \otimes \mat(\R,m)$. It will often be convenient to use the matrix \\ 
$A^Q \in \mat(\R,m) \otimes \mat(\R,k_i)$ instead though, which is simply $A_Q$ conjugated by the natural braiding isomorphism 

\begin{equation}
B: \R^{k_i} \otimes \R^m \rightarrow \R^{m} \otimes \R^{k_i} \, .
\end{equation}
In particular $A_Q$ and $A^Q$ have the same eigenvalues, counted with algebraic multiplicity. \hspace*{\fill}$\triangle$
\end{remk}

\begin{defi}
Let $Q$ be a particular choice of real matrices representing the division algebra $\End(W_i^K)/\nil(W_i^K)$. We denote by $\mathcal{R}^Q_{k_i}$, $\mathcal{C}^Q_{k_i}$ and $\mathcal{H}^Q_{k_i}$ the linear space of all matrices $A^Q$ for $A \in \End\left(\bigoplus^{k_i} W_{i}^K\right)$ when $W_{i}^K$ is of real, complex or quaternionic type, respectively.
\end{defi}
\noindent We will now make a particular choice for the representation $Q$ of the division algebra $\K$. If $\K = \R$ then we will represent $1 \in \R$ simply by $\Id_1 \in \mat(\R,1)$. For $\K = \C$ we will choose $\Id_2 \in \mat(\R,2)$ and 

\begin{equation}
\tilde{I} :=  \begin{pmatrix}
  0&1\\
  -1 & 0
 \end{pmatrix} \in \mat(\R,2)\,.
\end{equation}
For $\K = \mathbb{H}$ we choose $\Id_4 \in \mat(\R,4)$ and 

\begin{equation}
\tilde{I}:= \begin{pmatrix}
  0&1&0&0\\
  -1 & 0 & 0 & 0\\
  0 & 0 & 0 & -1\\
  0 & 0 & 1 & 0
 \end{pmatrix} \, ,
\end{equation}

\begin{equation}
\tilde{J}:= \begin{pmatrix}
  0&0&1&0\\
  0 & 0 & 0 & 1\\
  -1 & 0 & 0 & 0\\
  0 & -1 & 0 & 0
 \end{pmatrix} 
\end{equation}\,
and
\begin{equation}
\tilde{K}:= \begin{pmatrix}
  0&0&0&1\\
  0 & 0 & -1 & 0\\
  0 & 1 & 0 & 0\\
  -1 & 0 & 0 & 0 
 \end{pmatrix} \, ,
\end{equation}
all in $\mat(\R,4)$. It is easily verified that $\tilde{I}^2 = \tilde{J}^2 = \tilde{K}^2 = \tilde{I}\tilde{J}\tilde{K} = -\Id_4$.\\
With this specific choice of $Q$, we get the following algebras.

\begin{equation}
\begin{split}
\mathcal{R}^Q_n &= \mat(\R,n)\\
\mathcal{C}^Q_n &=\left \{  \begin{pmatrix}
  A&B\\
  -B & A
 \end{pmatrix}, \quad A,B \in \mat(\R,n) \right \}\\
 \mathcal{H}^Q_n &=\left \{  \begin{pmatrix}
  A&B&C&D\\
  -B & A & -D & C\\
  -C & D & A & -B\\
  -D & -C & B & A
 \end{pmatrix}, \quad A,B, C , D \in \mat(\R,n) \right \}
\end{split} \, .
\end{equation}
We will sometimes write $\mathcal{K}^Q_{n} $ to denote one of these algebras when the type (real, complex or quaternionic) is clear. We furthermore define maps 

\begin{equation}
\begin{split}
\psi^Q_{K,k_i} : &\End\left(\bigoplus^{k_i} W_{i}^K\right) \rightarrow \mathcal{K}^Q_{k_i} \\
& A \mapsto A^Q
\end{split}
\end{equation}
for $K \in \{R,C,H\}$. (Here,  $\mathcal{K}^Q_{k_i}$ denotes  $\mathcal{R}^Q_{k_i}$  when $K = R$, $\mathcal{C}^Q_{k_i}$  when $K = C$ and $\mathcal{H}^Q_{k_i}$  when $K = H$.) We will sometimes omit the subscripts in $\psi^Q_{K,k_i}$ and simply write $\psi^Q$ when this data is clear from context. It follows from the proof of Proposition \ref{repres1} that the map $A \mapsto A_Q$ is a morphism of unitary algebras. Therefore, so is the map $\psi^Q$. It can furthermore be seen that this map is surjective with kernel $\mathcal{J} \subset \End\left(\bigoplus^{k_i} W_{i}^K\right)$. It also follows from Proposition \ref{repres1} and Remark \ref{switch} that a value $\la \in \C$ is an eigenvalue of $A$ if and only if it is an eigenvalue of $\psi^Q(A)$, where Proposition \ref{repres1} gives more detailed information about the algebraic multiplicity. \\

\noindent We now return to the more general setting where $W$ is given by the direct sum of (not necessarily isomorphic) indecomposable representations:

\begin{equation}\label{decompp0}
W := \bigoplus^{r_1} W_{1}^R \dots \bigoplus^{r_u} W_{u}^R \bigoplus^{c_1} W_{1}^C \dots \bigoplus^{c_v} W_{v}^C \bigoplus^{h_1} W_{1}^H \dots \bigoplus^{h_w} W_{w}^H\, .
\end{equation}
Recall that an element $D_A \in \End(W)$ is of block diagonal form, where the blocks correspond to isomorphic indecomposable representations. Consequently, we can apply the maps $\psi^Q$ to each of these blocks to define a map

\begin{equation}
\begin{split}
\Psi^Q  &: \{ D_A \mid A \in \End(W)\} \\ 
&\rightarrow  \, \mathcal{R}^Q_{r_1} \oplus  \dots \mathcal{R}^Q_{r_u} \oplus  \mathcal{C}^Q_{c_1} \oplus \dots \mathcal{C}^Q_{c_v} \oplus \mathcal{H}^Q_{h_1} \oplus \dots \mathcal{H}^Q_{h_w} =: \mathcal{K}^Q_W\\
\Psi^Q &:= \psi^Q_{R,r_1} \times \dots \psi^Q_{R,r_u} \times \psi^Q_{C,c_1} \times \dots \psi^Q_{C,c_v} \times \psi^Q_{H,h_1} \times \dots \psi^Q_{H,h_w} \,. 
\end{split}
\end{equation}
Note that by construction of $D_A$, the map $\Psi^Q$ is a linear bijection. Furthermore, as the elements $D_A$, $A \in \End(W)$ form a full set of representatives for the equivalence classes of $\End(W)/\mathcal{J}$, we may define $\Psi^Q$ to be a linear bijection from $\End(W)/\mathcal{J}$ to $\mathcal{K}^Q_W$. Precomposing with the natural projection from $\End(W)$ onto  $\End(W)/\mathcal{J}$, we finally obtain a linear surjective map from $\End(W)$ onto  $\mathcal{K}^Q_W$ with kernel equal to $\mathcal{J}$. We will also denote this latter map by $\Psi^Q$. The following theorem summarizes most of our results so far by listing some properties of this map.

\begin{thr}\label{main2}
The map $\Psi^Q : \End(W) \rightarrow \mathcal{K}^Q_W$ is a surjective morphism of real unitary algebras with kernel $\mathcal{J}$. Moreover, a value $\la \in \C$ is an eigenvalue of $A \in \End(W)$ if and only if it is an eigenvalue of one of the components of $\Psi^Q(A)$.
\end{thr}

\begin{proof}
It follows from the above discussion that $\Psi^Q$ is a linear surjective map with kernel $\mathcal{J}$. It is also clear that $\Psi^Q(\Id_W) = \Id \in  \mathcal{K}^Q_W$. To show that $\Psi^Q$ is a morphism of algebras, we need to show that $\Psi^Q(A)\psi^Q(B) = \Psi^Q(AB)$ for all $A,B \in \End(W)$. To this end, let us denote by $l :=u + v + w $ the number of different types of indecomposable representations appearing in the decomposition \eqref{decompp0} of $W$. Let $D \in \End(W)$ be a block diagonal endomorphism with respect to isomorphic indecomposable representations. We denote by $(\psi^Q)^l(D)$ the element of $\mathcal{K}^Q_W$ obtained by applying the map $\psi^Q$ to each of the $l$ block components of $D$. Note that by definition, $(\psi^Q)^l(D_A) = \Psi^Q(D_A)$ for all $A \in \End(W)$. We have also  seen that $\psi^Q: \End\left(\bigoplus^{k_i} W_{i}^K\right) \rightarrow \mathcal{K}^Q_{k_i} $ is a morphism of algebras, hence we have that $(\psi^Q)^l(D)(\psi^Q)^l(D') = (\psi^Q)^l(DD')$ for all block diagonal endomorphisms $D$ and $D'$. Given $A,B \in \End(W)$ it now follows that

\begin{equation}
\begin{split}
\Psi^Q(A)\Psi^Q(B) &= \Psi^Q(D_A)\Psi^Q(D_B) \\
&= (\psi^Q)^l(D_A)(\psi^Q)^l(D_B) = (\psi^Q)^l(D_AD_B) \, .
\end{split}
\end{equation}
From Proposition \ref{prods} it follows that $D_AD_B$ and $D_{AB}$ differ by an element of $\mathcal{J}$. This difference is furthermore block diagonal, as both $D_AD_B$ and $D_{AB}$ are. It follows that 

\begin{equation}
\begin{split}
 (\psi^Q)^l(D_AD_B) = (\psi^Q)^l(D_{AB}) = \Psi^Q(AB) 
\end{split}\, 
\end{equation}
and we see that indeed $\Psi^Q(A)\psi^Q(B) = \Psi^Q(AB)$. \\
To show that $A$ and $D_A$ share the same eigenvalues, we note that by Proposition \ref{eigen1} a value $\la \in \C$ is an eigenvalue of $A$ if and only if it is an eigenvalue of $D_A$. This is furthermore equivalent to $\la$ being an eigenvalue of one of the diagonal blocks of $D_A$. By Proposition \ref{repres1} and Remark \ref{switch} this is equivalent to $\la$ being an eigenvalue of one of the components of $(\psi^Q)^l(D_A) = \Psi^Q(D_A) =  \Psi^Q(A)$. This proves the claim.
\end{proof}

\begin{ex}\label{ex3}
As in Examples \ref{ex1} and \ref{ex2}, we let $W$ be given by 

\begin{equation}
W = W_1^R \oplus W_1^R \oplus W_1^C\, ,
\end{equation}
where $W_1^R$ and $W_1^C$ are indecomposable representations of real, respectively complex type. An element $A \in \End(W)$ is then given by  

\begin{equation}
A = 
\begin{pmatrix}
 a\Id + N_{1,1} &  b\Id + N_{1,2}  & A_{1,3} \\
c\Id + N_{2,1} &  d\Id + N_{2,2}  & A_{2,3} \\
A_{3,1} & A_{3,2} & e\Id + fI + N_{3,3}
\end{pmatrix}\, ,
\end{equation}
where $D_A$ is the block diagonal matrix 

\begin{equation}
D_A = 
\begin{pmatrix}
 a\Id &  b\Id   & 0 \\
c\Id  &  d\Id  & 0 \\
0& 0 & e\Id + fI 
\end{pmatrix}\, .
\end{equation} 
It follows that $\Psi^Q(A)$ is given by

\begin{equation}
\Psi^Q(A) = \begin{pmatrix} 
a & b \\
c & d
\end{pmatrix} \oplus
 \begin{pmatrix} 
e & f \\
-f & e
\end{pmatrix} \in \mathcal{R}^Q_2 \oplus \mathcal{C}^Q_1  \subset \mat(\R,2) \oplus \mat(\R,2) \,.
\end{equation}

\hspace*{\fill}$\triangle$
\end{ex}

\begin{remk}
Readers familiar with non-commutative algebra might have recognised in $\mathcal{J}$ the Jacobson radical of $\End(W)$. Consequently, $\End(W)/\mathcal{J}$ is a semisimple algebra. By Wedderburn's structure theorem such an algebra is isomorphic to the direct sum of a number of matrix algebras over a division algebra. This is precisely the result of Theorem \ref{main2}. Moreover, this latter theorem tells us that Wedderburn's isomorphism can be done in the case of $\End(W)/\mathcal{J}$ while keeping track of the eigenvalues. \hspace*{\fill}$\triangle$
\end{remk}

\subsection{The Third Reduction} \label{The Third Reduction}
Before explicitly describing those matrices with a vanishing or purely imaginary spectrum, we will now make one last reduction. This reduction will allow us to work with matrices over the complex numbers, which not only makes calculations easier, but will also allow us to use results from algebraic geometry. This latter observation will be crucial in Section \ref{Intermezzo; Some algebraic geometry}. Unlike in the last two reductions, the reduced matrix might not have all the eigenvalues that the original matrix has (even when ignoring algebraic multiplicity). However, this discrepancy will happen in a way that is not relevant to the proof of Theorem \ref{codimsteady}. \\
\noindent We define the real and complex algebras

\begin{equation}
\begin{split}
\mathcal{R}^P_n &:= \mathcal{R}^Q_n = \mat(\R,n)\\
\mathcal{C}^P_n &:= \mat(\C,n)\\
\mathcal{H}^P_n &:=\left \{  \begin{pmatrix}
  X&Y\\
  -\overline{Y} & \overline{X}
 \end{pmatrix}, \quad X, Y \in \mat(\C,n) \right \} \subset \mat(\C,2n) 
\end{split} \, .
\end{equation}
We will write $\mathcal{K}^P_{n} $ to denote one of these algebras when the type is clear. It can readily be seen that the maps

\begin{equation}
\begin{split}
C_n: \begin{pmatrix}
  A&B\\
  -B & A
 \end{pmatrix} &\mapsto A+Bi \\
H_n: \begin{pmatrix}
  A&B&C&D\\
  -B & A & -D & C\\
  -C & D & A & -B\\
  -D & -C & B & A
 \end{pmatrix} &\mapsto
 \begin{pmatrix}
  A+Bi&C+Di\\
  -C+Di & A-Bi
 \end{pmatrix} 
 \end{split}
 \end{equation}
Identify $\mathcal{C}^Q_n$ with $\mathcal{C}^P_n$ and $\mathcal{H}^Q_n$ with $\mathcal{H}^P_n$ as real unitary algebras. We furthermore note that $\mathcal{H}^P_n$ can be described as

 \begin{equation}
\mathcal{H}^P_n = \{Z \in \mat(\C,2n) \text{ such that } SZ = \overline{Z}S\} \, ,
 \end{equation}
 for 
 
 \begin{equation}
S = \begin{pmatrix}
 0 & \Id \\
 -\Id & 0 
 \end{pmatrix} \, .
 \end{equation}
Lastly, it can be seen that

\[\mat(\C,2n) =  \mathcal{H}^P_n \oplus i\mathcal{H}^P_n \, ,\]
as vector spaces over $\R$. See \cite{Rodm} for more on quaternionic matrices and their real and complex representations. Obviously $X \in \mathcal{C}^Q_n$ and $C_n(X)  \in \mathcal{C}^P_n$ cannot have the same eigenvalues, as they are matrices of different sizes. This likewise holds for the map $H_n$. Nevertheless,  the following theorem tells us that matrices with a vanishing or purely imaginary spectrum are respected by these identifications.

\begin{prop}\label{comprepeig}
Any eigenvalue of $C_n(X)  \in \mathcal{C}^P_n$ is an eigenvalue of $X \in \mathcal{C}^Q_n$. Furthermore, If $\lambda \in \C$ is an eigenvalue of $X$ then either $\lambda$ or its complex conjugate $\overline{\lambda}$ is an eigenvalue of $C_n(X)$.  \\
$\lambda \in \C$  is an eigenvalue of $X \in \mathcal{H}^Q_n$ if and only if it is an eigenvalue of $H_n(X)  \in \mathcal{H}^P_n$.

\end{prop}

\begin{proof}
\noindent We define the map

\begin{equation}
\begin{split}
C_n \oplus \overline{C}_n: \,\mathcal{C}^Q_n &\rightarrow \mat(\C,2n)\\
X = \begin{pmatrix}
  A&B\\
  -B & A
 \end{pmatrix} &\mapsto
 \begin{pmatrix}
  C_n(X)&0\\
  0 & \overline{C}_n(X)
 \end{pmatrix} =
  \begin{pmatrix}
  A+Bi&0\\
  0 & A-Bi
 \end{pmatrix}
 \end{split}
\end{equation}
between real sub-algebras of $\mat(\C,2n)$. Because $C_n(X)\cdot C_n(Y) = C_n(X\cdot Y)$ for all $X,Y \in \mathcal{C}^Q_n$, it follows that $C_n \oplus \overline{C}_n$ likewise respects matrix multiplication. Since we also see that $\tr(X) = 2\tr(A)= \tr((C_n \oplus \overline{C}_n)(X))$, we conclude from Lemma \ref{mapalg} that $X$ and $(C_n \oplus \overline{C}_n)(X)$ have the same eigenvalues, counted with algebraic multiplicity. From this it follows that any eigenvalue of $C_n(X)$ is an eigenvalue of $X$. Conversely, an eigenvalue of $X$ is either an eigenvalue of $C_n(X)$ or of $\overline{C}_n(X)$, in which case its complex conjugate is an eigenvalue of $C_n(X)$. This proves the first claim of the theorem. \\
\noindent The second claim follows likewise from applying Lemma \ref{mapalg} to the map

\begin{equation}
\begin{split}
H_n \oplus H_n: \,\mathcal{H}^Q_n &\rightarrow \mat(\C,4n)\\
 \begin{pmatrix}
  A&B&C&D\\
  -B & A & -D & C\\
  -C & D & A & -B\\
  -D & -C & B & A
 \end{pmatrix} &\mapsto 
  \begin{pmatrix}
  A+Bi & C+Di & 0 & 0 \\
  -C+Di & A-Bi & 0 & 0 \\
  0 & 0 & A+Bi&C+Di\\
  0 & 0 & -C+Di & A-Bi
  \end{pmatrix} \, .
 \end{split}
\end{equation} 
Note that both argument and image have trace equal to $4\tr(A)$.
\end{proof}

\noindent Combining the results in this section, we see that there exists a surjective morphism of unitary algebras between $\End(W)$ and the direct sum of a number of spaces of the types $\mathcal{R}^P_n$, $\mathcal{C}^P_n$ and $\mathcal{H}^P_n$. As this morphism also respects the property of having a vanishing or purely imaginary spectrum, the task of describing those elements in $\End(W)$ with this property is reduced to finding those in the three families of matrices $\mathcal{R}^P_n$, $\mathcal{C}^P_n$ and $\mathcal{H}^P_n$. In particular, we will be able to prove Theorem \ref{codimsteady} once we prove an analogous result for the matrix algebras $\mathcal{K}^P_n$. Note that reducing the problem to one in $\mathcal{K}^P_n$ means that we may essentially forget about the symmetry monoid $\Sigma$. Indeed, it already follows that the codimension of the set of endomorphisms in $\End(W)$ with a vanishing or purely imaginary spectrum depends only on the decomposition of $W$ into indecomposable representations. Apparently all other details of the action of the symmetry monoid $\Sigma$ on the space $W$ do not play a role.

\subsection{A Remark on Uniqueness}\label{A Remark on Uniqueness}

\noindent During the reductions in this section we have made a particular choice of generators $\Id, I_i, J_i, K_i \in \End(W_i)$ for any indecomposable representation $W_i$ of quaternionic type, and likewise for those of complex type. It may be insightful to see what effect a different choice of generators has on the manifolds in Theorem \ref{main}. We claim that there is no effect. More precisely, a different choice of generators will correspond to a particular isomorphism of the algebra $\mathcal{K}^P_n$. As the manifolds in $\mathcal{K}^P_n$ that we will construct in Section \ref{Geometry; Counting dimensions} will be invariant under these isomorphisms, we will conclude that the manifolds in Theorems \ref{codimsteady} and \ref{main} will be unaltered by a different choice of generators as well.\\
%
%\noindent Recall that in Section \ref{Geometric Reduction}  we have constructed the manifolds in Theorem \ref{main} from those in Theorem \ref{codimsteady} independently of any choices. More concretely, for a subspace $U \subset W$ and a manifold $M_i \in \End(U)$ we have defined a manifold $M'_i \in \End(W)$ as the set of those endomorphisms that when restricted to its generalized kernel are isomorphic to an element of $M_i$. Likewise for the center subspace case. In Section \ref{proof of main results} we will construct the manifolds in Theorem \ref{codimsteady} by first identifying $W$ with a direct sum of indecomposable representations. After that, we identify $\End(W)\mathcal{J}$ with the direct sum of a number of algebras of the type $\mathcal{K}^P_n$, using a particular choice of generators for the division algebras of the indecomposable representations as explained in this section. This step has been explained in this section.  i.e. what has been done in this section. Lastly, we will describe manifolds in $\mathcal{K}^P_n$ that contain all matrices with

\noindent To illustrate, suppose $\{ \Id,I,J,K\} \subset \End(W_i^H)$ and $\{\Id, I',J',K' \} \subset \End(W_i^H)$ both generate the quaternionic structure. Then an element $X \in \End(W_i^K)$ can be expressed in either set of generators and a nilpotent term. For convenience, let us write $\Id = I_0$, $I = I_1$, $J = I_2$ and $K = I_3$. Likewise, we write $\Id = I'_0$, $I' = I'_1$ and so forth. It follows that there are coefficients $A_{i,j}$ and  $A'_{i,j}$ such that 
\begin{align}
I_i = \sum_{j = 0}^3 A_{i,j} I'_j + N_i\, \text{ and } I'_i = \sum_{j = 0}^3 A'_{i,j} I_j + N_i'
\end{align}
for $i \in \{0, \dots 3\}$, and where the $N_i$ and $N'_i$ are nilpotent endomorphisms. If we write
\begin{align}
X = \sum_{i = 0}^3 a_i I_i + N=  \sum_{j = 0}^3 a'_j I'_j + N' \in  \End(W_i^K) \, ,
\end{align}
for $N$ and $N'$ nilpotent, then it follows that 
\begin{align} \label{linn1}
a'_j =  \sum_{i = 0}^3 A_{i,j}a_i \, \text{ and } a_i =  \sum_{j = 0}^3 A'_{j,i}a'_j \, .
\end{align}
Motivated by this, we define a map $\phi$ from the quaternions to itself given by 
\begin{align}
\phi(a_0 + a_1\tilde{i} + a_2\tilde{j} + a_3\tilde{k}) =  a'_0 + a'_1\tilde{i} + a'_2\tilde{j} + a'_3\tilde{k}\, .
\end{align}
From \eqref{linn1} we see that this map is linear and invertible. If we furthermore write 
\begin{align}
Y = \sum_{i = 0}^3 b_i I_i + N=  \sum_{j = 0}^3 b'_j I'_j + N' \in  \End(W_i^K) \, ,
\end{align}
and
\begin{align}
XY = \sum_{i = 0}^3 c_i I_i + N=  \sum_{j = 0}^3 c'_j I'_j + N' \in  \End(W_i^K) \, ,
\end{align}
Then we see that the $c_i$ are formed from the $a_i$ and $b_i$ following the rules of regular multiplication in the quaternions. The same holds for the $c'_j$, $a'_j$ and $b'_j$. In other words, we have 

\begin{align}
&(a_0 + a_1\tilde{i} + a_2\tilde{j} + a_3\tilde{k})(b_0 + b_1\tilde{i} + b_2\tilde{j} + b_3\tilde{k}) =  c_0 + c_1\tilde{i} + c_2\tilde{j} + c_3\tilde{k}\, , \nonumber \\
&(a'_0 + a'_1\tilde{i} + a'_2\tilde{j} + a'_3\tilde{k})(b'_0 + b'_1\tilde{i} + b'_2\tilde{j} + b'_3\tilde{k}) =  c'_0 + c'_1\tilde{i} + c'_2\tilde{j} + c'_3\tilde{k} \, .
\end{align}
However, what this tells us is exactly that $\phi(xy) = \phi(x)\phi(y)$ for all $x,y \in \mathbb{H}$. As it furthermore holds that $\phi(1) = 1$, we conclude that a different choice of generators for the quaternionic structure leads to a (unitary) automorphism of the quaternions applied to the entries of $\mathcal{H}_n^P$. The same holds analogously for the complex case. For the real case there is no choice left, as we always take $\Id$ as the operator whose class in $\End(W_i^R)/\nil(W_i^R)$ generates the real structure. \\

\noindent It can easily be verified that the only unitary isomorphisms of $\C$ are the identity and complex conjugation. This latter operation yields component-wise complex conjugation to the entries of $\mathcal{C}_n^P = \mat(\C,n)$. However, it will follow from Remark \ref{Cdis} in Section \ref{Geometry; Counting dimensions} that this gives the same manifolds. It follows from the \textit{Skolem-Noether theorem} (see   
Proposition 2.4.7. of \cite{Rodm}) that every automorphisms $\phi$ of $\mathbb{H}$ is an \textit{inner automorphism}. In other words, there exists an $\alpha = \alpha(\phi) \in \mathbb{H}$ such that $\phi(x) = \alpha x \alpha^{-1}$ for all $x \in \mathbb{H}$. This means that a different choice of generators for the quaternionic structure of $\End(W_i^H)/\nil(W_i^H)$ will yield a transformation in $\mathcal{H}_n^P$ that is just given by conjugation by an element in $\mathcal{H}_n^P$. As the manifolds in Theorems \ref{Hnil} and \ref{Hcen} will be conjugacy invariant, such a transformation will not effect these manifolds. 
\hspace*{\fill}$\triangle$

\newpage

\section{Proof of Main Results} \label{proof of main results}
Here we prove Theorem \ref{codimsteady}, assuming a technical result on general matrix algebras that will be proven in the next sections. We also present the theorem in the introduction as a consequence of Theorem \ref{codimsteady}.

\subsection{Proof of Theorem \ref{codimsteady}}
We will now prove Theorem \ref{codimsteady} under the assumption of Theorem \ref{summary} below. Theorem \ref{summary} itself will then be proven in the following sections as  Theorems \ref{Cnil}, \ref{Rnil}, \ref{Hnil}, \ref{Ccen}, \ref{Rcen} and \ref{Hcen}. Recall that in Section \ref{Geometric Reduction} we have already proven Theorem \ref{main}, assuming the result of Theorem \ref{codimsteady}.

\begin{thr}\label{summary}
The set of all nilpotent elements in $\mathcal{R}^P_{n}$, $\mathcal{C}^P_{n}$ and $\mathcal{H}^P_{n}$ is the disjoint union of a finite number of conjugacy invariant, embedded submanifolds of real codimension greater or equal to $n$, $2n$ and $4n$, respectively. In all three cases, there is unique manifold of this exact codimension. \\

\noindent The set of all elements in $\mathcal{R}^P_{n}$, $\mathcal{C}^P_{n}$ and $\mathcal{H}^P_{n}$ with a purely imaginary spectrum is the disjoint union of a finite number of conjugacy invariant, embedded submanifolds of real codimension greater or equal to $\ceil{n/2}$, $n$ and $n$, respectively. Again, there is in all three cases exactly one submanifold of this precise codimension.
\end{thr}
\noindent Note that the monoid $\Sigma$ does not occur in Theorem \ref{summary} anymore.
\begin{cor} \label{dimuit}
Let $W$ be the representation
\begin{equation}\label{decom3b}
W = \bigoplus^{r_1} W_{1}^R \dots \bigoplus^{r_u} W_{u}^R \bigoplus^{c_1} W_{1}^C \dots \bigoplus^{c_v} W_{v}^C \bigoplus^{h_1} W_{1}^H \dots \bigoplus^{h_w} W_{w}^H\, ,
\end{equation}
where the $W_i^K$, $K \in \{R,C,H\}$, are mutually non-isomorphic indecomposable representations. We write 
\begin{equation}
\mathcal{K}^P_W := \mathcal{R}^P_{r_1} \oplus  \dots \mathcal{R}^P_{r_u} \oplus  \mathcal{C}^P_{c_1} \oplus \dots \mathcal{C}^P_{c_v} \oplus \mathcal{H}^P_{h_1} \oplus \dots \mathcal{H}^P_{h_w} 
\end{equation}
for the corresponding representation of the endomorphism algebra of $W$. The set of nilpotent elements in $\mathcal{K}^P_W$ consists of a finite number of disjoint, conjugacy invariant manifolds of real codimension
\[K_W = r_1 + \dots + r_u + 2c_1 + \dots  + 2c_v + 4h_1 + \dots + 4h_w\] 
or higher. Exactly one of these manifolds has codimension precisely equal to this number. \\
Likewise,  the set of elements in $\mathcal{K}^P_W$ with a purely imaginary spectrum consists of a finite number of disjoint, conjugacy invariant manifolds of real codimension
\[C_W = \ceil{r_1/2} + \dots + \ceil{r_u/2} + c_1 + \dots  + c_v + h_1 + \dots + h_w\] 
and higher, with this exact number appearing only once. Here, the spectrum of an element of $\mathcal{K}^P_W$ is to be understood as the union of the spectra of the individual components, conform an interpretation as block matrices. Conjugacy invariance is with respect to invertible elements in the algebra $\mathcal{K}^P_W$.  
\end{cor}

\begin{proof}
We define submanifolds in $\mathcal{K}^P_W$ by taking all possible product sets of the submanifolds in $\mathcal{R}^P_{n}$, $\mathcal{C}^P_{n}$ and $\mathcal{H}^P_{n}$. It follows directly from Theorem \ref{summary} that these submanifolds satisfy all the conditions of the statement.  Note that if we have a finite number of manifolds $M_1, \dots M_k$ with embedded submanifolds $N_i \subset M_i$ of codimension $n_i$, then $N_1 \times \dots N_k$ is an embedded submanifold of $M_1 \times \dots M_k$ of codimension $n_1 + \dots n_k$.
\end{proof}

\begin{proof}[Proof of Theorem \ref{codimsteady}]
As we may write

\begin{equation}\label{decom3a}
W \cong \bigoplus^{r_1} W_{1}^R \dots \bigoplus^{r_u} W_{u}^R \bigoplus^{c_1} W_{1}^C \dots \bigoplus^{c_v} W_{v}^C \bigoplus^{h_1} W_{1}^H \dots \bigoplus^{h_w} W_{w}^H\, ,
\end{equation}
there exists an isomorphism between the right hand side and the left hand side of equation \eqref{decom3a}. We fix such an isomorphism, so that we may assume without loss of generality that $W$ equals the direct sum on the right of equation \eqref{decom3a}. Note that a different choice of identification (i.e. isomorphism) will not yield other submanifolds, as we will prove that these manifolds are conjugacy invariant. We furthermore write 
\begin{equation}
\mathcal{K}^Q_W = \mathcal{R}^Q_{r_1} \oplus  \dots \mathcal{R}^Q_{r_u} \oplus  \mathcal{C}^Q_{c_1} \oplus \dots \mathcal{C}^Q_{c_v} \oplus \mathcal{H}^Q_{h_1} \oplus \dots \mathcal{H}^Q_{h_w} 
\end{equation}
and 
\begin{equation}
\mathcal{K}^P_W = \mathcal{R}^P_{r_1} \oplus  \dots \mathcal{R}^P_{r_u} \oplus  \mathcal{C}^P_{c_1} \oplus \dots \mathcal{C}^P_{c_v} \oplus \mathcal{H}^P_{h_1} \oplus \dots \mathcal{H}^P_{h_w} 
\end{equation}
for the algebras corresponding to the real and complex representations of the endomorphism algebras. From Theorem \ref{main2} we know that there exists a surjective, linear map $\Psi^Q : \End(W) \rightarrow \mathcal{K}^Q_W$. Moreover, by the identification between $\mathcal{K}^P_W$ and $\mathcal{K}^Q_W$ we get a surjective, linear map $\Psi^P : \End(W) \rightarrow \mathcal{K}^P_W$. From Corollary \ref{dimuit} we see that the set of nilpotent elements in $\mathcal{K}^P_W$ consists of a finite number of disjoint, conjugacy invariant manifolds of codimension $K_W$ and higher, with this exact number appearing only once. Likewise, we see that the set of  elements in $\mathcal{K}^P_W$ with a purely imaginary spectrum consists of a finite number of disjoint, conjugacy invariant manifolds of codimension $C_W$ and higher. Again, this exact number appears only once. From Theorem \ref{main2} and Proposition \ref{comprepeig} we see that the map $\Psi^P$ preserves the property of having a vanishing or purely imaginary spectrum. Therefore, the disjoint manifolds in $\End(W)$ with a vanishing or purely imaginary spectrum will be just the inverse images under $\Psi^P $ of the manifolds in $\mathcal{K}^P_W$. Because $\Psi^P$ is a surjective linear map, the inverse images are indeed embedded submanifolds of the same codimension as their original. \\
It remains to show that the manifolds in $\End(W)$ are conjugacy invariant. However, as $\Psi^Q$ and hence $\Psi^P$ are morphisms of unitary algebras, it holds that $\Psi^P(C^{-1}) = \Psi^P(C)^{-1}$ for any invertible element $C \in \End(W)$. Therefore, if $M$  is a conjugacy invariant subset of $\mathcal{K}^P_W$ and we have $A \in (\Psi^P)^{-1}(M)$, then $\Psi^P(CAC^{-1}) = \Psi^P(C)\Psi^P(A)(\Psi^P(C))^{-1}$ for any invertible $C \in \End(W)$. From this it follows that $CAC^{-1}$ is an element of $(\Psi^P)^{-1}(M)$ as well. This finishes the proof.
\end{proof}

\subsection{Transversality} \label{Transversality}
We will now show how the technical result of Theorem \ref{main} implies the more intuitive result in the introduction. To this end, we need the following definition:

\begin{defi}
Let $M$ and $N$ be manifolds and let $A \subset N$ be a submanifold of $N$. A $C^1$ map $f: M \rightarrow N$ is called \textit{transverse} to $A$ (notation $f \pitchfork A$) if for all $x \in M$ with $f(x) \in A$ it holds that $\im(T_xf) + T_{f(x)}A = T_{f(x)}N$.
\end{defi}

\begin{remk}
Note that whenever $\dim M + \dim A < \dim N$, the condition \\ $\im(T_xf) + T_{f(x)}A = T_{f(x)}N$ cannot be satisfied. Hence, in this case the set of all $f$ transverse to $A$ is exactly the set of all $C^1$ maps $f$  such that $f(M) \cap A = \emptyset$. In other words, transverse to $A$ then means avoiding the set $A$. \hspace*{\fill}$\triangle$
\end{remk}

\noindent Next, we introduce different topologies on the set of smooth maps from $M$ to $N$. We will see that the set of maps from $M$ to $N$ transverse to a given finite set of submanifolds is dense in these topologies.

\begin{defi}
Let $s$ be a natural number and let $M$ and $N$ be two $C^s$ ma\-nifolds (in particular, $M$ and $N$ might be $C^{\infty}$ manifolds). Denote by $C^s(M,N)$ the set of all $C^s$ maps from $M$ to $N$. We will give two ways of defining a topology on $C^s(M,N)$. To this end, let $(U\subset M,\phi)$ and $(V\subset N,\psi)$ be charts on $M$ and $N$, and let $K \subset U \subset M$ be a compact subset of $M$. Let furthermore $\epsilon > 0$ be given and let $f \in C^s(M,N)$ be a map satisfying $f(K) \subset V$. We denote by

\[\mathcal{N}^s(f,(U,\phi), (V,\psi),K,\epsilon) \]
the set of all $g \in C^s(M,N)$ such that $g(K) \subset V$ and

\[ ||D^k_x(\psi f \phi^{-1}) - D^k_x(\psi g \phi^{-1})|| < \epsilon\]
\noindent for all $x \in K$ and $k \in \{0, \dots s\}$. The \textit{weak} or \textit{compact-open} topology on $C^s(M,N)$ is the smallest topology containing all sets of this form. We denote $C^s(M,N)$ with this topology by $C_W^s(M,N)$. In particular, a base for $C_W^s(M,N)$ is given by all sets of the form

\begin{equation} \label{strongweak}
\bigcup_{i \in I} \mathcal{N}^s(f_i,(U_i,\phi_i), (V_i,\psi_i),K_i,\epsilon_i) \, ,
\end{equation}
where $I$ is some finite index set. \\

\noindent We can enlarge this topology by allowing not just finite index sets $I$ in equation \eqref{strongweak}, but also sets $I$ such that the family of sets $(U_i)_{i \in I}$ is locally finite. That is, any point in $M$ has a neighbourhood intersecting $U_i$ for only finitely many $i \in I$. The topology on $C^s(M,N)$ with base the sets in \eqref{strongweak} with $(U_i)_{i \in I}$ locally finite is called the \textit{strong} or \textit{Whitney} topology. We denote $C^s(M,N)$ with this topology by $C_S^s(M,N)$. $C_W^{\infty}(M,N)$ and $C_S^{\infty}(M,N)$ are then defined as the union of the topologies on $C^{\infty}(M,N)$ induced by the inclusions in $C_W^{s}(M,N)$ and $C_S^{s}(M,N)$ respectively, for all finite $s$. See \cite{Hir} for more on these topologies. 
\end{defi}

\begin{remk}
In \cite{Hir}, a base for the strong topology is given only by sets  of the form \eqref{strongweak} with $f_i = f_j$ for all $i,j \in I$ (and with $(U_i)_{i \in I}$ locally finite). However, for 

\begin{equation} 
g \in \bigcup_{i \in I} \mathcal{N}^s(f_i,(U_i,\phi_i), (V_i,\psi_i),K_i,\epsilon_i) \, ,
\end{equation}
with 

\begin{equation} 
\max_{k \leq s} \sup_{x \in K_i} ||D^k_x(\psi_i f_i \phi_i^{-1}) - D^k_x(\psi_i g \phi_i^{-1})|| := \mu_i < \epsilon_i\, ,
\end{equation}
it is readily seen by the triangle inequality that 

\begin{align} 
g \in  &\bigcup_{i \in I} \mathcal{N}^s(g,(U_i,\phi_i), (V_i,\psi_i),K_i,\epsilon_i-\mu_i) \\
\subset &\bigcup_{i \in I} \mathcal{N}^s(f_i,(U_i,\phi_i), (V_i,\psi_i),K_i,\epsilon_i) \, .
\end{align}
Therefore, the strong topology can also be defined by a base consisting of sets of the form \eqref{strongweak} with $f_i = f_j$ for all $i,j \in I$. Note that if $M$ is compact, any locally finite family of open sets $(U_i)_{i \in I}$ is finite ($M$ can then be covered by finitely many sets, each intersecting only finitely many $U_i$). Therefore, the weak and strong topologies coincide when $M$ is compact. \hfill$\triangle$
\end{remk}
\noindent Harder to prove is that $C_W^s(M,N)$ and $C_S^s(M,N)$ are \textit{Baire spaces} for $0 \leq s \leq \infty$.  That is, the intersection of countably many dense open sets is again dense. See \cite{Hir}. We call a set \textit{residual} if it contains the intersection of countably many dense open sets. In particular, residual sets are therefore dense in $C_W^s(M,N)$ and $C_S^s(M,N)$.

\noindent The following proposition follows from Theorems 2.1 and 2.5 in Chapter 3 of \cite{Hir}. This result will be used to argue that a subrepresentation $U$ is not expected to occur as a kernel or center subspace in a $k$-parameter bifurcation if $k$ is less than $K_U$ or $C_U$, respectively. 
\begin{prop}\label{transver}
Let $M$ and $N$ be manifolds and let $N_1$ till $N_r$ be submanifolds of $N$. Let $s$ be a (non-zero) natural number or infinity. The set of $C^s$-maps from $M$ to $N$ that are transverse to all manifolds $N_i$ is residual (and therefore dense) in both $C_W^s(M,N)$ and $C_S^s(M,N)$.
\end{prop}
\noindent In contrast to Proposition \ref{transver}, the next result will be used to argue that a subrepresentation $U$ \textit{is} expected to occur as a kernel or center subspace in a $k$-parameter bifurcation if $k$ equals or exceeds $K_U$ or $C_U$, respectively. Proposition \ref{hit} is a well-known consequence of transversality, but an explicit proof can be hard to find in the literature. The method of proof that we will use is adapted from \cite{pol}. 
\begin{prop}\label{hit}
Let $A$ be an $m$-dimensional submanifold of the $n$-dimensional manifold $N$ and let $U \subset \R^k$ be a non-empty open subset. Assume furthermore that $k+m \geq n$. Then there exists a non-empty open subset of maps $f \in C^{\infty}(U,N)$ (in both the strong and weak topologies) such that $f(U) \cap A \not= \emptyset$.
\end{prop}

\begin{proof}
By choosing a submanifold chart, we may assume that $A $ equals $\R^m \times 0 := \{(x_1, \dots  x_m, 0, \dots 0)\} \subset \R^n$. Identifying $\R^k$ with 
$0 \times \R^k :=\\  \{(0, \dots 0, x_{n-k+1}, \dots  x_n)\} \subset \R^n$, we may then simply set $f$ equal to the identity to obtain a map whose image intersects $A$ (by shifting $U$ we may always assume that $U$ contains $0$). Because $U$ is open and because $k \geq n-m$, $U$ contains a closed $(n-m)$-dimensional disk centered around $0$, $D_{\mu}(0) \subset 0 \times \R^{n-m} \subset0 \times \R^k$, for some $\mu > 0$. Let $P$ denote the projection from $\R^n$ to $0 \times \R^{n-m}$. Then $P \circ f$ restricts to the identity on $D_{\mu}(0)$. We claim that whenever $g$ is a smooth map from $U$ to $\R^n$ satisfying $||g(x) - f(x)|| < \frac{1}{2} \mu$ for all $x \in D_{\mu}(0)$, then the image of $g$ intersects $\R^m \times 0$. To this end, it suffices to show that the image of $P \circ g|_{D_{\mu}(0)}$ contains $0$. \\
\noindent Suppose the converse, so that $g: U \rightarrow \R^n$ satisfies $||g(x) - f(x)|| < \frac{1}{2} \mu$ for all $x \in D_{\mu}(0)$ and so that $\tilde{g} := P \circ f|_{D_{\mu}(0)}: D_{\mu}(0) \rightarrow D_{\mu}(0)$ does not reach $0$. We will use the concept of the degree of a smooth map to arrive at a contradiction. Let $f$ be a smooth map between a compact manifold $X$ and a connected manifold $Y$. If $X$ and $Y$ have the same dimension, then the \textit{mod 2 degree} of $f$, notation $\degt(f)$, equals the number of points in $f^{-1}(y)$ modulo $2$ for any regular value $y \in Y$. It is shown in \cite{pol} (chapter 2, paragraph 4) that this is a well-defined concept. This reference also contains the following statements that we will be using:

\begin{enumerate}
\item The mod 2 degree of a map is homotopy invariant.
\item If $X$ is the boundary of a manifold $W$, and $f$ can be extended to all of $W$, then $\degt(f) = 0$.
\end{enumerate}
We also note that if $X = Y$, the degree of the identity map equals $1$, as every point has only itself as a preimage.\\
\noindent Let $S_{\mu}(0)$ denote the boundary of $D_{\mu}(0)$. As $\tilde{g}$ does not reach $0$, we may define the map 
\begin{equation}\label{extendthis}
h: x \mapsto \mu\frac{\tilde{g}(x)}{||\tilde{g}(x)||}
\end{equation}
\noindent from $S_{\mu}(0)$ to itself. As $h$ may be extended to a map from $D_{\mu}(0)$ to $S_{\mu}(0)$ (also given by equation \eqref{extendthis}) we conclude from the second statement that $\degt(h) = 0$.  We will now show that $h: S_{\mu}(0) \rightarrow S_{\mu}(0)$ is homotopic to the identity, thereby arriving at a contradiction using the first statement. To this end, note that for all $x \in D_{\mu}(0)$ we have

\begin{align}
||\tilde{g}(x) - x|| &= ||(P\circ g)(x) - (P \circ f)(x)|| \leq ||P|| \cdot ||g(x) - f(x)||   \nonumber \\
& < ||P||\cdot \frac{1}{2} \mu \leq \frac{1}{2} \mu \, .
\end{align}
\noindent It follows that for all $x \in  S_{\mu}(0)$ we have

\begin{align}
||h(x) - x|| &= \left|\left|\mu \frac{\tilde{g}(x)}{||\tilde{g}(x)||} - x\right|\right| = \left|\left| \left( \frac{\mu}{||\tilde{g}(x)||} - 1 \right) \tilde{g}(x) + \tilde{g}(x) - x \right|\right|  \nonumber \\
& \leq \left| \frac{\mu}{||\tilde{g}(x)||} - 1 \right| \cdot ||\tilde{g}(x)|| + ||\tilde{g}(x) - x || \nonumber \\
& = |\mu - ||\tilde{g}(x)|| \,| + ||\tilde{g}(x) - x || \nonumber \\
& = | \, ||x|| - ||\tilde{g}(x)||\,| + ||\tilde{g}(x) - x || \nonumber \\
& \leq ||\tilde{g}(x) - x || + ||\tilde{g}(x) - x || < 2 \cdot \frac{1}{2} \mu = \mu \, .
\end{align}
From this we see that for all $t \in [0,1]$ and $x \in  S_{\mu}(0)$ it holds that

\begin{align}
||th(x) + (1-t)x|| &= ||x - t(x - h(x))||  \geq ||x|| - ||t(x - h(x))|| \nonumber \\
&= \mu - |t| \cdot ||x - h(x)|| \geq \mu - ||x - h(x)|| \nonumber \\
&> \mu - \mu = 0   \, .
\end{align}
Hence, $||th(x) + (1-t)x||$ never vanishes and we get a well-defined homotopy between $h$ and the identity given by

\begin{equation}\label{extendthis}
 (x,t) \mapsto \mu\frac{th(x) + (1-t)x}{||th(x) + (1-t)x||} 
\end{equation}
for $(x,t) \in S_{\mu}(0) \times [0,1]$. We conclude that $1 = \degt(h) = 0$. This is of course a contradiction, and the proposition follows. 
\end{proof}

\begin{remk}\label{mainremk}
We may interpret Theorem \ref{main} and Propositions \ref{transver} and \ref{hit} as results on generic bifurcations in equivariant systems. To see why, let $\Sigma$ be a monoid acting on a finite dimensional representation space $V$ by linear maps $A_{\sigma}$, $\sigma \in \Sigma$. Let $F(x,\la)$ be a family of equivariant vector fields on $V$, indexed by a parameter $\la$  in some open set $\Omega \subset \R^k$. Suppose furthermore that $F(x(\la),\la) = 0$ for some smooth curve of values $x(\la)$. If the zeroes $x(\la)$ are all invariant, that is if $A_{\sigma}X(\la) = x(\la)$ for all $\sigma \in \Sigma$ and $\la \in \Omega$, then linearization gives a map $f$ from $\Omega$ to $\End(V)$ given by $f(\la) = D_x(x(\la), \la)$. For a bifurcation to appear along $x(\la)$, one of the endomorphisms $D_x(x(\la), \la)$ has to have a non-trivial kernel or center subspace. Hence, we are interested in maps from the manifold $\Omega$ into the manifold $\End(V)$ that have a non-trivial kernel or center subspace. Note that $f$ can be perturbed into any other map from  $\Omega$ to $\End(V)$, by adding the equivariant map $B(\la)(x - x(\la))$ to $F(x(\la),\la)$ for any map $B$ from $\Omega$ to $\End(V)$. \\
As a result, if $U$ is any (complementable) invariant subspace of $V$ with $k < K_U$, then it follows from Proposition \ref{transver} that $U$ will not 'robustly' occur as the generalized kernel of any of the linear maps $f(\la) = D_x(x(\la), \la)$. More precisely, if $U$ does appear as the generalized kernel of any of the maps $f(\la)$, then after an arbitrarily small perturbation of $f(\la)$ it may not anymore. Likewise, one does not expect $U$ to appear as the center subspace of any of the maps $f(\la)$ if $k < C_U$. \\
If, however, it holds that $k \geq K_U$, then by Proposition \ref{hit} there is an open set of maps $f: \Omega \rightarrow \End(V)$ for which $U$ appears as the generalized kernel of one of the maps $f(\la)$. Similarly for $U$ as the center subspace if $k \geq C_U$. We summarize these results by saying that a generic $k$-parameter steady state bifurcation occurs along those $U$ for which $K_U \leq k$ and that generically a center manifold is a graph over those $U$ for which $C_U \leq k$. In particular, a generic $1$-parameter steady state bifurcation appears along exactly one indecomposable representation of real type, as this is the only way it can hold that $K_U = 1$. Likewise, a generic $1$-parameter Hopf bifurcation appears along two isomorphic indecomposable representations of real type or along one indecomposable representation of either complex or quaternionic type. Note that $C_U=1$ only when $U$ is indecomposable of any type or when $U$ is the direct sum of two isomorphic indecomposable representations of real type. Moreover, if $U$ is indecomposable of real type then elements of $\End(U)$ have only one, real eigenvalue (of algebraic multiplicity $\dim(U)$). This excludes the standard Hopf bifurcation scenario whereby two (separate) conjugate eigenvalues pass through the imaginary axis.  \hspace*{\fill}$\triangle$
\end{remk}
%\noindent See Theorem 2.5 in chapter 3 of \cite{Hir}. This source also contains a description of the strong topology. Note that two manifolds are transverse when their tangent spaces sum up to the tangent space of the ambient manifold $N$ for all points of intersection. In particular, if the sum of the dimension of $N_i$ and that of (the image of) $M$ is less than the dimension of $N$, then a transverse intersection is equivalent to $N_i$ and the image of $M$ being disjoint. In our setting, let $N$ be the space $\End{W}$, the $N_i$ be the manifolds of Theorem \ref{main} and $M$ be an open subset of $\R^k$, corresponding to a $k$-parameter bifurcation. It follows that if $k$ is less than the codimensions of the manifolds in Theorem \ref{main} corresponding to a subspace $U$, then there is a residual set of maps missing these manifolds. In other words, if $U$ appears as a generalized subspace, or as a center subspace (depending on the set-up), then after a small perturbation it will not anymore. \\

\section{Intermezzo; Some Algebraic Geometry} \label{Intermezzo; Some algebraic geometry}
In this section and the next we will prove Theorem \ref{summary} by identifying those elements in  $\mathcal{R}^P_n$, $ \mathcal{C}^P_n$ and $\mathcal{H}^P_n$ with a vanishing or purely imaginary spectrum. Our first step is proving the technical result of Theorem \ref{embedd} below. It should be noted that the results in this section are known to experts, but hard to find in the literature. 

\begin{thr}\label{embedd}
Conjugacy classes in $\mat(\C,n)$ are embedded manifolds. That is, given $X \in \mat(\C,n)$, the set $\{A^{-1}XA \mid A \in \Gl(\C,n)\}$ is an embedded submanifold of $\mat(\C,n)$.
\end{thr}
\noindent In order to prove Theorem \ref{embedd}, we will use a theorem from \cite{alggroup}. To this end, we will have to introduce some basic algebraic geometry. We begin by defining an alternative topology on $\C^n$.

\begin{defi}
The \textit{Zariski topology} on $\C^n$ is defined by stating that its closed sets are given by the common zeroes of a set of polynomials. More precisely, let $\C[X_1, \dots, X_n]$ denote the set of polynomials in $n$-variables and coefficients in $\C$. A subset $Z \subset \C^n$ is closed in the Zariski topology (or simply \textit{Zariski-closed}) when it can be written as

\begin{align}
Z = Z(S) = \{x\in \C^n \mid p(x) = 0 \, \forall \, p \in S\}
\end{align}
for some set of polynomials $S \subset \C[X_1, \dots, X_n]$. 
\end{defi}
\noindent Note that Zariski-closed (or Zariski-open) sets are also closed (respectively open) in the usual, Euclidean topology on $\C^n$. The following well-known result states that Zariski-closed sets can be described as cut out by only finitely many polynomials. The proof can be found in for example \cite{abalg}, paragraph 9.6.

\begin{thr}[Hilbert's basis theorem] \label{Hilbasis}
Every ideal in $ \C[X_1, \dots, X_n]$ is generated by finitely many elements. Consequently, for any Zariski-closed set $Z$ there exist finitely many polynomials $p_1, \dots p_k$ such that
 
\begin{align}
Z = Z(\{p_1, \dots, p_k\}) = \{x\in \C^n \mid p_1(x) = \dots =  p_k(x) = 0 \} \, .
\end{align}
\end{thr}
\noindent We will shortly comment on how the second part of Theorem \ref{Hilbasis} follows from the first. One easily verifies that it does not matter whether a Zariski-closed set is defined by the vanishing of a set of polynomials $S$, or by the vanishing of the ideal generated by this set, $\langle S \rangle$. In other words, we have $Z(S) = Z(\langle S \rangle)$. As the ideal $\langle S \rangle$ is also generated by finitely many elements $p_1, \dots p_k$, we find $Z(S) = Z(\langle p_1, \dots, p_k \rangle) = Z(\{p_1, \dots, p_k\})$.

\noindent We continue to introduce some terminology from \cite{alggroup}. An example of an affine algebraic variety over $\C$ is a Zariski-closed set $Z \subset \C^n$ together with the algebra of functions $\C[X_1, \dots, X_n]|_{Z}$. More generally, we have the following definition.

\begin{defi}
An \textit{(abstract) affine algebraic variety} is a set $V$ with an algebra $A$ of functions from $V$ to $\C$, such that the following property holds. There exists a bijection $\iota$ from $V$ to a Zariski-closed set $Z \subset \C^n$ for which $\iota^*: \C[X_1, \dots, X_n]|_{Z} \rightarrow A$, $f \mapsto f \circ \iota$ is an isomorphism of algebras.  
\end{defi}
\noindent The Zariski topology can more generally be defined on any affine algebraic variety $(V,A)$, by stating that a subset of $V$ is closed when it is given by the vanishing of some elements of $A$. Moreover, given two affine varieties $(V,A)$ and $(W,B)$, one can define the product variety $(V \times W, A\otimes B)$. This is readily seen to be an affine algebraic variety in its own right. Here, $A\otimes B$ is interpreted as an algebra of functions from $V \times W$ to $\C$ by setting $(a \otimes b)(v,w) := a(v)\cdot b(w)$ for $a \otimes b \in A\otimes B $ and $(v,w) \in V \times W$. To complete the description of the category of affine algebraic varieties, we have:

\begin{defi}
A morphism between affine varieties $(V,A)$ and $(W,B)$ is a map $f$ from $V$ to $W$ such that $f^*b := b \circ f$ is an element of $A$ for all $b \in B$.
\end{defi}
\noindent Finally, we need the notions of an affine algebraic group and the action of an affine algebraic group on an affine algebraic variety. These will be the appropriate generalizations of the action of $\Gl(\C,n)$ on $\mat(\C,n)$.

\begin{defi}\hspace*{\fill}
\begin{enumerate}
\item \textit{An affine algebraic group} is an affine algebraic variety $(G,A)$ that is also a group for which the operations of multiplication \\
$ m(\bullet, \bullet): (G\times G, A \otimes A) \rightarrow   (G,A)$ and taking inverses $ \bullet^{-1}:  (G,A) \rightarrow  (G,A)$ are morphisms of affine algebraic varieties. 
\item An \textit{algebraic action} of an affine algebraic group $(G,A)$  on an affine algebraic variety  $(V,B)$ is defined as an action of the group $G$ on the set $V$ such that the defining map $\xi: (G \times V, A \otimes B) \rightarrow (V,B)$, $\xi(g,v) = g \cdot v$ is a morphism of affine algebraic varieties.
\end{enumerate}
\end{defi}

\begin{ex}\label{glisalg}
It can be shown that 
\[ (\Gl(\C,n), \C[\dots, X_{i,j}, \dots, \Det^{-1}]|_{\Gl(\C,n)})\] 
\noindent with the usual matrix multiplication is an affine algebraic group. Here, 
\[\C[\dots, X_{i,j}, \dots, \Det^{-1}]|_{\Gl(\C,n)}\] 
\noindent is the algebra of functions on $\Gl(\C,n)$ generated by the matrix coefficients $X_{i,j}$ for $1 \leq i,j \leq n$ and $1$ over the determinant, $\Det^{-1}(X) := 1/\Det(X)$ for $X \in \Gl(\C,n)$. See \cite{alggroup}.
\end{ex}
\noindent As an example of an algebraic action, we have:

\begin{lem}\label{algebraicaction}
We give $\mat(\C,n)$ the structure of an affine algebraic variety by identifying it with the (Zariski-closed) set $\C^{n\times n}$, together with the natural algebra of polynomials in $n^2$ variables. The conjugacy action of $\Gl(\C,n)$ (with the affine algebraic structure defined in Example \ref{glisalg}) on the affine algebraic variety $\mat(\C,n)$ is an example of an algebraic action.
\end{lem}

\begin{proof}
It is clear that conjugation defines an action of $\Gl(\C,n)$ on $\mat(\C,n)$. Therefore, it remains to show that the defining map of this action,

\begin{align}
\xi: &\Gl(\C,n) \times \mat(\C,n) \rightarrow \mat(\C,n) \\ \nonumber
&(C,X) \mapsto C^{-1}XC
\end{align}
\noindent is a morphism of affine algebraic varieties. In other words, given a polynomial $p \in \C[\dots, X_{i,j}, \dots]$, we need to show that the map $(C,X) \mapsto p(C^{-1}XC)$ is an element of 
\[\C[\dots, X_{i,j}, \dots, \Det^{-1}]|_{\Gl(\C,n)} \otimes \C[\dots, X_{i,j}, \dots] \, .\] 
\noindent It suffices to verify this for the generators $X_{i,j}$, as pre-composition by $\xi$ is linear and multiplicative on function space.  Therefore,  the statement of the lemma is true whenever the map $(C,X) \mapsto (C^{-1}XC)_{i,j}$ is an element of 
\[\C[\dots, X_{i,j}, \dots, \Det^{-1}]|_{\Gl(\C,n)} \otimes \C[\dots, X_{i,j}, \dots]\] 
\noindent for all $i,j$. Writing it out, we get

\begin{equation}
(C^{-1}XC)_{i,j} = \sum_{k,l}(C^{-1})_{i,k}X_{k,l}C_{l,j} = \sum_{k,l}(C^{-1})_{i,k}C_{l,j}X _{k,l} \, .
\end{equation}

\noindent Now, $X \mapsto X_{k,l}$ is clearly an element of $\C[\dots, X_{i,j}, \dots]$. Likewise, $C \mapsto C_{l,j}$ is an element of 
\[\C[\dots, X_{i,j}, \dots, \Det^{-1}]|_{\Gl(\C,n)}\, .\]
\noindent Furthermore, as taking the inverse is a morphism from the affine algebraic group 
\[(\Gl(\C,n), \C[\dots, X_{i,j}, \dots, \Det^{-1}]|_{\Gl(\C,n)})\] 
\noindent to itself, and as the map $C \mapsto C_{i,k}$ is an element of 
\[\C[\dots, X_{i,j}, \dots, \Det^{-1}]|_{\Gl(\C,n)}\, ,\]
\noindent so is the map $C \mapsto (C^{-1})_{i,k}$. (This fact can also be directly verified using Cramer's rule for inverses.) From this we conclude that the map $(C,X) \mapsto (C^{-1}XC)_{i,j}$ is indeed an element of the algebra 
\[\C[\dots, X_{i,j}, \dots, \Det^{-1}]|_{\Gl(\C,n)} \otimes \C[\dots, X_{i,j}, \dots] \, .\] 
\noindent This proves the lemma.
\end{proof}
\noindent The following theorem is key in proving Theorem \ref{embedd}. Its proof can be found in \cite{alggroup}.

\begin{thr}\label{orbitsmagic}
Let $(G,A)$ be a connected (in the Zariski topology), affine algebraic group acting algebraically on an affine algebraic variety $(V,B)$. Then every orbit is Zariski-open in its Zariski-closure.  
\end{thr}

\noindent Next, we need the concept of a non-singular point in an algebraic variety. Heuristically, this means that around this point, the variety looks like a submanifold of $\C^n$. The following definition is adapted from \cite{hart}.

\begin{defi}
A Zariski-closed subset $Z \subset \C^n$ is called \textit{irreducible} if it cannot be written as the union of two Zariski-closed, strict subsets of $Z$. Such a set can be given the notion of a (finite) dimension. This can be done algebraically by looking at an ideal in $\C[X_1, \dots, X_n]$ defining $Z$, or geometrically by looking at chains of irreducible varieties that are contained in $Z$. See \cite{hart} for more on these concepts. Let $x$ be a point in the Zariski-closed, irreducible subset $Z = Z(\{p_1, \dots p_s\}) \subset \C^n$ of dimension $r$. Writing $P := (p_1, \dots p_s):\C^n \rightarrow \C^s$, we say that $x$ is \textit{non-singular} if the rank of the Jacobian $DP(x)$ equals $n-r$. $x$ is called \textit{singular} if it is not non-singular. It can be shown that the definition of a non-singular point is independent of the set of functions $\{p_1, \dots p_s\}$ used to define $Z$.
\end{defi}
\noindent The following lemma shows that 'most points' are non-singular. This lemma is Theorem 5.3 in Chapter 1 of \cite{hart}, adapted to our setting.

\begin{lem}\label{closesingu}
Let $Z$ be a Zariski-closed, irreducible subset of $\C^n$. The set of singular points of $Z$ is a Zariski-closed, strict subset of $Z$.
\end{lem}
\noindent The following lemma reinforces our remark that non-singular points should be thought of as those points where the variety looks like a differentiable manifold. Its proof can be found in \cite{oss}.

\begin{lem}\label{locmani}
Let $Z  = Z(\{p_1, \dots p_s\})$ be a Zariski-closed, irreducible subset of $\C^n$ of dimension $r$, and let $x$ be a non-singular point of $Z$. Suppose that, without loss of generality,  the derivatives of $p_1$ till $p_{n-r}$ are linearly independent at $x$. In other words, defining $Q := (p_1, \dots p_{n-r})$, the Jacobian $DQ(x)$ has full rank. Then there exists a Zariski-open set $U \subset \C^n$ containing $x$ such that $Z \cap U = Z(\{p_1, \dots p_{n-r}\}) \cap U$. In other words, $Z$ can locally be seen as cut out by just the polynomials $p_1$ till $p_{n-r}$.
\end{lem}

\noindent It should be noted that \cite{hart} uses a slightly different definition of affine algebraic variety, calling a Zariski-closed subset of $\C^n$ an affine algebraic variety only when it is irreducible. Dropping the irreducibility condition, \cite{hart} speaks of an \textit{algebraic set}. This will not be problem though, as we will show using the lemma below that the Zariski-closure of a conjugacy orbit is in fact irreducible. Note that the notion of irreducibility can be defined on any affine algebraic variety using its Zariski topology. The following lemma can be deduced from Remark 3 in Chapter 1 of \cite{alggroup}. 

\begin{lem}\label{connec}
An affine algebraic group is irreducible if and only if it is connected in the Zariski topology.
\end{lem}

\noindent We are now in a position to prove Theorem \ref{embedd}.

\begin{proof}[Proof of Theorem \ref{embedd}]
We have shown in Lemma \ref{algebraicaction} that conjugation is an algebraic action of the affine algebraic group $\Gl(\C,n)$ on the affine algebraic variety $\mat(\C,n)$. From the fact that the functions in 
\[\C[\dots, X_{i,j}, \dots, \Det^{-1}]|_{\Gl(\C,n)}\]
are continuous in the Euclidean topologies on $\Gl(\C,n)$ and $\C$, we conclude that Zariski-closed and Zariski-open sets in $\Gl(\C,n)$ are closed, respectively open in the Euclidean topology on $\Gl(\C,n)$ as well. As a result, we may conclude that $\Gl(\C,n)$ is connected in the Zariski topology from the fact that it is connected in the Euclidean topology. See also Remark 4 in Chapter 1 of \cite{alggroup}. It therefore follows from Theorem \ref{orbitsmagic} that the conjugacy orbit of any fixed element in $\mat(\C,n)$ is Zariski-open in its Zariski-closure. Let us denote this single orbit by $\mathcal{O}_X \subset  \mat(\C,n)$, for $X \in  \mat(\C,n)$, and its Zariski-closure by $\overline{\mathcal{O}}_X \subset  \mat(\C,n)$. By Lemma \ref{connec}, the algebraic group $\Gl(\C,n)$ is an irreducible variety. Fixing $X$, we get a map $\xi_X := \xi(\bullet,X): \Gl(\C,n) \rightarrow \mat(\C,n)$ whose image is $\mathcal{O}_X$. From the definition of a product variety, one sees that $\xi_X$ is a morphism of varieties as well. Furthermore, from the definitions of a morphism and of the Zariski-topology, one easily verifies that a morphism of varieties is continuous in the Zariski-topology.\\
\noindent Now suppose we have $\overline{\mathcal{O}}_X = A \cup B$ for two Zariski-closed sets $A,B \subset \mat(\C,n)$. It follows that $\Gl(\C,n) = \xi_X^{-1}(A) \cup \xi_X^{-1}(B)$, with  $\xi_X^{-1}(A)$ and $\xi_X^{-1}(B)$ Zariski-closed. As $\Gl(\C,n)$ is irreducible, it follows that (without loss of generality) $\Gl(\C,n) = \xi_X^{-1}(A)$. We conclude that $\mathcal{O}_X = \ \xi_X(\Gl(\C,n)) \subset A$, and therefore $\overline{\mathcal{O}}_X \subset A$. This proves that $\overline{\mathcal{O}}_X$ is an irreducible set as well. \\
\noindent By Lemma \ref{closesingu} we conclude that the set of singular points of $\overline{\mathcal{O}}_X$ is Zariski-closed. As it is also a strict subset of $\overline{\mathcal{O}}_X$, we conclude that ${\mathcal{O}}_X$ cannot be completely contained in this set of singular points (otherwise the closure of ${\mathcal{O}}_X$ would be smaller). We conclude that there exists a point $Y \in {\mathcal{O}}_X$ that is a non-singular point of the irreducible variety $\overline{\mathcal{O}}_X$. Now, from Lemma \ref{locmani} we see that there exist a (Zariski)-open set $U \subset \mat(\C,n)$ containing $Y$ and polynomials $p_1$ till  $p_{n-r}$ such that  $\overline{\mathcal{O}}_X \cap U = Z(\{p_1, \dots p_{n-r}\}) \cap U$. Let $V$ furthermore be a (Zariski)-open subset of $\mat(\C,n)$ such that $\overline{\mathcal{O}}_X \cap V = \mathcal{O}_X$. Note that the Zariski-topology on a Zariski-closed subset $W \subset \C^m$ coincides with the topology induced by the Zariski-topology on $\C^m$, as the algebra of functions on $W$ is obtained from that on $\C^m$ by restriction. It follows that
\begin{equation}
\mathcal{O}_X \cap U = \overline{\mathcal{O}}_X \cap V \cap U = Z(\{p_1, \dots p_{n-r}\}) \cap U \cap V = P^{-1}(0)\cap U \cap V \, ,
\end{equation}
for $P:= (p_1, \dots p_{n-r})$. As $DP(Y)$ has maximal rank, we conclude from the constant rank theorem that $\mathcal{O}_X$ is locally around $Y$ an embedded submanifold. In particular, if $U'$ is an open set containing $Y$ such that $\mathcal{O}_X \cap U'$ is a submanifold, then for any $C \in Gl(\C,n)$, $C\mathcal{O}_XC^{-1} \cap CU'C^{-1} = \mathcal{O}_X \cap CU'C^{-1}$ is a submanifold, with  $CU'C^{-1}$ containing $CYC^{-1}$. We conclude that $\mathcal{O}_X$ is globally a submanifold of $\mat(\C,n)$. This proves the theorem.
\end{proof}

\section{Geometry; Counting Dimensions} \label{Geometry; Counting dimensions}
Next, we will determine the dimensions of the set of those elements in $\mathcal{R}^P_n$, $ \mathcal{C}^P_n$ and $\mathcal{H}^P_n$ with a vanishing and purely imaginary spectrum. In order to do so, we will need the results of Lemma \ref{dimnil0} below. These are well known, but included for completeness. Furthermore, we will amply use the techniques behind these results, as well as their generalizations later on.

\begin{defi}
For $n \in \N$ and $\la \in \C$ we introduce the Jordan block matrices $B_n(\la) \in \mat(\C,n)$ given by

\begin{equation}
B_n(\la) := \begin{pmatrix}
  \la &1& 0 & \dots & 0 & 0\\
  0 & \la & 1 & 0 & \dots & 0\\
   & & \ddots & \ddots & & \\
   & & \ddots & \ddots & & \\
  0 & \dots & 0 & 0 & \la & 1 \\
  0 & 0 & \dots & 0 & 0 & \la
 \end{pmatrix} \, .
\end{equation}
More generally, let $p = (s_1, \dots s_k)$ with $s_1 \geq \dots \geq s_k \geq 1$ and  $s_1 + \dots + s_k = n$ be a partition of $n$, We define the block-diagonal matrix $B_{n,p}(\la) \in \mat(\C,n)$ by 

\begin{equation}\label{blokk}
B_{n,p}(\la) := \begin{pmatrix}
  B_{s_1}(\la) &0& \dots & 0 \\
  0 & B_{s_2}(\la) & \dots & 0\\
   &  & \ddots &  \\
  0 & \dots  & 0 & B_{s_k}(\la)
 \end{pmatrix} \, .
\end{equation}
\end{defi}

\begin{lem}\label{dimnil0}
For fixed $n \in \N$ and partition $p$, the vector spaces

\[\im(\mathcal{L}_{B_{n,p}(\la), B_{n,p}(\la)}) = \im([ B_{n,p}(\la), \bullet ])\] 
\noindent and
\[\ker(\mathcal{L}_{B_{n,p}(\la), B_{n,p}(\la)}) = \ker([ B_{n,p}(\la), \bullet ])\] 
\noindent are independent of $\la \in \C$. Here, $[ B_{n,p}(\la), \bullet ]$ denotes the commutator operator with $B_{n,p}(\la)$.

\noindent For the trivial partition $p = (n)$ (that is, when $B_{n,p}(\la) = B_{n}(\la)$), the complex dimension of the image is equal to $n^2 - n$. For all other partitions this dimension is strictly smaller than $n^2 - n$.
\end{lem}

\begin{proof}
From the definition of $B_{n,p}(\la)$ we see that $B_{n,p}(\la) = B_{n,p}(0) + \la\Id$. Consequently, for all $X \in \mat(\C,n)$ it holds that

\begin{align}
[B_{n,p}(\la), X] = [B_{n,p}(0),X] + [\la\Id,X] = [B_{n,p}(0),X] \, ,
\end{align}
as every matrix commutes with $\la\Id$. We therefore conclude that $[B_{n,p}(\la), \bullet] = [B_{n,p}(0), \bullet]$ as operators. In particular, it holds that their images agree and that their kernels agree.

\noindent To determine the dimension of $\im (\mathcal{L}_{B_{n}(0), B_{n}(0)})$, let $\{e_i\}_{i=1}^n$ be the standard basis of $\C^n$. By the definition of $B_n(0)$ it holds that $B_n(0)e_1 = 0$ and $B_n(0)e_i= e_{i-1}$ for $2 \leq i \leq n$. Likewise, we see that  $B_n(0)^Te_n = 0$ and $B_n(0)^Te_i= e_{i+1}$ for $1 \leq i \leq n-1$. We will set $f_i := e_{n+1-i}$, so that $B_n(0)^Tf_1 = 0$ and $B_n(0)^Tf_i= f_{i-1}$ for $2 \leq i \leq n$. As in the proof of Lemma \ref{cominv}, the set $\{e_if_j^T\}_{i,j}$ forms a basis of $\mat(\C, n)$, and we have

\begin{align}\label{nil-ac}
\mathcal{L}_{B_{n}(0), B_{n}(0)}(e_if_j^T) &= e_if_{j-1}^T - e_{i-1}f_j^T  &\mathcal{L}_{B_{n}(0), B_{n}(0)}(e_1f_j^T) &= e_1f_{j-1}^T \\ \nonumber
\mathcal{L}_{B_{n}(0), B_{n}(0)}(e_if_1^T) &= - e_{i-1}f_1^T 	&\mathcal{L}_{B_{n}(0), B_{n}(0)}(e_1f_1^T) &= 0 \, ,\nonumber
\end{align}
for $2 \leq i,j \leq n$. Next, we define the spaces

\[V_m := \spn \{ e_if_j^T \mid i+j = m\}\, .\]
From the equations \eqref{nil-ac} it follows that $\mathcal{L}_{B_{n}(0), B_{n}(0)}$ restricts to a map from $V_m$ to $V_{m-1}$ for every $m$. We claim that for $m>n+1$, the map $\mathcal{L}_{B_{n}(0), B_{n}(0)}|_{V_m}$ has vanishing kernel, whereas for $m \leq n+1$ the map $\mathcal{L}_{B_{n}(0), B_{n}(0)}|_{V_m}$ has a one-dimensional kernel. Indeed, setting $[i,j] := e_if_j^T$ we find for $m>n+1$:

\begin{align}
&\mathcal{L}_{B_{n}(0), B_{n}(0)}\left( \sum_{i=m-n}^n a_{i}[i, m-i]\right)  \\ \nonumber
&=\sum_{i=m-n}^n a_{i}[i,m-i-1] - \sum_{i=m-n}^n a_{i}[i-1,m-i]  \\ \nonumber
&=\sum_{i=m-n}^n a_{i}[i,m-i-1] - \sum_{j=m-n-1}^{n-1} a_{j+1}[j,m-j-1]  \\ \nonumber
&= - a_{m-n}[m-n-1,n] + \sum_{i=m-n}^{n-1} (a_{i}-a_{i+1})[i, m-i-1] + a_{n}[n, m-n-1] \nonumber \, .
\end{align}
One readily verifies that this map has vanishing kernel. For $m \leq n+1$ we find

\begin{align}
&\mathcal{L}_{B_{n}(0), B_{n}(0)}\left( \sum_{i=1}^{m-1} a_{i}[i, m-i]\right)  \\ \nonumber
&=\sum_{i=1}^{m-2} a_{i}[i,m-i-1] - \sum_{i=2}^{m-1} a_{i}[i-1,m-i]  \\ \nonumber
&=\sum_{i=1}^{m-2} a_{i}[i,m-i-1] - \sum_{j=1}^{m-2} a_{j+1}[j,m-j-1]  \\ \nonumber
&=\sum_{i=1}^{m-2} (a_{i}-a_{i+1})[i, m-i-1]  \nonumber \, ,
\end{align}
which has a one-dimensional kernel given by $a_1 = \dots = a_{m-1} $. From the fact that 

\begin{equation}
\mat(\C, n) = \bigoplus_{m=2}^{2n} V_m \,
\end{equation}
it follows that the kernel of $\mathcal{L}_{B_{n}(0), B_{n}(0)}$ has dimension $n$. Therefore, its image has dimension equal to $n^2-n$.

\noindent Next, we prove that $\dim \im(\mathcal{L}_{B_{n,p}(0), B_{n,p}(0)})$ is strictly smaller than $n^2-n$, whenever $p = (s_1, \dots s_k) \not= (n)$. To this end, let us denote a matrix \\ $X \in \mat(\C,n)$ as $X = (X_{i,j})$, $1 \leq i,j \leq k$, with respect to the block decomposition of the matrix $\eqref{blokk}$. We see that $[B_{n,p}(0), X]_{i,i} = B_{s_i}(0)X_{i,i} - X_{i,i}B_{s_i}(0)$ for all $1 \leq i\leq k$. From this and the second part of the theorem, it follows that the image of the map given by the $i$'th diagonal block of $[B_{n,p}(0), \bullet]$ has codimension $s_i$. Together, the image of all the diagonal blocks therefore has codimension $s_1 + \dots + s_k = n$. However, for \\ $i \not= j$ we see that $[B_{n,p}(0), X]_{i,j} = B_{s_i}(0)X_{i,j} - X_{i,j}B_{s_j}(0) = -\mathcal{L}_{B_{s_j}(0),B_{s_i}(0)}(X_{i,j})$. By Lemma \ref{cominv} the image of this map has a strictly positive codimension. From this it follows that the dimension of the image of $\mathcal{L}_{B_{n,p}(0), B_{n,p}(0)}$ is strictly less than $n^2-n$, thereby proving the theorem.
\end{proof}

\subsection{The Case $\mathcal{C}^P_n$} \label{The case Cpn}
We will start the proof of Theorem \ref{summary} with the algebra $\mathcal{C}^P_n = \mat(\C,n)$. Our goal is to determine the dimension of the set of nilpotent matrices in $\mathcal{C}^P_n$ and of those matrices with a purely imaginary spectrum.

\begin{thr}\label{Cnil}
The set of nilpotent matrices in $\mathcal{C}^P_n$ is composed of finitely many conjugacy invariant embedded manifolds of complex dimension $n^2 - n$ or lower. Exactly one of these manifolds has dimension equal to $n^2 - n$.
\end{thr}

\begin{proof}
Every nilpotent matrix is conjugate to exactly one of the matrices $B_{n,p}(0)$. Therefore, the manifolds will be the conjugacy orbits 

\[\mathcal{O}_{B_{n,p}(0)}:= \{A^{-1}B_{n,p}(0)A \mid A \in \Gl(\C,n)\} \, . \]
We have proven in the last section that these are indeed embedded submanifolds of $\mat(\C,n)$, so it remains to determine their dimensions. To this end, we note that every set $\mathcal{O}_{B_{n,p}(0)}$ is equal to the image of the smooth map 
\begin{align}
\Psi_{B_{n,p}(0)}: \, &Gl(n,\C) \rightarrow \mathcal{C}^P_n\\ \nonumber
&A \mapsto A^{-1} B_{n,p}(0) A \, .\nonumber
\end{align}
Its derivative at $A \in \Gl(n,\C)$ in the direction of $V \in \mat(\C,n)$ can be evaluated relatively easily by precomposing with the curve $t \mapsto \Exp(tVA^{-1})A$, which goes through $A$ with velocity $V$. We get 
\begin{align}
T_A\Psi_{B_{n,p}(0)}(V) &=  \left. \frac{d}{dt} \right|_{t=0} \Psi_{B_{n,p}(0)}(\Exp(tVA^{-1})A)  \\ \nonumber
&=\left. \frac{d}{dt} \right|_{t=0}  A^{-1}\Exp(-tVA^{-1})B_{n,p}(0)\Exp(tVA^{-1})A  \\ \nonumber
&=A^{-1}(B_{n,p}(0)VA^{-1} - VA^{-1}B_{n,p}(0))A  \\ \nonumber
&= A^{-1}[B_{n,p}(0),VA^{-1}]A \, . \nonumber
\end{align}
By varying $V$, we see that $\im(T_A\Psi_{B_{n,p}(0)}) = \im(A^{-1} [B_{n,p}(0),\bullet] A)$. Consequently, we have that 
\[\dim \im(T_A\Psi_{B_{n,p}(0)}) =\dim \im(A^{-1} [B_{n,p}(0),\bullet] A) = \dim \im( [B_{n,p}(0),\bullet]) \, ,\] which is independent of $A$. Hence, the map $\Psi$, seen as a map from $\Gl(\C,n)$ to the manifold $\mathcal{O}_{B_{n,p}(0)}$, is a surjective, smooth map whose derivative has constant rank. Moreover, it is known that any smooth, surjective map of constant rank between two manifolds is a submersion (see \cite{lee}). Hence, the dimension of $\mathcal{O}_{B_{n,p}(0)}$ is equal to the dimension of $ \im( [B_{n,p}(0),\bullet])$, which  is equal to $n^2-n$ for $p = (n)$ and strictly less in all other cases. This proves the theorem.
\end{proof}

\begin{thr}\label{Ccen}
The set of matrices in $\mathcal{C}^P_n$ with purely imaginary spectrum is composed of finitely many conjugacy invariant embedded manifolds of real dimension $2n^2 - n$ or lower. Exactly one of these manifolds has dimension equal to $2n^2 - n$.
\end{thr}
\noindent Before we can prove Theorem \ref{Ccen}, we need another lemma. It provides a special local chart for any conjugacy orbit. To simplify notation, we will introduce yet another way of denoting a Lie bracket, namely $\ad_X(Y) := [X,Y]$ for $X,Y \in \mat(\C, n)$ .

\begin{lem}\label{neigh}
For $X \in \mat(\C, n)$, let $U, V \subset \mat(\C, n)$ be two complex linear spaces such that

\[U \oplus \im \ad_X = V \oplus \ker \ad_X = \mat(\C,n)\, .\]
Then there exist open neighborhoods $W_U \subset U$ and $W_V \subset V$, both containing $0$, and $W \subset \mat(\C,n)$ containing $X$ such that the map

\begin{align}
\mathcal{X}: W_U\times W_V &\rightarrow W \\ \nonumber
(u,v) &\mapsto \exp(-v)(u+X)\exp(v) \nonumber
\end{align}
is a diffeomorphism. $W_U$, $W_V$ and $W$ can furthermore be chosen such that $\mathcal{O}_X \cap W = \mathcal{X}(\{(u,v) \in W_U \times W_V \mid u =0\})$.
\end{lem}

\begin{proof}
The proof goes in three steps. \\

\noindent 
Step $1$: We first prove that $W_U$, $W_V$ and $W$ exist such that $\mathcal{X}$ restricts to a diffeomorphism as in the first part of the theorem. To this end, we first define $\mathcal{X}$ as a map from the whole of $U \times V$ to $\mat(\C,n)$, again given by $\mathcal{X}(u,v) = \exp(-v)(u+X)\exp(v)$. We see that  $\mathcal{X}(0,0) = X$. Furthermore, the derivative at $(0,0)$ in the directions of $(u_0,0)$ and $(0,v_0)$ are given respectively by
\begin{align}
\left. \frac{d}{dt} \right|_{t=0} \exp(0)(tu_0+X)\exp(0) = u_0
\end{align}
\noindent and 
\begin{align}
\left. \frac{d}{dt} \right|_{t=0} \exp(-tv_0)(0+X)\exp(tv_0) =\ad_X(v_0) \, .
\end{align}
As $V \oplus \ker \ad_X = \mat(\C,n)$, we see that $\{\ad_X(v_0) \mid v_0 \in V\} = \im \ad_X$. From $U \oplus \im \ad_X  = \mat(\C,n)$, we may then conclude that the derivative of $\mathcal{X}$ at $(0,0)$ is a surjective map. Furthermore, as the dimension of $U \oplus V$ is equal to that of $\mat(\C,n)$, we may conclude that the derivative is in fact a bijection. The result of the first step now follows from applying the inverse function theorem to $\mathcal{X}$.\\

\noindent Step $2$: Next, we argue that there exists an open neighborhood $S_V \subset V$ containing $0$ so that for any open $T_V \subset S_V$ containing $0$ there exists an open $R \subset \mat(\C,n)$ containing $X$ with the property that  $\mathcal{O}_X \cap R = \\ \{\exp(-v)X\exp(v) \mid v \in T_V\}$. To this end, we define the map
\begin{align}
\mathcal{Y}: V \oplus \ker \ad_X &\rightarrow \Gl(\C,n) \\ \nonumber
(v,s) &\mapsto\exp(s)\exp(v) \nonumber\, .
\end{align}
This map clearly sends $(0,0)$ to $\Id$. Furthermore, identifying the tangent space of $Gl(\C,n)$ with $\mat(\C,n) = V \oplus \ker \ad_X$, we see that the derivative of $\mathcal{Y}$ at the point $(0,0)$ is exactly given by the identity. Hence, there exist open neighborhoods $0 \ni S_V \subset V$,  $0 \ni M \subset \ker \ad_X$ and $\Id \ni N \subset \Gl(\C,n)$ so that $\mathcal{Y}$ restricts to a bijection from $S_V \times M$ to $N$. As a result, given $T_V \subset S_V$ we have that $\mathcal{Y}$ restricts to a bijection from $T_V \times M$ to the open set $N' := \mathcal{Y}(T_V \times M)$. If $T_V$ furthermore contains $0$ then $N'$ contains $\Id$. \\

\noindent Now, recall from the previous section that $\mathcal{O}_X$ is an embedded submanifold of $\mat(\C,n)$. Furthermore, exactly as in the proof of Theorem \ref{Cnil}, the map

\begin{align}
\Psi_{X}: \, &Gl(n,\C) \rightarrow \mat(\C,n)\\ \nonumber
&A \mapsto A^{-1} X A \, \nonumber
\end{align}
\noindent defines a surjective submersion onto $\mathcal{O}_X$. As a submersion is an open map, and as the topology on any embedded submanifold coincides with its induced topology, we see that there exists an open set $R$ containing $X$ so that $\Psi_{X}(N') = \mathcal{O}_X \cap R$. Writing out $\Psi_{X}(N')$ we get

\begin{align}
\Psi_{X}(N') &= \{A^{-1}XA \mid A \in N'\} \\ \nonumber
&= \{\exp(-v)\exp(-s)X\exp(s)\exp(v) \mid (v,s) \in T_V \times M\} \\ \nonumber 
&= \{\exp(-v)X\exp(v) \mid v \in T_V \}\, , \\ \nonumber
\end{align}
where in the last step we have used that $M \subset \ker \ad_X$. More specifically, $s \in M$ gives that $s$ commutes with $X$. Therefore, so does $\exp(s)$. We see that indeed $\mathcal{O}_X \cap R = \{\exp(-v)X\exp(v) \mid v \in T_V\}$, thereby proving the second step.\\

\noindent Step $3$: To conclude, we show that $W_U$, $W_V$ and $W$ can be chosen small enough such that $\mathcal{O}_X \cap W = \mathcal{X}(\{(u,v) \in W_U \times W_V \mid u =0\})$. First, choose $W_U$, $W_V$ and $W$ as in the first part of the theorem. That is, any element $w$ of $W$ can be uniquely written as $w= \mathcal{X}(u,v) = \exp(-v)(u+X)\exp(v)$ for some $(u,v) \in W_U \times W_V$. Next, let $T_V := W_V \cap S_V \subset S_V$, where $S_V$ is determined in step $2$. It follows that there is an open set $R$ so that \\
\noindent $\mathcal{O}_X \cap R = \{\exp(-v)X\exp(v) \mid v \in T_V\}$. As $T_V$ is contained in $W_V$, we see that $\{\exp(-v)X\exp(v) \mid v \in T_V\} \subset \{\exp(-v)X\exp(v) \mid v \in W_V\} \subset W$. Therefore, we may assume that $R$ lies in $W$. Finally, choose $W'_U \subset W_U$, $W'_V \subset W_V$ and $W' \subset W$ such that the first part of the theorem applies to the triple $W'_U$, $W'_V$ and $W'$, and such that $W' \subset R$. We claim that this new triple satisfies $\mathcal{O}_X \cap W' = \mathcal{X}(\{(u,v) \in W'_U \times W'_V \mid u =0\})$. Any element of the right hand side is of the form $\mathcal{X}(0,v) = \exp(-v)X\exp(v)$, and is therefore clearly contained in $\mathcal{O}_X$ (as well as in $W'$). Conversely, we pick an element $w \in W'$. It follows that $w$ may be written as $w = \mathcal{X}(u,v)$ for $u \in W'_U$ and $v \in W'_V $. If we furthermore assume that $w \in \mathcal{O}_X$, then since $W' \subset R$ we may also write $w =  \exp(-v')X\exp(v') = \mathcal{X}(0,v')$ for some $v' \in T_V$. However, as $W'_U \subset W_U$, $W'_V \subset W_V$ and $T_V \subset  W_V$, we see that $w = \mathcal{X}(u,v) =  \mathcal{X}(0,v') \in W$ can apparently be written in two ways as the image of $\mathcal{X}$ restricted to $W_U \times W_V$. This can only be true if $u=0$ and $v=v' \in W'_V$. Hence we conclude that $\mathcal{O}_X \cap W' \subset \mathcal{X}(\{(u,v) \in W'_U \times W'_V \mid u =0\})$. From this we see that the two sets are in fact equal. This proves the lemma.
\end{proof}

\noindent Next, we will describe the sets and matrices that will eventually parametrize the manifolds in Theorem \ref{Ccen}.

\begin{defi}
Let $\mathcal{P}(n)$ denote the set of partitions of $n$. Given any partition $p =(s_1, \dots, s_k) \in \mathcal{P}(n)$, we may make a sub-partition by assigning elements $p_1 \in \mathcal{P}(s_1), \dots,  p_k \in \mathcal{P}(s_k)$. All the possible ways of doing this are captured by the set

\[ \Xi_n := \{ (p; p_1, \dots p_k) \mid p = (s_1, \dots, s_k) \in \mathcal{P}(n), \, p_i \in \mathcal{P}(s_i) \, \forall i \in \{1, \dots, k\}  \} \, .\]
Note that $\Xi_n$ is a finite set, as we have $\# \Xi_n \leq (\# \mathcal{P}(n))^{n+1}$.

\noindent Given an element $\xi = (p; p_1, \dots p_k) \in \Xi_n$, we will define the set $V_{\xi}$, given by

\[ V_{\xi} := \{(x_1, \dots, x_k) \in \R^k \mid x_i \not= x_j  \text{ for } i \not= j \} \, .\]
Note that $V_{\xi}$ is an open subset of $\R^k$. The definitions of $\Xi_n$ and $V_{\xi}$ will serve to parametrize all matrices in  $\mathcal{C}^P_n$ with a purely imaginary spectrum. In particular, given $\xi  = (p; p_1, \dots p_k) \in \Xi_n$ and $x = (x_1, \dots, x_k) \in V_{\xi}$, we define the complex $n \times n$ matrix

\begin{equation}\label{blokk2}
B_{\xi}(x) = \begin{pmatrix}
  B_{s_1, p_1}(x_1i) &0& \dots & 0 \\
  0 & B_{s_2, p_2}(x_2 i) & \dots & 0\\
   &  & \ddots &  \\
  0 & \dots  & 0 & B_{s_k, p_k}(x_k i)
 \end{pmatrix} \, ,
\end{equation}
for $i$ the complex unit. We will also define the matrix $B_{\xi}(0)$ to be the matrix in $\eqref{blokk2}$ with $x_1 = \dots = x_k = 0$ (even though $0 \in \R^k$ is clearly not an element of $V_{\xi}$ for $k>1$).
\end{defi}

\noindent The following lemma gathers up some facts about the matrices $B_{\xi}(x)$ needed to prove Theorem \ref{Ccen}.

\begin{lem}\label{biggest}
Given $\xi = (p; p_1, \dots p_k) \in \Xi_n$ and $x, y \in V_{\xi}$, it holds that

\begin{align}
\im( \ad_{B_{\xi}(x)} ) &= \im( \ad_{B_{\xi}(y)} ) \\ \nonumber
\ker( \ad_{B_{\xi}(x)} ) &= \ker( \ad_{B_{\xi}(y)} ) \\ \nonumber
\im( \ad_{B_{\xi}(0)} ) &\subset \im( \ad_{B_{\xi}(x)} )  \, . \nonumber
\end{align}

\noindent Furthermore, we have that 

\begin{align}
\dim_{\C} \im( \ad_{B_{\xi}(x)} )  \leq n^2 - n \, ,
\end{align}
with equality only when $p_i = (s_i)$ for all $i \in \{1, \dots, k\}$. Lastly, if a matrix 

\begin{equation}\label{blokk3}
 I_{p}(z) := \begin{pmatrix}
  z_1 \Id_{s_1 \times s_1} &0& \dots & 0 \\
  0 & z_2 \Id_{s_2 \times s_2} & \dots & 0\\
   &  & \ddots &  \\
  0 & \dots  & 0 & z_k \Id_{s_k \times s_k}
 \end{pmatrix} \, ,
\end{equation}
for $z = (z_1, \dots, z_k) \in \C^k$ lies in $\im( \ad_{B_{\xi}(x)} ) $, then $z_1 = \dots = z_k = 0$. 
\end{lem}

\begin{proof}
Given $X \in \mat(\C, n)$ we write $X = (X_{i,j})$, $1 \leq i,j \leq k$, with respect to the block decomposition of the matrix $\eqref{blokk2}$. Whenever $i \not= j$, we see that 

\begin{align}
(\ad_{B_{\xi}(x)}(X))_{i,j} &= [B_{\xi}(x),X]_{i,j} =  B_{s_i, p_i}(x_i i)X_{i,j} -  X_{i,j} B_{s_j, p_j}(x_j i) \\ \nonumber
&= -\mathcal{L}_{B_{s_j, p_j}(x_j i), B_{s_i, p_i}(x_i i)}(X_{i,j}) \, . \nonumber
\end{align}
As $x_i \not= x_j$, it follows from Lemma \ref{cominv} that the operator $\mathcal{L}_{B_{s_j, p_j}(x_j i), B_{s_i, p_i}(x_i i)}$ is a bijection. Hence by choosing $X_{i,j}$ appropriately, any value of the block $(\mathcal{L}_{B_{\xi}(x),B_{\xi}(x)}(X))_{i,j}$ can be attained.  Likewise, it holds that

\begin{align}
&(\ad_{B_{\xi}(x)}(X))_{j,j} = \mathcal{L}_{B_{s_j, p_j}(x_j i), B_{s_j, p_j}(x_j i)}(X_{j,j}) \, . \nonumber
\end{align}
By Lemma \ref{dimnil0}, the image and kernel of $\mathcal{L}_{B_{s_j, p_j}(x_j i), B_{s_j, p_j}(x_j i)}$ are independent of $x_j i$. Therefore, the image and kernel of $\ad_{B_{\xi}(x)} $ are independent of $x \in V_{\xi}$. We also conclude from this that the image of $\ad_{B_{\xi}(0)}$ is contained in that of $\ad_{B_{\xi}(x)}$, as the image of these operators is the same in every $(j,j)$ block entry and because $\ad_{B_{\xi}(x)}$ is bijective in the other block entries. Next, we note that the dimension of the image of $(\ad_{B_{\xi}(x)}(X))_{j,j} $ is equal to $s_j^2 - s_j$ when $p_j = (s_j)$ and strictly less otherwise. This proves that 
\begin{align}
\dim \im( \mathcal{L}_{B_{\xi}(x),B_{\xi}(x)} )  \leq n^2 - n \, ,
\end{align}
with equality only when $p_i = (s_i)$ for all $i \in \{1, \dots, k\}$. Finally, it holds that 

\begin{align}
\tr((\mathcal{L}_{B_{\xi}(x),B_{\xi}(x)}(X))_{j,j}) = \tr([X_{j,j},B_{s_j, p_j}(x_j i)]) = 0\, ,
\end{align}
for all $j \in \{1, \dots, k\}$ and $X \in \mat(\C, n)$. From this it follows that $I_{p}(z)$ can only be in the image of $\mathcal{L}_{B_{\xi}(x),B_{\xi}(x)}$ when $z=0$. This finishes the proof of the lemma.
\end{proof}

\noindent The statement that $p_i = (s_i)$ for all $i \in \{1, \dots, k\}$ can be put more succinctly as the statement that the characteristic polynomial of $B_{\xi}(x)$ is equal to its minimal polynomial. It is not hard to see at this point that this condition on a matrix is equivalent to it having an adjoint orbit of maximal dimension $n^2-n$. See also \cite{alggroup}.  \\

\noindent Lastly, we will use the following lemma.

\begin{lem}\label{bloks}
Let $\{A_i\}_{i=1}^k$ and $\{B_i\}_{i=1}^k$ be two sets of matrices, where $A_i, B_i \in \mat(\C,s_i)$ for some numbers $s_i$, $1 \leq i \leq $k. Set $n:= s_1 + \dots + s_k$ and define $A$ and $B$ to be the $n \times n$ block diagonal matrices with blocks the matrices $\{A_i\}_{i=1}^k$ and $\{B_i\}_{i=1}^k$, respectively. Suppose the eigenvalues of $A_i$ and $B_i$ are the same for all $i$, that $A_i$ and $A_j$ do not share any eigenvalues for $i \not= j$ and that $A$ and $B$ are conjugate. Then, $A_i$ and $B_i$ are conjugate for all $i$.
\end{lem}

\begin{proof}
Both $A$ and $B$ can be seen as block diagonal matrices with just two blocks, by taking the first block to be $A_1$ or $B_1$ and the second block to contain all the other $A_i$ or $B_i$. As these two blocks also satisfy the conditions of the lemma, we see that we may assume that $k=2$. An induction argument then finishes the proof. Therefore, let $X$ be an invertible matrix  such that $A = X^{-1}BX$, or equivalently, $XA = BX$. Writing $X = \{X_{i,j}\}$, $i,j \in \{1,2\}$, with respect to the block structure of $A$ and $B$, we see that $X_{1,2}A_2 = B_1X_{1,2}$. In other words, we have  $\mathcal{L}_{A_2, B_1}(X_{1,2}) = 0$. However, as $A_2$ and $B_1$ do not share any eigenvalues, we conclude from Lemma \ref{cominv} that $X_{1,2} = 0$. Likewise, we see that $X_{2,1} = 0$. Since $X$ is invertible, we conclude that both $X_{1,1}$ and $X_{2,2}$ are invertible. Hence, it follows that $A_1 = X_{1,1}^{-1}B_1X_{1,1}$ and $A_2 = X_{2,2}^{-1}B_2X_{2,2}$. This proves the lemma.
\end{proof}

\begin{proof}[Proof of Theorem \ref{Ccen}]
It follows from the definitions of $\Xi_n$ and $V_{\xi}$ that any matrix with purely imaginary spectrum is conjugate to at least one matrix $B_{\xi}(x)$ with $x \in V_{\xi}$. Therefore, our manifolds will be the sets

\[ \mathcal{O}_{\xi} := \{A^{-1}B_{\xi}(x)A \mid A \in \Gl(\C,n), \, x \in V_{\xi}\} \,, \]
for $\xi = (p; p_1, \dots p_k) \in \Xi_n$. These sets are not necessarily disjoint. For example, if $p \in \mathcal{P}(4)$ is given by $p = (2,2)$, then $\xi = (p;p_1,p_2)$ and $\xi' = (p;p_2,p_1)$ will define the same sets $ \mathcal{O}_{\xi} = \mathcal{O}_{\xi'}$ for all $p_1, p_2 \in \mathcal{P}(2)$. However, this is the only thing that may happen; as soon as $ \mathcal{O}_{\xi}$ and $\mathcal{O}_{\xi'}$ share an element, they coincide as sets. We may therefore assume these sets are disjoint after discarding doubles. \\
To show that they are indeed embedded submanifolds of the proposed dimension, we fix a matrix $B_{\xi}(x)$. As in Lemma \ref{neigh}, let $U, V  \subset \mat(\C, n)$ be two complex linear spaces such that

\[U \oplus \im \ad_{B_{\xi}(x)} = V \oplus \ker \ad_{B_{\xi}(x)} = \mat(\C,n)\, .\]
By the proof of Lemma \ref{biggest}, we may assume that all elements of $U$ are block diagonal matrices with respect to the structure of $B_{\xi}(x)$ into $k$ blocks. Furthermore, we may assume that $I_p(z)$ is an element of $U$ for all $z \in \C^k$. Let $W_U \subset U$, $W_V \subset V$ and $W \subset \mat(\C,n)$ be open sets as in Lemma \ref{neigh} applied to $B_{\xi}(x)$. We may assume that $W_U$ is small enough so that for all $I_p(z) \in W_U$ it still holds that $B_{\xi}(x) + I_p(z)$ has different diagonal entries among its $k$ blocks. Now, the set 

\[ \{\exp(-v)(B_{\xi}(x) + I_p(z))\exp(v) \mid v \in W_V, \, I_p(z) \in W_U\, , z \in (i\R)^k  \}   \subset W\]
is readily seen to be contained in $W \cap \mathcal{O}_{\xi}$. Hence, if we can show that equality holds for these two sets then we have proven that $\mathcal{O}_{\xi}$ is (around $B_{\xi}(x)$) an embedded submanifold of real dimension $k + 2\dim_{\C}(V)$. Note that $V$ may be chosen the same for all $x \in V_{\xi}$. As any element of $\mathcal{O}_{\xi}$ is conjugate to some element $B_{\xi}(x)$, we would conclude by homogeneity of $\mathcal{O}_{\xi}$ that $\mathcal{O}_{\xi}$ is an embedded submanifold. \\

\noindent Therefore, let us assume that for any open set $S \subset W$ around $B_{\xi}(x)$ there is an element in $S \cap \mathcal{O}_{\xi}$ that is not of the form $\exp(-v)(B_{\xi}(x) + I_p(z))\exp(v)$ for $v \in W_V$, $I_p(z) \in W_U$ and with $z \in (i\R)^k$. We will show that this leads to a contradiction. From the assumptions on $B_{\xi}(x)$ we get a sequence of matrices $(X_r)_{r=0}^{\infty}$ such that 

\begin{enumerate}
\item $\displaystyle \lim_{r \rightarrow \infty} X_r = B_{\xi}(x)$.
\item $X_r$ is conjugate to  $B_{\xi}(x^r)$ for some  $x^r \in V_{\xi}$.
\item Every $X_r$ is not of the form $\exp(-v)(B_{\xi}(x) + I_p(z))\exp(v)$ for $v \in W_V$, $I_p(z) \in W_U$ and with $z \in (i\R)^k$.
\end{enumerate}
Since $X_r \in W$ for all $r$, we may write $X_r = \exp(-v_r)(B_{\xi}(x) + u_r)\exp(v_r)$ for $v_r \in W_V$ and $u_r \in W_U$. As the limit of $X_r$ equals $B_{\xi}(x) = \exp(0)(B_{\xi}(x) + 0)\exp(0)$, it follows that 

\begin{enumerate}
\item $\displaystyle \lim_{r \rightarrow \infty} u_r = 0$.
\item $B_{\xi}(x) + u_r$ is conjugate to $B_{\xi}(x^r)$  for some  $x^r \in V_{\xi}$.
\item $\displaystyle \lim_{r \rightarrow \infty} v_r = 0$.
\end{enumerate}
Note that every $u_r$ is a block diagonal matrix, as it is an element of $U$. Therefore, so is $B_{\xi}(x) + u_r$ for all $r$. We will denote the individual blocks by $(B_{\xi}(x) + u_r)_j = B_{\xi}(x)_j + u^j_r = B_{s_j, p_j}(x_ji) + u^j_r$ for $1 \leq j \leq k$. Now, the limit of all the eigenvalues of $(B_{\xi}(x) + u_r)_j$ is $x_ji$. Hence, as per assumption $x_j \not= x_l$ for $j \not= l$, we may conclude that for $r$ big enough, the blocks $(B_{\xi}(x) + u_r)_j$ and $(B_{\xi}(x) + u_r)_l$ do not share any eigenvalues if $j \not= l$. On the other hand, the eigenvalues of $B_{\xi}(x) + u_r$ are equal to those of $B_{\xi}(x^r)$, as these matrices are conjugate. It follows that the eigenvalue $x^r_ji$, which appears with algebraic multiplicity $s_j$, appears with the same multiplicity in exactly one of the blocks of $B_{\xi}(x) + u_r$. Therefore, every block of $B_{\xi}(x) + u_r$ has exactly the same eigenvalues as some block of $B_{\xi}(x^r)$. More precisely, for every $r$ there exists a permutation $\sigma_r \in S_k$ so that $(B_{\xi}(x) + u_r)_j$ and $B_{\xi}(x^r)_{\sigma_r(j)} = B_{s_{\sigma_r(j)}, p_{\sigma_r(j)} }(x^r_{\sigma_r(j)}i)$ have the same eigenvalues. It therefore follows from Lemma \ref{bloks} that 

\[(B_{\xi}(x) + u_r)_j \text{ is conjugate to } B_{\xi}(x^r)_{\sigma_r(j)} \text{ for every } j \, .\]

\noindent Next, by comparing traces and by noting that the limit of $(B_{\xi}(x) + u_r)_j$ is $B_{\xi}(x)_j$, we see that 

\begin{equation}\label{limiteig}
\lim_{r \rightarrow \infty} x^r_{\sigma_r(j)} = x_j \, .
\end{equation}
We will use the facts we have gathered so far, together with the fact that the orbit of $B_{\xi}(0)$ is an embedded manifold, to arrive at a contradiction. To this end, we look at the expressions $u_r^j + (x_j - x^r_{\sigma_r(j)})i\Id_{s_j}$. From \eqref{limiteig} we see that

\begin{enumerate}
\item $\displaystyle \lim_{r \rightarrow \infty} u_r^j + (x_j - x^r_{\sigma_r(j)})i\Id_{s_j} = 0$.
\item $B_{\xi}(0)_j + u_r^j + (x_j - x^r_{\sigma_r(j)})i\Id_{s_j} = (B_{\xi}(x) + u_r)_j - x^r_{\sigma_r(j)}i\Id_{s_j}$ 
 is conjugate to  $B_{\xi}(x^r)_{\sigma_r(j)} -  x^r_{\sigma_r(j)}i\Id_{s_j} = B_{\xi}(0)_{\sigma_r(j)}$ .
\end{enumerate}
In part 2, we have simply used the fact that if two matrices $A$ and $B$ are conjugate, then so are $A+z\Id$ and $B+z\Id$ for any $z \in \C$. If we define $y^r \in \C^k$ by $y^r_j := (x_j - x^r_{\sigma_r(j)})i$ for $1 \leq j \leq k$, then we get for the full matrices

\begin{enumerate}
\item $\displaystyle \lim_{r \rightarrow \infty} u_r + I_p(y^r) = 0$.
\item $B_{\xi}(0) + u_r + I_p(y^r)$  is conjugate to $B_{\xi}(0)$.
\end{enumerate}
Note that $u_r + I_p(y^r) \in U$ for all $r$. \\

\noindent Finally, let $\tilde{U}, \tilde{V} \subset \mat(\C, n)$ be two complex linear spaces such that

\[\tilde{U} \oplus \im \ad_{B_{\xi}(0)} = \tilde{V} \oplus \ker \ad_{B_{\xi}(0)} = \mat(\C,n)\, .\]
From $\im \ad_{B_{\xi}(0)} \subset \im \ad_{B_{\xi}(x)}$ we see that we may choose $\tilde{U}$ such that $U \subset \tilde{U}$. We furthermore choose open sets $W_{\tilde{U}}$, $W_{\tilde{V}}$ and $\tilde{W}$ as in the statement of Lemma \ref{neigh}, so that $\mathcal{O}_{B_{\xi}(0)} \cap \tilde{W} = \mathcal{X}(\{(u,v) \in W_{\tilde{U}} \times W_{\tilde{V}} \mid u =0\})$. Now, for large enough values of $r$, the matrices $B_{\xi}(0) + u_r + I_p(y^r)$ will lie in $\tilde{W}$. Therefore, since $B_{\xi}(0) + u_r + I_p(y^r) \in \mathcal{O}_{B_{\xi}(0)}$ and $u_r + I_p(y^r) \in U \subset \tilde{U}$, it has to follow that $u_r + I_p(y^r) = 0$ for large enough $r$. Going back to $X_r$, we see that

\begin{align}
X_r &= \exp(-v_r)(B_{\xi}(x) + u_r)\exp(v_r) \\ \nonumber
&= \exp(-v_r)(B_{\xi}(x) + I_p(-y^r))\exp(v_r)\, ,
\end{align}
with $y^r_j := (x_j - x^r_{\sigma_r(j)})i$ so that $y^r \in (i\R)^k$. This is a direct contradiction to the third assumption on $X_r$. \\

\noindent Hence, there does exist an open set $S \in W$ around $B_{\xi}(x)$ where every element of $\mathcal{O}_{\xi}$ is of the form $\exp(-v)(B_{\xi}(x) + I_p(z))\exp(v)$ for $v \in W_V$, $I_p(z) \in W_U$ and  $z \in (i\R)^k$. In particular, we may choose $W'_U \subset W_U$, $W'_V \subset W_V$ and $W' \subset S \subset W$ as in Lemma \ref{neigh} for $B_{\xi}(x)$. Then 
\[\mathcal{O}_{\xi} \cap W' = \{\exp(-v)(B_{\xi}(x) + I_p(z))\exp(v) \mid v \in W'_V, \, I_p(z) \in W'_U\, , z \in (i\R)^k  \} \, ,\] 
\noindent as this otherwise contradicts the unique expression as $\mathcal{X}(u,v)$ in $W$. \\

\noindent We see that the real dimension of $\mathcal{O}_{\xi}$ is $k + 2\dim_{\C}(V)$. This value cannot exceed $2n^2 - n$, in which case $k = n$ and $\dim_{\C}(V) = n^2 - n$. By Lemma \ref{biggest} this is indeed the case when $\xi = ((1,1, \dots, 1); (1), \dots, (1))$. This is furthermore the only possibility, as $k = n$ forces the partitions in $\xi$ to be trivial.  This concludes the proof. 
\end{proof}

\begin{remk}\label{Cdis}
Note that the manifold of Theorem \ref{Ccen} of highest dimension consists of exactly those matrices with $n$ distinct (purely imaginary) eigenvalues. Another observation is that both the matrices of Theorem \ref{Cnil} and of Theorem \ref{Ccen} are invariant under taking the (component-wise) complex conjugate. This is exactly the transformation that would occur if one would choose $[\Id], [-I] \in \End(U)/\nil(U)$ as the generators of the complex structure, instead of $[\Id]$ and $[I]$. See Subsection \ref{A Remark on Uniqueness}. \hspace*{\fill}$\triangle$
%\noindent Lastly, we note that the manifold of highest dimension is contained in an open set that does not intersect with any of the lower dimensional manifolds. To see this, let $\xi = ((1,1, \dots, 1); (1), \dots, (1))$ and choose $U$ and $V$ as in Lemma \ref{neigh} for $B_{\xi}(x)$. Since the (complex) dimension of $U$ is in this case equal to $n$, we may just set $U$ to be the space of diagonal matrices. If the $u$ component of an expression $\mathcal{X}(u,v) = \exp(-v)(B_{\xi}(x) + u)\exp(v)$ has only purely imaginary entries then $\mathcal{X}(u,v)$ lies in $\mathcal{O}_{\xi}$ (assume $W_U$ to be small enough so that $B_{\xi}(x) + u$ has distinct diagonal entries). If on the other hand $u$ has an entry with a non-vanishing real part, then $\mathcal{X}(u,v)$ does not have a purely imaginary spectrum and is hence not contained in any of the manifolds $\mathcal{O}_{\xi'}$ for $\xi' \in \Xi_n$. \\
%\noindent  A similar result holds for the nilpotent matrices. The orbit of the nilpotent matrix with highest dimension is contained in the open set \\ $\{ X \in \mat(\C,n) \mid X^{n-1} \not= 0\}$, which is disjoint from the orbits of the other nilpotent matrices.  
%
\end{remk}

\subsection{The Case $\mathcal{R}^P_n$} \label{The case Rpn}
For $\mathcal{R}^P_n = \mat(\R,n)$ we have the following results. 

\begin{thr}\label{Rnil}
The set of all nilpotent matrices in $\mathcal{R}^P_n$ consists of a finite number of conjugacy invariant embedded manifolds. Exactly one of these has real dimension $n^2 -n$, whereas the others have dimension strictly less.
\end{thr}

\begin{thr}\label{Rcen}
The set of all matrices in $\mathcal{R}^P_n$ with a purely imaginary spectrum consists of a finite number of conjugacy invariant embedded manifolds. Exactly one of these has real dimension $n^2 -\lceil \frac{n}{2} \rceil$, whereas the others have dimension strictly less.
\end{thr}
\noindent These results will follow from the analogous results for $\mathcal{C}^P_n$. An important ingredient here is the following lemma.

\begin{lem}\label{reeal}
Let $\mathcal{A} \subset \mat(\C,n)$ be a real subalgebra of matrices such that 

\[\mat(\C,n) = \mathcal{A} \oplus i \mathcal{A}\]
as real vector spaces.  Let $A,B \in \mathcal{A}$ be two conjugate matrices. Then they are also conjugate using an element in $\mathcal{A}$. More precisely, if there exists an $X \in \Gl(\C,n)$ such that $A = XBX^{-1}$, then there also exists a $C \in \mathcal{A} \cap \Gl(\C,n)$ such that $A = CBC^{-1}$. Moreover, writing $X = X_1 + iX_2$ for $X_1, X_2 \in \mathcal{A}$ and choosing $\epsilon > 0$, $C$ can be chosen such that $||C - X_1|| < \epsilon$. Here, $||\cdot||$ denotes (for example) the matrix norm, $||X||^2 := \tr(X^T\overline{X})$.
\end{lem}

\begin{proof}
Write $X = X_1 + iX_2$ for $X_1, X_2 \in \mathcal{A}$. From $A = XBX^{-1}$ it follows that $AX = XB$, and hence that $AX_1 + iAX_2 = X_1B + iX_2B$. Comparing parts in $\mathcal{A}$ and  $i\mathcal{A}$, we see that both $AX_1 = X_1B$ and $AX_2 = X_2B$ hold. In particular, for any $\la \in \R$ it holds that $A(X_1 + \la X_2) = (X_1 + \la X_2)B$. Therefore, it remains to show that $X_1 + \la X_2 \in \mathcal{A} $ is invertible for arbitrarily small values of $\la$, in which case we set $C := X_1 + \la X_2$. To this end, consider the polynomial in $\la$ given by $\det(X_1 + \la X_2)$. This polynomial cannot be identically $0$, as we have $\det(X_1 + i X_2) = \det(X) \not= 0$. Therefore, there are only finitely many values of $\la$ for which $\det(X_1 + \la X_2) = 0$. We conclude that there are real values of $\la$ arbitrarily close to $0$ for which $\det(X_1 + \la X_2) \not= 0$. If $X_2 = 0$ we set $C = X_1 = X$. Otherwise, choose $0  \leq \la < \epsilon ||X_2||^{-1}$ and such that $\det(X_1 + \la X_2) \not= 0$. Then setting $C := X_1 + \la X_2$, we have $||C - X_1|| = \la||X_2|| < \epsilon ||X_2||^{-1} ||X_2|| = \epsilon$. This proves the lemma.
\end{proof}

\begin{proof}[Proof of Theorem \ref{Rnil}]

\noindent Just as in the case of $\mathcal{C}^P_n$, our manifolds will be the conjugacy orbits of the elements $B_{n,p}(0)$ for $p \in \mathcal{P}(n)$:

\[ \mathcal{Q}_{B_{n,p}(0)} := \{AB_{n,p}(0)A^{-1} \mid A \in \Gl(\R,n)\} \, .\]
Equivalently, $\mathcal{Q}_{B_{n,p}(0)}$ is the image of the map $A \in \Gl(\R,n) \mapsto AB_{n,p}(0)A^{-1}$, which has constant rank equal to $\dim_{\R} \im (\ad_{B_{n,p}(0)}|_{\mat(\R,n)})$. Since $B_{n,p}(0)$ is a real matrix, it follows that 
 \[\im(\ad_{B_{n,p}(0)}|_{\mat(\C,n)}) = \im(\ad_{B_{n,p}(0)}|_{\mat(\R,n)}) \oplus i\im(\ad_{B_{n,p}(0)}|_{\mat(\R,n)}) \, .\]
 \noindent From this we conclude that $\dim_{\R} \im (\ad_{B_{n,p}(0)}|_{\mat(\R,n)}) \leq n^2 - n$, with equality only when $p = (n)$. Therefore, it follows that every $\mathcal{Q}_{B_{n,p}(0)}$ is an immersed submanifold of the proposed dimension. In particular, there exists an open set $S \subset \Gl(\R,n)$ containing $\Id$ such that $\{AB_{n,p}(0)A^{-1} \mid A \in S\} $ is an embedded submanifold of dimension $\dim_{\R} \im (\ad_{B_{n,p}(0)}|_{\mat(\R,n)})$ containing $B_{n,p}(0)$. It remains to show that for a small enough neighborhood $T \subset \mat(\R,n)$ containing $B_{n,p}(0)$, any element in $ \mathcal{Q}_{B_{n,p}(0)} \cap T$ lies in $\{AB_{n,p}(0)A^{-1} \mid A \in S\} $. \\
 
 \noindent Assume the converse. Then there exists a sequence of elements $X_r \in \mathcal{Q}_{B_{n,p}(0)} \setminus \{AB_{n,p}(0)A^{-1} \mid A \in S\}$ such that $\displaystyle \lim_{r \rightarrow \infty} X_r = B_{n,p}(0)$. This same sequence then exists in $\mat(\C,n)$. Applying Lemma \ref{neigh} we find open neighborhoods $W_V \subset V$ containing $0$ and $W \subset \mat(\C,n)$ containing $B_{n,p}(0)$ such that any element of $\mathcal{Q}_{B_{n,p}(0)} \cap W \subset \mathcal{O}_{B_{n,p}(0)} \cap W$ can be written as $\exp(-v)B_{n,p}(0)\exp(v)$ for $v \in W_V$. Here, $V$ is a complex linear space satisfying  $V \oplus \ker \ad_{B_{n,p}(0)} = \mat(\C,n)$. We therefore write $X_r = \exp(-v_r)B_{n,p}(0)\exp(v_r)$ for large enough $r$. It also follows from Lemma \ref{neigh} that $\displaystyle \lim_{r \rightarrow \infty} X_r = B_{n,p}(0)$ implies $\displaystyle \lim_{r \rightarrow \infty} v_r = 0$. Hence we have that $\displaystyle \lim_{r \rightarrow \infty} \exp(-v_r) = \Id$. By applying Lemma \ref{reeal} with $\mathcal{A} = \mat(\R,n)$, we find matrices $C_r \in \mat(\R,n)$ such that $X_r = C_rB_{n,p}(0)C_r^{-1}$. As (the real part of) $\exp(-v_r)$ goes to $\Id$, we may arrange for the $C_r$ to have the same property. However, then for big enough $r$ we find that $C_r \in S$, contradicting that $X_r \notin \{AB_{n,p}(0)A^{-1} \mid A \in S\}$. We conclude that $\mathcal{Q}_{B_{n,p}(0)}$ is an embedded manifold around $B_{n,p}(0)$, and hence by homogeneity globally. This proves the theorem. 
\end{proof}

\noindent To prove Theorem \ref{Rcen}, we will first introduce the matrices that serve to label the relevant manifolds. Given $m \in \N$, let $\xi \in \Xi_m$ be given by $\xi = (p; p_1, \dots p_l)$. Recall that this means that $p = (s_1, \dots s_l) \in \mathcal{P}(m)$ is a partition of $m$ in $l$ numbers, whereas each $p_i$ is an element of $\mathcal{P}(s_i)$ for $1 \leq i \leq l$. We define the open set 
\begin{align}
W_{\xi} = \{x \in \R^l \mid x_i > 0, x_i \not= x_j  \text{ for all } i, j \in \{1, \dots l\} \text{ such that } i \not= j\}\, .
\end{align}
Next, we fix a number $n \in \N$. Given $m \in \{1, \dots \lfloor \frac{n}{2} \rfloor\}$, $\xi \in \Xi_m$, $x \in W_{\xi}$ and $q \in \mathcal{P}(n-2m)$ we furthermore define the $n$ times $n$ matrices 

\begin{align}
D_{n,\xi,q}(x) :=& 
\begin{pmatrix}
B_{\xi}(x)  &  0  & 0 \\
0 & \overline{B_{\xi}(x)}  & 0 \\
0 & 0 & B_{n-2m,q}(0)
\end{pmatrix}  \\ =&
\begin{pmatrix}
B_{\xi}(x)  &  0  & 0 \\
0 & B_{\xi}(-x)  & 0 \\
0 & 0 & B_{n-2m,q}(0)
\end{pmatrix}\, .
\end{align}
Note that $D_{n,\xi,q}(x)$ is conjugate to $B_{\xi'}(y)$ for some choice of $\xi' \in \Xi_n$ and $y \in V_{\xi'}$ with entries those in $x$, in $-x$ and possibly $0$. This conjugation can be done by a permutation matrix that depends only on $\xi$ and $m$ and is just an artifact of our convention to have permutations ordered. We also define the $n$ times $n$ matrices

\begin{align}
Z_{n,m} := 
\begin{pmatrix}
\Id_m  &  i\Id_m  & 0 \\
i\Id_m & \Id_m  & 0 \\
0 & 0 & \Id_{n-2m} 
\end{pmatrix}\, .
\end{align}
These have the property that for any complex matrix $X$ of the form

\begin{align}
X := 
\begin{pmatrix}
Y  &  0  & 0 \\
0 & \overline{Y}  & 0 \\
0 & 0 & W
\end{pmatrix}\, ,
\end{align}
where $Y$ is a complex $m$ times $m$ matrix and $W$ is a real $n-2m$ times $n-2m$ matrix, the matrix $Z_{n,m} X Z^{-1}_{n,m}$ is real (i.e. has real entries). Lastly for $\xi = (p; p_1, \dots p_l) \in \Xi_m$ and $z \in \C^l$, we define the matrices 

\begin{align}
\bar{I}_{\xi}(z) := 
\begin{pmatrix}
I_{p}(z)  &  0  & 0 \\
0 & \overline{I_{p}(z)}  & 0 \\
0 & 0 & 0
\end{pmatrix}\, .
\end{align}

\begin{proof}[Proof of Theorem \ref{Rcen}]
Note that every real matrix with a purely imaginary spectrum is either nilpotent, or contained in one of the sets

\[ \mathcal{Q}_{\xi,q} := \{AZ_{n,m}D_{n,\xi,q}(x)Z_{n,m}^{-1}A^{-1} \mid A \in \Gl(\R,n), x \in W_{\xi}\} \subset \mat(\R,n) \, .\]
Here we have  $\xi \in \Xi_m$ and $q \in \mathcal{P}(n-2m)$, where $m$ may furthermore vary from $1$ to $\lfloor \frac{n}{2} \rfloor$. It can again be seen that two sets $\mathcal{Q}_{\xi,q}$ and $\mathcal{Q}_{\xi',q'}$ are either the same or disjoint. The set $\mathcal{Q}_{\xi,q}$ is equal to the image of the smooth map

\begin{align}
\Psi_{\xi,q}: \Gl(\R,n) \times W_{\xi} &\rightarrow \mat(\R,n) \\ \nonumber
(A,x) & \mapsto AZ_{n,m}D_{n,\xi,q}(x)Z_{n,m}^{-1}A^{-1}  \nonumber\, .
\end{align}
We will first show that this map has constant rank, thereby showing that its image is an immersed manifold of the proposed dimension. After that, we show that it is an embedded manifold, by comparing to the complex case similarly to what we did in the proof of Theorem \ref{Rnil}. \\
\noindent We fix a point $(A,x) \in \Gl(\R,n) \times W_{\xi}$ and a direction $(V,w) \in \mat(\R,n) \oplus \R^l$. A curve through $(A,x)$ with velocity $(V,w)$ is then given by \\  \noindent $t \mapsto (A\exp(tA^{-1}V), x+tw)$ and we find

\begin{align}
&T_{(A,x)}\Psi_{\xi,q}(V,w) \nonumber  \\
&= \left. \frac{d}{dt} \right|_{t=0}A\exp(tA^{-1}V)Z_{n,m}D_{n,\xi,q}(x+tw)Z_{n,m}^{-1}\exp(-tA^{-1}V)A^{-1} \nonumber \\
&= A[A^{-1}V, Z_{n,m}D_{n,\xi,q}(x)Z_{n,m}^{-1}]A^{-1} + AZ_{n,m}\bar{I}_{\xi}(iw) Z_{n,m}^{-1}A^{-1} \, .
\end{align} 
As conjugating by $A$ does not change the dimension of a space, and as $A^{-1}V$ varies over the real matrices as $V$ does, we see that the rank of the linearization is independent of $A$. We therefore set $A$ equal to the identity. It remains to determine the dimension of the real space

\begin{equation}\label{tanggspace}
 \{ [V, Z_{n,m}D_{n,\xi,q}(x)Z_{n,m}^{-1}]  + Z_{n,m}\bar{I}_{\xi}(iw) Z_{n,m}^{-1} \mid V \in \mat(\R,n), w \in \R^l\} \, . 
 \end{equation}
 
 \noindent First suppose a matrix $B$ is both of the form $Z_{n,m}\bar{I}_{\xi}(iw) Z_{n,m}^{-1}$ for some $w \in \R^l$ and of the form $[V, Z_{n,m}D_{n,\xi,q}(x)Z_{n,m}^{-1}]$ for some $V \in \mat(\R,n)$. Then $Z_{n,m}^{-1}B Z_{n,m}$ is a complex matrix that can be written as $\bar{I}_{\xi}(iw)$ and as $[V', D_{n,\xi,q}(x)]$ for some $V' \in \mat(\C,n)$. This is a contradiction to the fact that the diagonal blocks of any element of the form $[V', D_{n,\xi,q}(x)]$ have vanishing trace, unless $w=0$ and hence $B=0$ (compare to the proof of Lemma \ref{biggest}). We conclude that the space in \eqref{tanggspace} is a direct sum of its two components. Clearly we have that the real dimension of 
 
 \[ \{Z_{n,m}\bar{I}_{\xi}(iw) Z_{n,m}^{-1} \mid  w \in \R^l\} \]
 equals $l$. Furthermore, as $Z_{n,m}D_{n,\xi,q}(x)Z_{n,m}^{-1}$ is a real matrix, we have that the real dimension of 
 \[  \{ [V, Z_{n,m}D_{n,\xi,q}(x)Z_{n,m}^{-1}]  \mid V \in \mat(\R,n)\}  \]
 is equal to the complex dimension of 
 \[  \{ [V, Z_{n,m}D_{n,\xi,q}(x)Z_{n,m}^{-1}]  \mid V \in \mat(\C,n)\} \, . \] 
 This latter space has the same complex dimension as the space
 \[  \{ [V, D_{n,\xi,q}(x)]  \mid V \in \mat(\C,n)\} \] 
 which we know from Lemma \ref{biggest} to be independent of $x$, and furthermore at most equal to $n^2 - n$ (recall that $D_{n,\xi,q}(x)$ is conjugate to $B_{\xi'}(y)$ for some $\xi' \in \Xi_n$ and $y \in V_{\xi'}$). By the constant rank theorem, every set $\mathcal{Q}_{\xi,q}$ is an immersed manifold of real dimension at most $n^2 - n + \lfloor \frac{n}{2} \rfloor = n^2 - \lceil \frac{n}{2} \rceil$. Furthermore, to get this exact number, we need to have that $m = \lfloor \frac{n}{2} \rfloor$ and that $\xi = ((1, \dots 1); (1), \dots (1)) \in \Xi_m$. This also fixes $q$ to be either $(1)$ (if $n$ is odd) or empty, (if $n$ is even). In both cases all eigenvalues of $D_{n,\xi,q}(x)$ are different, and we see that the dimension of the image of its adjoint operator is indeed equal to $n^2 - n$. We conclude that the maximal value of $n^2 - \lceil \frac{n}{2} \rceil$ is attained in exactly one case. Note that the dimension of any nilpotent orbit is at most $n^2-n$, which is less than $n^2 - \lceil \frac{n}{2} \rceil$ for $n > 1$. If $n=1$ then the nilpotent matrices are the matrices with a purely imaginary spectrum, both sets being equal to $\{0\}$. \\
\noindent Next, we prove that $\mathcal{Q}_{\xi,q}$ is in fact an embedded manifold of $\mat(\R,n)$. To this end, we fix $x \in W_{\xi}$. By the constant rank theorem, there exist open neighborhoods $S \subset \Gl(\R,n)$ containing $\Id$ and $T \subset W_{\xi}$ containing $x$ such that 

\[ \{ \Psi_{\xi,q}(A,y)  \mid A \in S, y \in T \} \subset  \mat(\R,n) \]
is an embedded manifold of real dimension equal to the rank of the derivative of $\Psi_{\xi,q}$. It remains to show that for $S$ and $T$ sufficiently small, there are no other elements of  $\mathcal{Q}_{\xi,q}$ nearby. Assuming the converse, we get a sequence of real matrices $(X_r)_{r=0}^{\infty}$ in $\mathcal{Q}_{\xi,q}$ limiting $Z_{n,m}D_{n,\xi,q}(x)Z_{n,m}^{-1}$ that are not in this embedded manifold. Pick a permutation matrix $P$ such that $PD_{n,\xi,q}(x)P^{-1} =: B_{\xi'}(y)$ for some $\xi' \in \Xi_n$ and $y \in V_{\xi}$. This permutation matrix just reorders the blocks, so that their sizes are decreasing. We define $Y_r := PZ_{n,m}^{-1}X_rZ_{n,m}P^{-1}$, so that the limit of $Y_r$ is equal to $B_{\xi'}(y)$. Note that the $Y_r$ may not be real matrices anymore. From the Jordan normal form we see that $\mathcal{Q}_{\xi,q} \subset \mathcal{O}_{\xi'}$. Therefore, $X_r$ describes a sequence in $\mathcal{O}_{\xi'}$. By conjugacy invariance of $\mathcal{O}_{\xi'}$,  so does $Y_r$. By the conclusion at the end of the proof of Theorem \ref{Ccen}, we see that we may write 

\begin{align}
Y_r = \exp(-v_r)(B_{\xi'}(y) + I_{\xi'}(z^r))\exp(v_r) \,
\end{align}
for certain complex matrices $v_r$ and with $z^r \in (i\R)^k$. Here, $k$ is determined by $\xi' = ((s_1, \dots s_k); p_1, \dots p_k)$. It furthermore holds that $\displaystyle \lim_{r \rightarrow \infty} v_r = 0$ and $\displaystyle \lim_{r \rightarrow \infty} z^r = 0$. \\
Our next step is to show that for sufficiently large values of $r$, the matrices $P^{-1}(B_{\xi'}(y) + I_{\xi'}(z^r))P$ are of the form $D_{n,\xi,q}(x^r)$ for some $x^r \in W_{\xi}$. If this holds, then the matrices $Z_{n,m}P^{-1}(B_{\xi'}(y) + I_{\xi'}(z^r))PZ_{n,m}^{-1} = Z_{n,m}D_{n,\xi,q}(x^r)Z_{n,m}^{-1}$ are real. As they are furthermore conjugate to the real matrices $X_r,$ we conclude from Lemma \ref{reeal} that this conjugation can be done by real matrices as well. We will then finish the proof by showing that this leads to a contradiction. \\
\noindent To show that the matrices $P^{-1}(B_{\xi'}(y) + I_{\xi'}(z^r))P$ are of the form $D_{n,\xi,q}(x^r)$, we need to show that the eigenvalues in the different blocks of $B_{\xi'}(y) + I_{\xi'}(z^r)$ satisfy a property that states which pairs of blocks have eigenvalues with opposite sign. Motivated by this, we say that an element  $v \in \C^k$ satisfies the \textit{real-property} if there exists a function $\tau:  \{1, \dots k\} \rightarrow \{1, \dots k\}$  such that $v_{j} = -v_{\tau(j)}$ for all indices $j$. Note that  the eigenvalues in the blocks of $B_{\xi'}(y) + I_{\xi'}(z^r)$ have this property, as they are conjugate to the real matrices $X_r \in \mathcal{O}_{\xi'}$. Likewise, $y \in V_{\xi'} \subset \C^k$ has the real-property, for some involution $\tau_0$. In fact, there is only one  function from $\{1, \dots k\}$ to itself for which $y$ has this property. For, if $\tau_1$ is another, and we have $\tau_0(j) \not= \tau_1(j)$ for some index $j$, then $y_{\tau_0(j)} = -y_j = y_{\tau_1(j)}$. However, as the entries of $y$ are just those of $x \in W_{\xi}$, minus those and perhaps $0$, the entries of $y$ are all different. This shows that such a $j$ cannot exist, and therefore that $\tau_0$ is unique. The same therefore holds for $iy$\\
\noindent Now, all the elements in $\C^k$ that satisfy the real-property form a set that is the union of a finite number of hyperplanes. These hyperplanes are indexed by all the possible functions from $\{1, \dots k\}$ to itself, and are all of a strictly smaller dimension than $\C^k$. Since $iy$ lies in exactly one such hyperplane, its distance to the other hyperplanes is strictly positive. Therefore, as the elements $iy + z^r$ have the real-property and limit $iy$, they too will lie in the hyperplane indexed by $\tau_0$ for large enough values of $r$. This shows that for large enough values of $r$, the eigenvalues of $B_{\xi'}(y) + I_{\xi'}(z^r)$ are paired correctly, and we may write $P^{-1}(B_{\xi'}(y) + I_{\xi'}(z^r))P = D_{n,\xi,q}(x^r)$ for some $x^r \in W_{\xi}$.

\noindent Returning to the $X_r$, we have

\begin{align}
X_r &= Z_{n,m}P^{-1}\exp(-v_r)(B_{\xi'}(y) + I_{\xi'}(z^r))\exp(v_r)PZ_{n,m}^{-1} \\ \nonumber
&= A_rZ_{n,m}P^{-1}(B_{\xi'}(y) + I_{\xi'}(z^r))PZ_{n,m}^{-1}A_r^{-1} \\ \nonumber
&= A_rZ_{n,m}D_{n,\xi,q}(x^r)Z_{n,m}^{-1}A_r^{-1} \, , \nonumber
\end{align}
for $A_r := Z_{n,m}P^{-1}\exp(-v_r)PZ_{n,m}^{-1}$. As $v_r$ goes to $0$, we see that the limit of $A_r$ is the identity. By Lemma \ref{reeal}, there exist real matrices $C_r$ such that

\begin{align}
X_r = C_rZ_{n,m}D_{n,\xi,q}(x^r)Z_{n,m}^{-1}C_r^{-1} =  \Psi_{\xi,q}(C_r,x_r) 
\end{align}
with furthermore $\displaystyle \lim_{r \rightarrow \infty} C_r = \Id$. Since it also holds that $\displaystyle \lim_{r \rightarrow \infty} x_r = x$, we see that $(C_r,x_r) \in S\times T$ for large enough $r$. This contradicts our assumption, and hence $\mathcal{Q}_{\xi,q}$ is locally around $Z_{n,m}D_{n,\xi,q}(x)Z_{n,m}^{-1}$ an embedded manifold. By homogeneity, it is globally an embedded manifold. This proves the theorem.
\end{proof}

\noindent Note that it follows from the proof of Theorem \ref{Rcen} that the unique manifold of highest dimension consists exactly of those matrices with no double eigenvalues. 

\subsection{The Case $\mathcal{H}^P_n$} \label{The case Hpn}
Recall that
\begin{align}
\mathcal{H}^P_n :=\left \{  \begin{pmatrix}
  X&Y\\
  -\overline{Y} & \overline{X}
 \end{pmatrix}, \, X, Y \in \mat(\C,n) \right \} \subset \mat(\C,2n) 
\end{align}
satisfies $\mat(\C,2n) = \mathcal{H}^P_n \oplus i\mathcal{H}^P_n$ as real vector spaces. Recall also that 
 \begin{equation}
\mathcal{H}^P_n = \{Z \in \mat(\C,2n) \text{ such that } SZ = \overline{Z}S\} \, ,
 \end{equation}
 for 
 \begin{equation}
S = \begin{pmatrix}
 0 & \Id_n \\
 -\Id_n & 0 
 \end{pmatrix} \, .
 \end{equation}
This matrix satisfies $S^2 = -\Id_{2n}$. Our aim is to prove the following theorems. 

\begin{thr}\label{Hnil}
The set of all nilpotent matrices in $\mathcal{H}^P_n$ consists of a finite number of conjugacy invariant embedded manifolds. Exactly one of these has real dimension $4n^2 -4n$, whereas the others have dimension strictly less.
\end{thr}

\begin{thr}\label{Hcen}
The set of all matrices in $\mathcal{H}^P_n$ with a purely imaginary spectrum consists of a finite number of conjugacy invariant embedded manifolds. Exactly one of these has real dimension $4n^2 -n$, whereas the others have dimension strictly less.
\end{thr}

\noindent The following lemma will enable us to describe those elements in $\mathcal{H}^P_n$ with a vanishing or purely imaginary spectrum. This result is known, see for example \cite{Jord}, but a relatively short proof is given for completeness. The techniques used in the proof below are known to experts, but can be hard to find in the literature.
\begin{lem}\label{normalfo}
Any element of $\mathcal{H}^P_n$ is conjugate to an element of the form

\begin{equation}\label{normaal}
 \begin{pmatrix}
 N & 0 \\
0 & \overline{N} 
 \end{pmatrix} \, .
 \end{equation}
Here, $N$ is a complex matrix in Jordan normal form. Note that the matrix of \eqref{normaal} is an element of $\mathcal{H}^P_n$. Thus, in essence, the Jordan normal form of an element in  $\mathcal{H}^P_n$ is again in $\mathcal{H}^P_n$.
\end{lem}

\begin{proof}
Let $Z$ be an element of $\mathcal{H}^P_n$. We will show that the Jordan blocks of $Z$ corresponding to an eigenvalue $\la \in \C \setminus \R$ are exactly the same as those corresponding to $\overline{\la}$ (albeit complex conjugate), whereas those corresponding to an eigenvalue $\mu \in \R$ come in pairs. By permuting the Jordan blocks we can then arrange for $Z$ to be conjugate to an element of the form \eqref{normaal}. \\
\noindent To this end, let $\la \in \C \setminus \R$ be a complex eigenvalue of $Z$. Without loss of generality, we may assume that the dimension of the generalized eigenspace of $\la$ is at least that of its complex conjugate $\overline{\la}$. Let $\{e_1, \dots e_m\} \subset \mat(\C, 2n)$ be a set of linearly independent vectors spanning the generalized eigenspace of $\la$. Assume furthermore that $Ze_1 = \la e_1$ and $Ze_i = \la e_i + s_i e_{i-1}$ for $i \not= 1$ and with $s_i \in \{0,1\}$. In other words, $\{e_1, \dots e_m\}$ put $Z$, restricted to the generalized eigenspace of $\la$, in its Jordan normal form. We then have that 
\begin{align}
Z(S\overline{e_1}) = S\overline{Z}\overline{e_1} = S\overline{\la}\overline{e_1} = \overline{\la}(S\overline{e_1}) \, ,
\end{align}
and likewise
\begin{align}
Z(S\overline{e_i}) = S\overline{Z}\overline{e_i} = S(\overline{\la}\overline{e_i} + s_i \overline{e_{i-1}}) = \overline{\la}(S\overline{e_i})+  s_i (S\overline{e_{i-1}})   
\end{align}
for all other $i$. Hence, if we can prove that the set $\{S \overline{e_1}, \dots S\overline{e_m}\}$ is a basis for the generalized eigenspace of $\overline{\la}$ then the Jordan blocks do indeed agree. As we assumed that the dimension of the generalized eigenspace of $\la$ is at least that of $\overline{\la}$, it suffices to check linear independence of $\{S \overline{e_1}, \dots S\overline{e_m}\}$. Therefore, write
\begin{align}
\sum_{i=1}^m a_iS\overline{e_i} = 0 \,
\end{align}
with $a_i \in \C$. Applying $S$ and taking the complex conjugate yields

\begin{align}
\sum_{i=1}^m -\overline{a_i}e_i = 0 \, .
\end{align}
As the $e_i$ are linearly independent, we see that $-\overline{a_i} = a_i = 0$ for all $i$. Hence, the $S \overline{e_i}$ are linearly independent as well.\\
\noindent Next, let $\mu \in \R$ be a real eigenvalue of $Z$. As the statement of the lemma holds for $Z$ if and only if it holds for $Z - \mu \Id_{2n}$, we may assume that $\mu = 0$. We will need that the kernel of $Z$ is always even dimensional. To show this, let $e_1$ be a non-zero element of the kernel of $Z$. It follows that $S\overline{e_1}$ in also in the kernel of $Z$. Furthermore, $S\overline{e_1}$ and $e_1$ are linearly independent, as 
\begin{align}\label{sma1}
a{e_1} + b(S\overline{e_1}) = 0 \, , 
\end{align}
for $a,b \in \C$ implies 
\begin{align}
S(a{e_1} + bS\overline{e_1})= aS{e_1} - b\overline{e_1}=0 \,  ,
\end{align}
and so
\begin{align}\label{sma2}
\overline{a}(S\overline{e_1}) - \overline{b}e_1= 0 \,  .
\end{align}
Note that $\overline{S} = S$. Combining expressions \eqref{sma1} and \eqref{sma2}, we find

\begin{align}
0 = \overline{a}[a{e_1} + b(S\overline{e_1})] -  b[\overline{a}(S\overline{e_1}) - \overline{b}e_1] = (|a|^2 + |b|^2)e_1\, .
\end{align}
Hence, we have that $a = b = 0$. Next, assume that $W:= \spn_{\C}(e_1, S\overline{e_1}, \dots e_m, S\overline{e_m})$ is a $2m$ dimensional subspace of the kernel of $Z$. Suppose $f$ is a nonzero element that is in the kernel, but not in $W$. Then $S\overline{f}$ is also in the kernel. If furthermore we have 
\begin{align}
w + a f + b(S\overline{f}) = 0 \, , 
\end{align}
for some $w \in W$ and $a,b \in \C$, then we get

\begin{align}
0 &= \overline{a}[w + a f + b(S\overline{f})] - b[S\overline{w} + \overline{a} S \overline{f} - \overline{b} f]\\
&=(\overline{a}w - bS\overline{w}) + (|a|^2 + |b|^2)f \, .
\end{align}
Because $S\overline{w} \in W$, we see that $\overline{a}w - bS\overline{w} \in W$. Hence, it holds that $a = b = 0$ and $w =0$. This proves that the kernel has to be even dimensional. \\ \noindent Finally, let $Q_m$ denote the number of times $B_m(0)$ appears in the Jordan normal form of $Z$. Let $l$ denote the highest number such that $Q_l$ is odd. If $l=1$ then we see from
\begin{align}
\dim (\ker Z) = Q_1 + \dots Q_{2n} \, ,
\end{align}
that $\dim (\ker Z)$ has to be odd, contradicting our previous result. Likewise, $l=2n$ leads to the contradiction $\dim (\ker Z) =1$. For $1 < l < 2n$ we see that
\begin{align} \label{gen01}
\dim (\ker Z^{l-1}) = \sum_{i=1}^{l-1} i Q_i + \sum_{i=l}^{2n} (l-1) Q_i \, ,
\end{align}
and 
\begin{align}\label{gen02}
\dim (\ker Z^{l}) = \sum_{i=1}^{l} i Q_i + \sum_{i=l+1}^{2n} {l} Q_i \, .
\end{align}
Here we have used that $B_m(0)^k = 0$ for $k \geq m$ and that $\dim (\ker B_m(0)^k) = k$ for $1 \leq k \leq m$. Subtracting expression \eqref{gen01} from expression \eqref{gen02} and interpreting the numbers modulo $2$, we find 
\begin{align}
0 = (l-(l-1)) Q_l + \sum_{i=l+1}^{2n} (l-(l-1)) Q_i = Q_l \, ,
\end{align}
contradicting that $Q_l$ is odd. We conclude that a largest $l$ such that $Q_l$ is odd does not exist. As $Q_k = 0$ for all $k > 2n$ due to the size of $Z$, we see that all $B_k(0)$ appear an even number of times. This proves the lemma.
\end{proof}

\noindent Combining Lemmas \ref{normalfo} and \ref{reeal}, we see that $Z \in \mathcal{H}^P_n$ has a vanishing (respectively purely imaginary) spectrum, if and only if there exists an invertible $C \in \mathcal{H}^P_n$ such that $CZC^{-1}$ if of the form $\eqref{normaal}$ with $N$ in Jordan normal form and with a vanishing (respectively purely imaginary) spectrum.

\begin{proof}[Proof of Theorem \ref{Hnil}]
Define $\tilde{B}_{n,p}(0)$ for $p \in \mathcal{P}(n)$ to be the matrix

\begin{equation}
\tilde{B}_{n,p}(0) := \begin{pmatrix}
 B_{n,p}(0) & 0 \\
0 & B_{n,p}(0)
 \end{pmatrix} \in \mathcal{H}^P_n \, .
 \end{equation}
As before, the smooth map 
 
\begin{align}
\Psi_p: \mathcal{H}^P_n \cap \Gl(2n,\C)&\rightarrow \mathcal{H}^P_n \\ \nonumber
Z &\mapsto Z \tilde{B}_{n,p}(0) Z^{-1} \nonumber
\end{align}
has constant rank equal to the dimension of $\im ( \ad_{\tilde{B}_{n,p}(0)}\mid_{\mathcal{H}^P_n})$. Note that $Z^{-1} \in \mathcal{H}^P_n$ whenever $Z \in \mathcal{H}^P_n$ is invertible. This follows from the Cayley-Hamilton theorem, or from the fact that $\mathcal{H}^P_n$ is the image of an algebra of equivariant maps under a morphism of algebras. It remains to determine the dimension of $\im ( \ad_{\tilde{B}_{n,p}(0)}\mid_{\mathcal{H}^P_n})$, and to prove that the image of $\Psi_p$ is indeed an embedded manifold.  We denote this image by $\mathcal{S}_{\tilde{B}_{n,p}(0)}$. To determine the dimension, let us denote an element $Z \in \mathcal{H}^P_n$ given by

\begin{align}
Z = \begin{pmatrix}
  X&Y\\
  -\overline{Y} & \overline{X}
 \end{pmatrix}
\end{align}
 as $Z = [X | Y]$. We see that in this notation, \\ \noindent $ \ad_{\tilde{B}_{n,p}(0)}(Z) = [\ad_{B_{n,p}(0)}(X)| \ad_{B_{n,p}(0)}(Y)]$. Hence, as $X$ and $Y$ may be chosen freely, we see that the complex dimension of the image of $\ad_{\tilde{B}_{n,p}(0)}$ equals at most $n^2 - n + n^2 - n = 2n^2 - 2n$. Hence the real dimension is at most $4n^2 - 4n$. Equality is furthermore only attained when $p = (n)$. \\
Because $\Psi_p$ is a smooth map of constant rank, there exists an open set $S \subset \mathcal{H}^P_n$ containing $\Id$ such that $\Psi_p(S)$ is an embedded submanifold of $\mathcal{H}^P_n$ containing $\tilde{B}_{n,p}(0)$. Therefore, let $(X_r)_{r=0}^{\infty}$ be a sequence of elements in $\mathcal{S}_{\tilde{B}_{n,p}(0)} \subset \mathcal{H}^P_n$ that has $\tilde{B}_{n,p}(0)$ as its limit. We need to show that $X_r$ lies in  $\Psi_p(S)$ for large enough values of $r$. This would prove that $\mathcal{S}_{\tilde{B}_{n,p}(0)}$ is locally around $\tilde{B}_{n,p}(0)$ an embedded submanifold, analogous to the proof of Theorem \ref{Rnil}. As we have the inclusions $\mathcal{H}^P_n \subset \mat(\C,2n)$ and $\mathcal{S}_{\tilde{B}_{n,p}(0)} \subset \mathcal{O}_{\tilde{B}_{n,p}(0)}$, we may use Lemma \ref{neigh} to write 

\begin{align}
X_r = \exp(-v_r){\tilde{B}_{n,p}(0)}\Exp(v_r)
\end{align}
for sufficiently large $r$, and for certain complex matrices $v_r$ that limit $0$. As the matrices $\exp(-v_r)$ limit $\Id_{2n}$, we see by Lemma \ref{reeal} that there exist matrices $C_r \in \mathcal{H}^P_n \cap \Gl(2n,\C)$ such that 
\begin{align}
X_r = C_r{\tilde{B}_{n,p}(0)}C_r^{-1} \, .
\end{align}
 It follows from Lemma \ref{reeal} that the $C_r$ may furthermore be chosen such that their limit is $\Id_{2n}$ as well. In particular, we see that $C_r \in S$ for sufficiently large $r$, proving that $X_r \in \Psi_p(S)$ for sufficiently large $r$. This shows that $\mathcal{S}_{\tilde{B}_{n,p}(0)}$ is locally an embedded submanifold. Hence, by homogeneity, it is so globally. This proves the theorem.

\end{proof}
\noindent In order to prove Theorem \ref{Hcen}, we will again introduce some notation. Given an integer $1 \leq m \leq n$ and an elements $\xi = (p;p_1, \dots p_l) \in \Xi_m$, we (re)introduce the open set 

\begin{align}
W_{\xi} = \{x \in \R^l \mid x_i > 0, x_i \not= x_j  \text{ for all } i, j \in \{1, \dots l\} \text{ such that } i \not= j\}\, .
\end{align}
Next, given $\xi \in \Xi_m$, $x \in W_{\xi}$ and $q \in \mathcal{P}(n-m)$, we define the $n$ times $n$ matrix

\begin{align} \label{quatbase}
H_{\xi,q}(x) := \begin{pmatrix}
  B_{\xi}(x)&0\\
  0 & B_{n-m,q}(0)
 \end{pmatrix} \, ,
\end{align}
and the matrix

\begin{align}
\tilde{H}_{\xi,q}(x) := \begin{pmatrix}
  H_{\xi,q}(x)&0\\
  0 & \overline{H_{\xi,q}(x)}
 \end{pmatrix} = 
 \begin{pmatrix}
  H_{\xi,q}(x)&0\\
  0 & H_{\xi,q}(-x)
 \end{pmatrix} \in \mathcal{H}^P_n \, .
\end{align} 
In the notation of the proof of Lemma \ref{Hnil} we have $\tilde{H}_{\xi,q}(x) = [H_{\xi,q}(x)| 0]$. Lastly, we introduce the matrices

\begin{align} \label{quatbase}
I'_{\xi}(z) := \begin{pmatrix}
  I_{p}(z)&0\\
  0 & 0
 \end{pmatrix} \in \mat(\C,n)\, ,
\end{align}
for $z \in \C^l$, and $\tilde{I}_{\xi}(z) := [I'_{\xi}(z)|0] \in \mathcal{H}^P_n$.

\noindent The following lemma will be used to count the dimensions of the manifolds of Theorem \ref{Hnil}.

\begin{lem}\label{counth}
The dimension of the image of $\mathcal{L}_{H_{\xi,q}(x), \overline{H_{\xi,q}(x)}}$ is independent of $x \in W_{\xi}$. Furthermore, the operator $\mathcal{L}_{H_{\xi,q}(x), \overline{H_{\xi,q}(x)}}$ is surjective if and only if $m = n$. That is, if and only if $H_{\xi,q}(x)$ has only non-zero eigenvalues. 
\end{lem} 

\begin{proof}
Let us denote by $X_{i,j}$, $i,j \in \{1,2\}$, a block of a matrix $X \in \mat(\C,n)$ corresponding to the block structure of the matrix \eqref{quatbase}. For $i=1$ and $j=2$ we see that

\begin{align}
(\mathcal{L}_{H_{\xi,q}(x), \overline{H_{\xi,q}(x)}}(X))_{1,2} = \mathcal{L}_{B_{n-m,q}(0), \overline{B_{\xi}(x)}}(X_{1,2}) \, .
\end{align}
As $B_{n-m,q}(0)$ and $\overline{B_{\xi}(x)}$ have no eigenvalues in common, Lemma \ref{cominv} tells us that the operator $\mathcal{L}_{B_{n-m,q}(0), \overline{B_{\xi}(x)}}$ is a bijection. Similarly the map, 

\[ X_{2,1} \mapsto \mathcal{L}_{B_{\xi}(x), B_{n-m,q}(0)}(X_{2,1})\]
is a bijection, corresponding to the other off-diagonal block. For $i=j=1$ we see that 

\begin{align}
(\mathcal{L}_{H_{\xi,q}(x), \overline{H_{\xi,q}(x)}}(X))_{1,1} = \mathcal{L}_{B_{\xi}(x), \overline{B_{\xi}(x)}}(X_{1,1}) \, .
\end{align}
As $B_{\xi}(x)$ and $\overline{B_{\xi}(x)} = B_{\xi}(-x)$ have no eigenvalues in common, the map $\mathcal{L}_{B_{\xi}(x), \overline{B_{\xi}(x)}}$ is again a bijection. Lastly, we have that 

\begin{align}
(\mathcal{L}_{H_{\xi,q}(x), \overline{H_{\xi,q}(x)}}(X))_{2,2} = \mathcal{L}_{B_{n-m,q}(0), B_{n-m,q}(0)}(X_{2,2}) \, .
\end{align}
By Lemma \ref{cominv}, this map is never a bijection (when $n-m > 0$). It is, however, independent of $x \in W_{\xi}$. This proves that (the dimension of) the image of $\mathcal{L}_{H_{\xi,q}(x), \overline{H_{\xi,q}(x)}}$ is independent of $x \in W_{\xi}$. It also proves that this operator is a bijection if and only if $n = m$. This concludes the proof. 
\end{proof}

\begin{proof}[Proof of Theorem \ref{Hcen}]
The proof will be analogous to that of Theorem \ref{Rcen}. First, we define a smooth map for every pair $(\xi, q) \in \Xi_m \times \mathcal{P}(n-m)$ and show that this map is of constant rank. Then we show that the images of these maps are embedded manifolds, by comparing to the result of Theorem \ref{Ccen}.

\noindent For $m \in \{1, \dots n\}$, $\xi \in \Xi_m$ and $q \in \mathcal{P}(n-m)$ we define the smooth map
\begin{align}
\tilde{\Psi}_{\xi,q}: \, &\mathcal{H}^P_n \cap Gl(2n,\C) \times W_{\xi} \rightarrow \mathcal{H}^P_n\\ \nonumber
&(A,x) \mapsto A \tilde{H}_{\xi,q}(x) A^{-1} \, .\nonumber
\end{align}
As was the case for $\mathcal{C}^P_n$ and $\mathcal{R}^P_n$, some of the sets 

\begin{align}
\mathcal{S}_{\xi, q} := \{\tilde{\Psi}_{\xi,q}(A,x) \mid (A,x) \in  \mathcal{H}^P_n \cap Gl(2n,\C) \times W_{\xi} \}
\end{align}
may coincide for different values of $(\xi,q)$. However, after discarding doubles they will be disjoint. It follows from Lemmas \ref{reeal} and \ref{normalfo} that any element of\ $\mathcal{H}^P_n$ with a purely imaginary spectrum is either nilpotent or contained in one of these sets. 
\noindent Similar to the proof of Theorem \ref{Rcen}, the image of the derivative of $\tilde{\Psi}_{\xi,q}$ at a point $(A,x) \in \mathcal{H}^P_n \cap Gl(2n,\C) \times W_{\xi}$ is given by 

\begin{align}
&\im (T_{(A,x)}\tilde{\Psi}_{\xi,q}) \nonumber \\
&= \{ A[A^{-1}V, \tilde{H}_{\xi,q}(x)]A^{-1} + A\tilde{I}_{\xi}(iw) A^{-1}  \mid (V,w) \in \mathcal{H}^P_n \times \R^l\} \nonumber \\
&= \{ A[V, \tilde{H}_{\xi,q}(x)]A^{-1} + A\tilde{I}_{\xi}(iw) A^{-1}  \mid (V,w) \in \mathcal{H}^P_n \times \R^l\} \, .
\end{align}
The dimension of this space is equal to that of

\[ \{ [V, \tilde{H}_{\xi,q}(x)] + \tilde{I}_{\xi}(iw)   \mid (V,w) \in \mathcal{H}^P_n \times \R^l\} \, .\]
Now, $\tilde{H}_{\xi,q}(x)$ is a block-diagonal matrix. Hence, all the diagonal blocks of an element of the form $[V, \tilde{H}_{\xi,q}(x)]$ have vanishing trace. Therefore, the only element both of the form $[V, \tilde{H}_{\xi,q}(x)]$ for $V \in  \mathcal{H}^P_n$ and of the form $\tilde{I}_{\xi}(iw)$ for $w \in \R^l$ is $0$. Compare to the proof of Lemma \ref{biggest}. We conclude that 

\begin{align}
&\{ [V, \tilde{H}_{\xi,q}(x)] + \tilde{I}_{\xi}(iw)   \mid (V,w) \in \mathcal{H}^P_n \times \R^l\}  \nonumber \\
= &\{ [V, \tilde{H}_{\xi,q}(x)] \mid V \in \mathcal{H}^P_n \} \oplus \{\tilde{I}_{\xi}(iw)   \mid w \in \R^l\}  \, .
\end{align}
In order to show that $\tilde{\Psi}_{\xi,q}$ is a map of constant rank, it remains to show that the dimension of $\{ [V, \tilde{H}_{\xi,q}(x)] \mid V \in \mathcal{H}^P_n \} $ is independent of the choice of $x \in W_{\xi}$. To this end, we write $V = [V_1 | V_2]$ in the notation of the proof of Theorem \ref{Hnil}. Writing out the commutator, we see that

\begin{align}
[V, \tilde{H}_{\xi,q}(x)] = &[\mathcal{L}_{ {H}_{\xi,q}(x), {H}_{\xi,q}(x) }(V_1)\mid \mathcal{L}_{\overline{{H}_{\xi,q}(x)}, {H}_{\xi,q}(x)}(V_2)] \\ \nonumber
= &[- \ad_{ {H}_{\xi,q}(x)}(V_1)\mid \mathcal{L}_{\overline{{H}_{\xi,q}(x)}, {H}_{\xi,q}(x)}(V_2)] \, .
\end{align}
From a similar reasoning as that in the proof of Lemma \ref{biggest}, we see that the image of $\ad_{ {H}_{\xi,q}(x)}$ is independent of $x \in W_{\xi}$. Note in particular that there exists a permutation matrix $Q$ such that $Q{H}_{\xi,q}(x)Q^{-1} = B_{\xi'}(x')$ for some $\xi' \in \Xi_n$ and with $x' \in V_{\xi'}$. This permutation matrix just reorders the blocks to adhere to our convention to have partitions ordered, and does not depend on $x$. The entries of $x'$ are just those of $x$ with an added $0$ if $n \not= m$. From this we see that the complex dimension of the image of $\ad_{ {H}_{\xi,q}(x)}$ is at most $n^2 - n$, with equality only when the partitions $q$ and $p_1$ till $p_l$ in $\xi = (p; p_1, \dots p_l)$ are all trivial. From Lemma \ref{counth} we see that the image of $\mathcal{L}_{\overline{{H}_{\xi,q}(x)}, {H}_{\xi,q}(x)}$ is likewise independent of $x$. This shows that $\tilde{\Psi}_{\xi,q}$ is indeed a smooth map of constant rank. \\
\noindent It also follows that the real dimension of the image of $\tilde{\Psi}_{\xi,q}$ is equal to that of \\
 \noindent $\dim_{\R} \im (\ad_{\tilde{H}_{\xi,q}(x)}|_{\mathcal{H}^P_n}) + l$. By our bound on the dimension of the image of $\ad_{ {H}_{\xi,q}(x)}$, combined with Lemma \ref{counth} we see that this value cannot exceed $2(n^2 -n) + 2n^2 + n = 4n^2 - n$. For equality we need $l=n$. This value of $l$ forces $m$ to be equal to $n$ as well, and forces $\xi$ to equal $((1, \dots 1);(1) \dots (1))$. From our remark about the image of $\ad_{ {H}_{\xi,q}(x)}$ and from the result of Lemma \ref{counth} we see that the real dimension of $\mathcal{S}_{\xi, q}$ is indeed equal to $4n^2 - n$ in the unique case when $l = m = n$.  \\
\noindent To summarize so far, we have found that the sets $\mathcal{S}_{\xi, q}$ are all immersed manifolds of real dimension $4n^2 - n$ or lower. The exact value of $4n^2 - n$ only occurs when all eigenvalues of $\tilde{H}_{\xi,q}(x)$ are different (and hence unequal to $0$). It remains to show that the sets $\mathcal{S}_{\xi, q}$ are in fact embedded submanifolds.\\
By the constant rank theorem, we know that for any $x \in W_{\xi}$ there exists an open set $S\times T \subset \mathcal{H}^P_n \cap Gl(2n,\C) \times W_{\xi}$ containing $(\Id,x)$ such that $\tilde{\Psi}_{\xi,q}(S\times T)$ is a submanifold of $\mathcal{H}^P_n$ containing $\tilde{H}_{\xi,q}(x)$. It remains to show that other elements of $\mathcal{S}_{\xi, q}$ do not come arbitrarily close to $\tilde{H}_{\xi,q}(x)$. To this end, assume the converse, so that $(X_r)_{r=0}^{\infty}$ is a sequence in $\mathcal{S}_{\xi, q} \setminus \tilde{\Psi}_{\xi,q}(S\times T)$ converging to $\tilde{H}_{\xi,q}(x)$. Let $P$ be a permutation matrix such that $P\tilde{H}_{\xi,q}(x)P^{-1}= B_{\xi'}(x')$ for some $\xi' \in \Xi_{2n}$ and $x' \in V_{\xi'}$ (consisting of entries in $x$, minus those and possibly $0$). Note that $P$ depends only on $\xi$ and $q$. From their Jordan normal forms, we see that $\tilde{H}_{\xi,q}(y) \in \mathcal{O}_{\xi'} $ for all $y \in W_{\xi}$. By conjugacy invariance of $\mathcal{O}_{\xi'}$ we conclude that $\mathcal{S}_{\xi, q} \subset \mathcal{O}_{\xi'} \subset \mat(\C, 2n)$. Therefore, $(Y_r)_{r=0}^{\infty} := (PX_rP^{-1})_{r=0}^{\infty}$ is a sequence in $\mathcal{O}_{\xi'}$ converging to $B_{\xi'}(x')$. By the conclusion at the end of Theorem \ref{Ccen}, we see that we may write

\begin{align}
Y_r &= \exp(-v_r)(B_{\xi'}(x') + I_{\xi'}(z^r))\exp(v_r) \\ \, \nonumber
\end{align}
for large enough values of $r$. Here, the $v_r$ are complex matrices satisfying $\displaystyle \lim_{r \rightarrow \infty} v_r = 0$, so that $\displaystyle \lim_{r \rightarrow \infty} \exp(-v_r) = \Id$. We furthermore have $z^r \in (i\R)^k$ for $k := \dim(V_{\xi})$, satisfying $\displaystyle \lim_{r \rightarrow \infty} z^r = 0$.  \\
\noindent As in the proof of Theorem \ref{Rcen}, we want to conclude that $P^{-1}(B_{\xi'}(x') + I_{\xi'}(z^r))P$ is an element of the form $\tilde{H}_{\xi,q}(x^r) \in \mathcal{H}^P_n$ for large enough values of $r$. As the matrices $P^{-1}(B_{\xi'}(x') + I_{\xi'}(z^r))P$ are conjugate to $X_r \in \mathcal{H}^P_n$, we will then conclude from Lemma \ref{reeal} that this conjugation can be done by elements in $\mathcal{H}^P_n$. This will then lead to a contradiction, as it will force the $X_r$ to lie in $\tilde{\Psi}_{\xi,q}(S\times T)$ for large enough values of $r$. \\
\noindent To show that the matrices $P^{-1}(B_{\xi'}(x') + I_{\xi'}(z^r))P$ are of the form $\tilde{H}_{\xi,q}(x^r)$ for certain values of $x^r \in W_{\xi}$, we need to show that the eigenvalues in the blocks of $B_{\xi'}(x') + I_{\xi'}(z^r)$ come in prescribed pairs with opposite signs. In particular, if $m \not= n$ then a prescribed block has to be nilpotent. To this end, we reintroduce the real-property from the proof of Theorem \ref{Rcen}. An element $v \in \C^k$ has this property if for some function $\tau: \{1, \dots k\} \rightarrow \{1, \dots k\}$ we have $v_\{\tau(j)\} = -v_j$. As the matrices $B_{\xi'}(x')$ and $B_{\xi'}(x') + I_{\xi'}(z^r)$ are conjugate to elements in $\mathcal{S}_{\xi, q} \subset \mathcal{H}^P_n$, we see that if $\la \in i\R$ occurs as an eigenvalue, then so does $-\la = \overline{\la}$. Therefore, these matrices have the real-property. Exactly as the proof of Lemma \ref{Rcen}, $B_{\xi'}(x')$ has this property for exactly one function $\tau_0$. (Note that $P$ has been chosen such that the different blocks of $B_{\xi'}(x')$ have different eigenvalues, respecting the notation.) It follows that for large enough values of $r$, the eigenvalues of $B_{\xi'}(x') + I_{\xi'}(z^r)$ satisfy the real-property for $\tau_0$ as well. This shows that the eigenvalues in the blocks of $B_{\xi'}(x') + I_{\xi'}(z^r)$ are arranged so that we may write $P^{-1}(B_{\xi'}(x') + I_{\xi'}(z^r))P = \tilde{H}_{\xi,q}(x^r)$ for certain $x^r \in W_{\xi}$. Note that $\displaystyle \lim_{r \rightarrow \infty} x^r = x$. Returning to the $X_r$, we see that for $r$ large enough we have

\begin{align}
X_r &= P^{-1}\exp(-v_r)(B_{\xi'}(x') + I_{\xi'}(z^r))\exp(v_r)P \\  \nonumber
&= P^{-1}\exp(-v_r)PP^{-1}(B_{\xi'}(x') + I_{\xi'}(z^r))PP^{-1}\exp(v_r)P\\ \nonumber
&= (P^{-1}\exp(-v_r)P) \tilde{H}_{\xi,q}(x^r)(P^{-1}\exp(-v_r)P)^{-1} \, .
\end{align}
As $X_r$ and $\tilde{H}_{\xi,q}(x^r)$ are both elements of $\mathcal{H}^P_n$, we conclude from Lemma \ref{reeal} that there exist invertible matrices $C_r \in \mathcal{H}^P_n$ such that

\begin{align}
X_r &= C_r \tilde{H}_{\xi,q}(x^r)C_r^{-1} \\ \nonumber
&= \tilde{\Psi}_{\xi,q}(C_r, x^r) \, .
\end{align}
As $\displaystyle \lim_{r \rightarrow \infty} \exp(-v_r) = \Id$, it also holds that $\displaystyle \lim_{r \rightarrow \infty}P^{-1} \exp(-v_r) P =\Id$. Therefore, we see that the $C_r$ can be chosen such that $\displaystyle \lim_{r \rightarrow \infty} C_r = \Id$. Because it also holds that $\displaystyle \lim_{r \rightarrow \infty} x^r = x$, we see that for large enough values of $r$, we have that $X_r \in \tilde{\Psi}_{\xi,q}(S\times T)$. This directly contradicts our assumptions, and we conclude that $\mathcal{S}_{\xi, q}$ is locally around $\tilde{H}_{\xi,q}(x)$ an embedded submanifold. By homogeneity, $\mathcal{S}_{\xi, q}$ is globally an embedded submanifold. This proves the theorem.

\end{proof}

\noindent It follows from the proof of Theorem \ref{Hcen} that the unique manifold of highest dimension consists again of exactly those matrices with no double eigenvalues.

%\begin{thr}\label{main2}
%The map $\Psi^Q : \End(W) \rightarrow \mathcal{K}^Q_W$ is a surjective morphism of real unitary algebras with kernel $\mathcal{J}$. Moreover, a value $\la \in \C$ is an eigenvalue of $A \in \End(W)$ if and only if it is an eigenvalue of one of the components of $\Psi^Q(A)$.
%\end{thr}
%
%\[r_1 + \dots + r_u + 2c_1 + \dots  + 2c_v + 4h_1 + \dots + 4h_w\] 
%or higher. Exactly one of these manifolds has codimension precisely equal to this number. \\
%\noindent Furthermore, these manifolds can be chosen to be conjugacy invariant. That is, if $M$ denotes any of these manifolds and if $A$ is an element of $M$ and $C$ in  $\End(W)$ is invertible, then  $CAC^{-1}$ is an element of $M$ as well. \\
%Likewise, the set $\cen(W)$ is the disjoint union of a finite set of conjugacy invariant manifolds having codimension
%\[ \ceil{r_1/2} + \dots + \ceil{r_u/2} + c_1 + \dots  + c_v + h_1 + \dots + h_w\] 

%\begin{remk}\label{dontmiss}
%From Remarks \ref{Cdis},  \ref{Rdis} and  \ref{Hdis} we know that the unique manifold of highest dimension, corresponding to either the elements with a vanishing spectrum or those with a purely imaginary spectrum, are contained in an open set that is disjoint from the other manifolds. Hence, this property holds for the manifolds of highest dimension in $\End(W)$ as well. This property will guarantee that there is a set of parameter families of linear maps missing all but these particular manifolds, given that the dimension of the parameter space is high enough.  \hspace*{\fill}$\triangle$
%\end{remk}

\bibliography{lin_trans}
\bibliographystyle{plain}

\end{document}